\documentclass{article}

\usepackage{arxiv}

\RequirePackage{amsthm,amsmath,amsfonts,amssymb}
\RequirePackage[authoryear]{natbib}
\usepackage{graphicx}
\usepackage{enumerate}
\usepackage{natbib}
\usepackage{url} 
\RequirePackage[colorlinks,linkcolor=blue,citecolor=blue,urlcolor=black]{hyperref}
\usepackage{latexsym,amsopn,amscd,
	amsbsy,marginnote,cprotect}
\usepackage{mathrsfs}
\usepackage{subfig}
\usepackage{comment}
\usepackage[usenames,dvipsnames]{color}
\usepackage[normalem]{ulem}
\usepackage{float}
\usepackage{epsfig}
\usepackage{relsize}
\usepackage[dvipsnames]{xcolor}
\usepackage{bm}    
\usepackage{bbm}                   
\usepackage[english]{inputenc}  
\usepackage{soul}

\usepackage[T1]{fontenc}    % use 8-bit T1 fonts
\usepackage{url}            % simple URL typesetting
\usepackage{nicefrac}       % compact symbols for 1/2, etc.
\usepackage{microtype}      % microtypography
\usepackage{doi}

\usepackage[title]{appendix}

\newcommand{\N}{\ensuremath{\mathbb{N}}}
\newcommand{\R}{\ensuremath{\mathbb{R}}}

\newcommand{\Prob}{\ensuremath{\mathbb{P}}}
\newcommand{\Esp}{\ensuremath{\mathbb{E}}}

\newcommand{\B}{\mathcal{B}}

\newcommand{\daw}{\downarrow}
\newcommand{\Ind}{\mathlarger{\mathbbm{1}}}
\newcommand{\captionstring}[1]{\noexpand\noexpand\noexpand\string\string#1}

\newcommand{\cl}[1]{\mathcal{#1}}

\newcommand{\msf}[1]{\mathsf{#1}}
\newcommand{\mbf}[1]{\mathbf{#1}}

\newcommand{\mrm}[1]{\mathrm{#1}}
\newcommand{\comillas}[1]{``\,#1\,"}

\def\mi{\,\middle|\,}
\def\l{\left}
\def\r{\right}

\def\deq{\stackrel{d}{=}}
\def\aseq{\stackrel{a.s.}{=}}
\def\dto{\stackrel{d}{\to}}

\def\wto{\stackrel{w}{\to}}
\def\dwto{\stackrel{wd}{\longrightarrow}}
\def\L2to{\stackrel{\cl{L}_2}{\to}}

\def\iid{\stackrel{\mbox{\scriptsize{iid}}}{\sim}}
\def\ind{\stackrel{\mbox{\scriptsize{ind}}}{\sim}}

\def\Y{{\mathbb{Y}}}

\def\X{{\mathbb{X}}}

\def\D{{\mbf{D}}}

\def\bpsi{{\psi}}
\def\bmpsi{{\bm{\psi}}}

\def\Be{{\msf{Be}}}
\def\Un{{\msf{Unif}}}

\def\Bin{{\msf{Bin}}}

\def\Var{{\msf{Var}}}
\def\Cov{{\msf{Cov}}}

\def\oDk{{{\Delta^{\mathrm{o}}_k}}}
\def\Dk{{{\Delta}_k}}
\def\oD2{{{\Delta^{\mathrm{o}}_2}}}
\def\D2{{{\Delta}_2}}
\def\2F1{{{}_2{F}_1}}

\theoremstyle{plain}
\newtheorem{theo}{Theorem}[section]
\newtheorem{prop}[theo]{Proposition}
\newtheorem{lem}[theo]{Lemma}
\newtheorem{cor}[theo]{Corollary}

\theoremstyle{remark}
\newtheorem{rem}[theo]{Remark}

\title{Markov stick-breaking processes}

\author{ \href{https://orcid.org/0000-0003-0811-5083}{\includegraphics[scale=0.06]{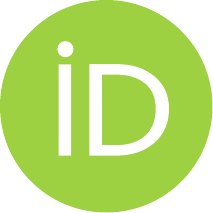}\hspace{1mm}Mar\'ia F.~Gil--Leyva}\\
	Department of Probability and Statistics\\
	IIMAS, UNAM\\
	\texttt{marifer@sigma.iimas.unam.mx} \\
	\And
	\href{https://orcid.org/0000-0001-6159-2650}{\includegraphics[scale=0.06]{orcid.pdf}\hspace{1mm}Antonio~Lijoi} \\
	Bocconi Institute for Data Science and Analytics\\
	Bocconi University\\
	\texttt{antonio.lijoi@unibocconi.it} \\
        \And
	\href{https://orcid.org/0000-0001-9608-8059}{\includegraphics[scale=0.06]{orcid.pdf}\hspace{1mm}Rams\'es H.~Mena} \\
	Department of Probability and Statistics\\
	IIMAS, UNAM\\
	\texttt{ramses@sigma.iimas.unam.mx} \\
	\And
	\href{https://orcid.org/0000-0003-2860-1476}{\includegraphics[scale=0.06]{orcid.pdf}\hspace{1mm}Igor~Pr\"unster} \\
	Bocconi Institute for Data Science and Analytics\\
	Bocconi University\\
	\texttt{igor@unibocconi.it} \\
}

\date{}

\renewcommand{\headeright}{}
\renewcommand{\undertitle}{}
\hypersetup{
pdftitle={Markov stick-breaking processes},
pdfauthor={Mar\'ia F.~Gil--Leyva, Antonio~Lijoi, Rams\'es H.~Mena, Igor~Pr\"unster},
pdfkeywords={MSB},
}

\begin{document}
\maketitle

\begin{abstract}
Stick-breaking has a long history and is one of the most popular procedures for constructing random discrete distributions in Statistics and Machine Learning. In particular, due to their intuitive construction and computational tractability they are ubiquitous in modern Bayesian nonparametric inference. Most widely used models, such as the Dirichlet and the Pitman-Yor processes, rely on iid or independent length variables. Here we pursue a completely unexplored research direction by considering Markov length variables and investigate the corresponding general class
of stick-breaking processes, which we term \textit{Markov stick-breaking processes}. 
We establish conditions under which the associated species sampling process is proper and the distribution of a Markov stick-breaking process has full topological support, two fundamental \textit{desiderata} for Bayesian nonparametric models. We also analyze the stochastic ordering of the weights and provide a new characterization of the Pitman-Yor process as the only stick-breaking process invariant under size-biased permutations, under mild conditions. Moreover, we identify two notable subclasses of Markov stick-breaking processes that enjoy appealing properties and include Dirichlet, Pitman-Yor and Geometric priors as special cases. Our findings include distributional results enabling posterior inference algorithms and methodological insights.
\end{abstract}


%\begin{keyword}[class=MSC]
%	\kwd{62F15}
%	\kwd{62C10}
%	\kwd{60G57}
%\end{keyword}

%--------------------------------------------------------------------------------------------------------------------------------
\section{Introduction}\label{sec:intro}
%--------------------------------------------------------------------------------------------------------------------------------
\subsection{Species sampling processes}
Species sampling processes \citep{P96,P06} are random probability measures that are defined as mixtures of a discrete and a diffuse component. The former is characterized by the structural assumption of independence between the (random) locations of the atoms and their associated (random) weights. They provide a natural framework for the analysis of discrete random probability measures, enjoy intuitive structural properties and constitute the backbone of Bayesian nonparametric analysis beyond the Dirichlet process. Formally a species sampling process over a complete and separable metric space $\X$ endowed with its Borel $\sigma$-field $\B_\X$, is a random probability measure with decomposition 
	\begin{equation}\label{eq:ssp}
		\bm{P} = \sum_{j=1}^{\infty}w_j\delta_{\theta_j} + \Bigl(1-\sum_{j=1}^{\infty}w_j\Bigr)P_0,
	\end{equation}
where the atoms $\bm{\theta} = (\theta_j)_{j = 1}^{\infty}$ are
iid from a diffuse distribution, $P_0$, over  $(\X,\B_\X)$, called base measure of $\bm{P}$; the weights $\bm{w} = (w_j)_{j=1}^{\infty}$ are non-negative random variables (r.v.) that satisfy $\sum_{j=1}^{\infty} w_j \leq 1$ almost surely (a.s.), and crucially $\bm{\theta}$ is assumed independent of  $\bm{w}$. When $\sum_{j=1}^{\infty}w_j = 1$ a.s., $\bm{P}$ becomes an a.s.~discrete random probability measure, termed \textit{proper}. 
From a Bayesian perspective, interest lies in proper species sampling processes, since including a deterministic diffuse component would be pointless, as it cannot learn from the data.

Concrete instances of discrete random probability measures arise by fully specifying the law of a species sampling process $\bm{P}$: this requires assigning a base measure $P_0$ and the law of the weights $\bm{w}$, which can be a challenging task. There is a rich literature spanning Statistics, Probability and Machine Learning on strategies to specify laws for $\bm{w}$ either directly or indirectly. Indirect characterizations leverage the one-to-one correspondence, ensured by de Finetti's Representation Theorem, between random probability measures and their marginal counterparts given by exchangeable sequences, exchangeable random partitions or predictive distributions. Popular approaches include generalized P\'olya urn constructions \citep[see e.g.][]{BM73,P96,P06,caron07,caron17,ruggiero18} and Gibbs-type partitions and associated prediction rules \citep[see e.g.][]{gnedin06,LMP07b,lpw2008, gnedin, flp13,DeBlasi2015,linero,camerlengo23}. Direct strategies to assign the law of  $\bm{w}$ can be divided into two main categories. The first consists in normalizing the jumps of a subordinator or, more generally, of a completely random measure \citep[see e.g.][]{F73,K75,P97,Regazzini03,P03}. The second approach, which is the focus of the present contribution, relies on the so--called \textit{stick-breaking} procedure.

\subsection{Stick-breaking representations}
Perhaps the most popular direct strategy to assign the law of  $\bm{w}$ consists of decomposing the weights' sequence $\bm{w}$ by means of a stick-breaking procedure as
	\begin{equation}\label{eq:sb}
		w_1 = v_1, \quad w_j = v_j \prod_{i=1}^{j-1}(1-v_i), \quad j \geq 2,
	\end{equation}
where $\bm{v} = (v_j)_{j=1}^{\infty}$ are $[0,1]$-valued r.v. which we refer to as \textit{length variables}. Stick-breaking representations date back to \cite{H44} and were greatly popularized by the stick-breaking representation of the Dirichlet process in \cite{S94}, which arises from \eqref{eq:sb} by taking iid $\mathsf{Be}(1, \theta)$ length variables with $\theta>0$ and $\mathsf{Be}$ denoting the beta distribution. Other key early contributions include \citet{FR63,McCloskey65,PPY92,P96b,IJ01}. One of the main merits of the stick-breaking approach is to translate the problem of defining a distribution on the infinite dimensional simplex into the simpler one of defining the law of a sequence of $[0,1]$-valued r.v. 
Computationally, this advantage has proved decisive, leading to a proliferation 
of papers on complex models based on stick-breaking constructions.
 These include the seminal papers on dependent \citep{MacE99,mac00,MacKG2001} and hierarchical processes \citep{teh06}. We refer to \citep{HHMW10,MQJH,GvdV17,qmjm22} for exhaustive accounts and also point to recent novel perspectives in \cite{soriano,giordano,canale,hu,horiguchi}.

Another appealing feature of the stick--breaking approach lies in the fact that every weights' sequence $\bm{w}$ admits a stick-breaking decomposition \citep[see e.g.][]{P96b}. Hence, any species sampling process \eqref{eq:ssp} can be recovered with suitable stick-breaking weights \eqref{eq:sb}, which is not the case for alternative constructions of $\bm{w}$. Henceforth we refer to a species sampling process \eqref{eq:ssp} with weights \eqref{eq:sb} as a \textit{stick--breaking process}, keeping in mind that they are in  correspondence with the class of species sampling processes.

\subsection{Motivation and main contributions}

Although, in principle, the stick--breaking approach allows one to recover the whole class of species sampling processes, this richness has not yet been fully exploited. To date most efforts have been concentrated on cases of either iid or independent length variables, which represent only a narrow subset.  Although a handful of examples of stick-breaking processes with explicitly dependent length variables  have been derived \citep{FLP12,FLNNPT16,FMW10, GMN20}, the general dependent case remained somehow elusive. The only exception is \cite{GM21}, which considers exchangeable length variables, though this is a limited form of dependence when broader structures are allowed.

Our goal is to define and investigate stick--breaking processes with length variables having Markov dependence structure, in short \textit{Markov stick--breaking processes}. While this is a natural form of dependence, widely studied in many different contexts, in the framework of stick-breaking constructions it represents unexplored territory. Therefore we first focus on general properties of Markov stick--breaking processes, which are pivotal to their use for Bayesian nonparametric inference. In particular, we establish conditions ensuring the fundamental property of properness, namely $\sum_{j=1}^{\infty}w_j = 1$ a.s., and also for them to have full topological support. The latter represents a key requirement for any Bayesian nonparametric model, as stated already in the landmark paper \cite{F73}, and is equivalent to the law of the random probability measure
assigning positive mass to arbitrarily small neighbourhoods of any possible data generating distribution. 
We also 
provide a posterior characterization of the length variables, essential for devising computational schemes for posterior inference.
Furthermore, we recover as limiting cases 
well-known models such as the Dirichlet \citep{F73,S94}, Geometric \citep{FMW10} and Pitman-Yor \citep{PPY92,P97} processes.

\noindent 
The Pitman-Yor process, also known as the two-parameter Poisson-Dirichlet process \citep{P06}, is the most popular generalization of the Dirichlet process and has found countless applications in Statistics and Machine Learning. Part of its popularity is due to its simple stick--breaking representation: indeed, it arises from \eqref{eq:sb} by taking independent length variables distributed as $v_i \sim \mathsf{Be}(1-\sigma, \theta+i \sigma)$ for $\sigma \in (0,1)$, $\theta>-\sigma$ and $i\geq 1$. A fundamental characterization \citep[][Theorem 2]{P96b} essentially states that, within the class of weight sequences $\bm{w}$ representable in terms of independent length variables, invariance under size--biased permutations holds if and only if the length variables correspond to those of a Pitman-Yor process. This result is relevant to deriving marginal properties and computational schemes for stick--breaking processes, as well as to applications in probability and population genetics \citep[see e.g.][]{P96b,ABT,P06,FLP12}. Here we generalize the result to Markov length variables and obtain a highly non-trivial characterization: again, and somehow surprisingly, invariance under size--biased permutations holds essentially only for the Pitman--Yor regime.

\subsection{Organization of the paper}
Section \ref{sec:MSB} defines Markov stick-breaking processes and establishes their fundamental properties. 
Section \ref{sec:examples} deals with the two degenerate extreme cases, independent and completely dependent length variables.
We then identify and study two notable subclasses 
of Markov stick-breaking processes:
Section \ref{sec:BBSB} is devoted to Beta Markov stick-breaking processes, whereas Section \ref{sec:LMSB} deals with lazy Markov stick-breaking processes. 
For both classes we derive the key ingredients for posterior inference and their performance is illustrated in Section \ref{sec:illustration}.
Additional research directions are highlighted in Section \ref{sec:conclusion}.
In the Appendix (App.), Appendix \ref{sec:app:illust} presents the algorithm for mixture models and an additional multivariate simulation study, while Appendix \ref{sec:app:results}--\ref{sec:app:LMSB_sup} provide proofs, technical derivations, and background. Throughout the main text, we refer to these appendices as App.~\ref{sec:app:illust}, App.~\ref{sec:app:results}, etc.

%--------------------------------------------------------------------------------------------------------------------------------
\section{Markov stick-breaking processes}\label{sec:MSB}
%--------------------------------------------------------------------------------------------------------------------------------
We consider species sampling processes $\bm{P}$, as in \eqref{eq:ssp}, whose weights, $\bm{w} = (w_j)_{j =1}^{\infty}$, admit the stick-breaking decomposition \eqref{eq:sb} for a Markov process $\bm{v} = (v_j)_{j =1}^{\infty}$. Hence, for each $j \geq 1$ and $B \in \B_{[0,1]}$,
\[
\Prob[v_{j+1}\in B \mid v_1,\ldots,v_j] = \Prob[v_{j+1} \in B \mid v_j] = \bpsi_j(v_j,B),
\]
with $\bpsi_j:[0,1]\times \B_{[0,1]} \to [0,1]$ the (one step ahead) transition probability kernels. Denote by $\pi_j$ the marginal distribution of $v_j$. The initial distribution of the process, $\pi = \pi_1$, and the collection of transition probability kernels, $\bm{\psi} = (\psi_j)_{j=1}^{\infty}$, completely characterize the law of $\bm{v}$. In particular, the marginal distributions of $\bm{v}$ can be computed recursively through
\begin{equation}\label{eq:vj_marg}
\pi_{j+1}(B) = \int_{[0,1]} \psi_{j}(v,B)\pi_{j}(\mrm{d}v).
\end{equation}
For these length variables, $\bm{w}$ as in \eqref{eq:sb} is termed \emph{Markov stick-breaking weights sequence} (MSBw) with parameters $(\pi,\bm{\psi})$. Accordingly, we call $\bm{P}$ as in \eqref{eq:ssp} a  \emph{Markov stick-breaking process} (MSBp) with parameters $(\pi,\bm{\psi},P_0)$, recalling that the base measure, $P_0$, is the distribution of the atoms of $\bm{P}$. 

\noindent When all probability kernels are identical, i.e. $\psi_j = \psi_1$ for every $j \geq 1$, the Markov process $\bm{v}$ becomes time homogeneous and we write $\psi = \psi_1$. If we further have that the initial distribution, $\pi$, is invariant for $\psi$, i.e. for every $B \in \B_{[0,1]}$
$$
\pi(B) = \int_{[0,1]} \psi(v,B)\pi(\mrm{d}v),
$$
$\bm{v}$ is stationary. This means that $(v_1,\ldots,v_n)$ is equal in distribution to $(v_{j+1},\ldots,v_{j+n})$ for every $j, n \geq 1$ and, in particular, all marginal distributions are identical i.e. $\pi_j = \pi$ for all $j \geq 1$. In this case, we refer to $\bm{w}$ as a stationary Markov stick-breaking weights sequence (sMSBw) with parameters $(\pi,\psi)$, and we call $\bm{P}$ a stationary Markov stick-breaking process (sMSBp) with parameters $(\pi,\psi,P_0)$.

\subsection{Properness, support and convergence}\label{subsec:1}

Since, from an inferential perspective, interest lies in \textit{proper} species sampling processes, the first property one should verify is that the stick-breaking weights sum up to one. From the stick-breaking decomposition it follows that $1-\sum_{j=1}^m w_j = \prod_{j=1}^m(1-v_j)$ for every $m \geq 1$. Thus $\sum_{j=1}^{\infty}w_j =1$ if and only if $\prod_{j=1}^m (1-v_j) \to 0$ as $m \to \infty$. Being that $0 \leq v_j \leq 1$, this in turn is equivalent to 
\begin{equation}\label{eq:sum1_MSB}
\sum_{j=1}^{\infty} v_j = \infty, \text{ a.s.} \quad \text{ or } \quad \Prob\l[\bigcup_{j=1}^{\infty}\{v_j = 1\}\r] = 1.
\end{equation}
For sMSBw, these considerations and the ergodic theorem yield the following result.

\begin{theo}\label{theo:sMSB_1}
Consider a sMSBw, $\bm{w}$, with parameters $(\pi,\psi)$, and let $\bm{v}$ be its underlying Markov process of length variables. 
\begin{enumerate}
\item[\emph{(i)}] If $\Prob\l[\bigcup_{j=1}^{\infty}\{v_j = 1\}\r] = 1$, then $\sum_{j=1}^{\infty}w_j = 1$ a.s.
\item[\emph{(ii)}] If $\pi(\{0\}) = 0$, then $\sum_{j=1}^{\infty}w_j = 1$ a.s.
\item[\emph{(iii)}] If $\pi$ is $\psi$-ergodic, then $\pi \neq \delta_0$ if and only if $\sum_{j=1}^{\infty}w_j = 1$ a.s.
\end{enumerate} 
\end{theo}

To clarify Theorem \ref{theo:sMSB_1} (iii) we recall that for a transition, $\psi$, the set of invariant measures, $\mathcal{I}[\psi]$, is convex, and that $\pi \in \mathcal{I}[\psi]$ is called $\psi$-ergodic if it is extremal in $\mathcal{I}[\psi]$. In particular, if $\pi$ is the only invariant measure, i.e.  $\mathcal{I}[\psi] = \{\pi\}$, then  $\pi$ is $\psi$-ergodic. Further details on invariant and $\psi$-ergodic measures are provided in App.~\ref{sec:app:Markov}.
Theorem \ref{theo:sMSB_1} shows that most stationary Markov length variables define proper stick-breaking processes, as it suffices to require $\Prob[v_1 = 0] = 0$. Furthermore, if $\pi$ is $\psi$-ergodic this condition can be relaxed to $\Prob[v_1 = 0] < 1$. 
As Theorem \ref{theo:sMSB_1} follows from the ergodic theorem, a more general version of the result can be attained for length variables, $\bm{v}$, that are stationary  but not necessarily Markovian. For non-stationary Markov length variables $\bm{v}$, the relaxation of the conditions in \eqref{eq:sum1_MSB} ensuring $\sum_{j=1}^{\infty}w_j = 1$, is difficult to achieve. Nonetheless, as we will show in Sections \ref{sec:examples}--\ref{sec:LMSB}, for many interesting examples the condition $\sum_{j=1}^{\infty}v_j = \infty$ a.s.~can be proven to hold. Still, some interesting facts can be established for general MSBw, as we discuss next.

\begin{prop}\label{prop:MSB_supp_points}
Let $\bm{w}$ be a MSBw with parameters $(\pi,\bm{\psi})$ and let $\bm{v}$ be the underlying Markov process of length variables. If the marginal distributions $\pi_j$ given by \eqref{eq:vj_marg} satisfy $\pi_j(\{0\}) = 0$ for every $j \geq 1$,  then $w_j > 0$ if and only if $j \leq t = \inf\{j \geq 1: v_j = 1\}$ a.s. Furthermore, if $\pi_j(\{1\}) = 0$ for every $j \geq 1$, $\Prob[w_j > 0, \forall \, j \geq 1] = \Prob[t = \infty] = 1$.
\end{prop}

For the special case of stationary $\bm{v}$ we see that whether the initial distribution $\pi$ has atoms on the extreme points of the interval $[0,1]$ or not, determines important properties of the sMSBp. While $\Prob[v_1= 0] = \pi(\{0\}) = 0$ ensures $\bm{P}$ is proper and that the number of support points of $\bm{P} = \sum_{j=1}^{t}w_j\delta_{\theta_j}$ is the stopping time $t$, the condition $\Prob[v_1 = 1] = \pi(\{1\}) = 0$ assures $\bm{P}$ has infinitely many points. We will mainly focus on the case of diffuse marginal distributions $\pi_j$, so $t = \infty$. Nonetheless, it is important to acknowledge that the class of MSBp also contains species sampling processes with a random number of support points.

Another essential requirement from an inferential perspective is for the nonparametric prior to have full topological support. In the context of discrete random probability measures, this corresponds to a property of the random weights, which guarantees that all probability distributions (whose support is contained in the support of the base measure $P_0$) belong to the weak topological support of $\bm{P}$. In particular, if the support of $P_0$ is $\X$,  $\bm{P}$ has full weak support if its topological support coincides with the whole space $\cl{P}(\X)$ of probability measures over $(\X, \B_\X)$. 
Conditions for MSBp to enjoy this desirable property are given in the following.

\begin{prop}\label{prop:MSB_full_supp}
Let $\bm{P}$ be a MSBp with parameters $(\pi,\bmpsi,P_0)$. If for each $\varepsilon > 0$ there exists $\delta > 0$ such that 
$$
\int_{(\delta,\varepsilon)^{m}}\bpsi_{m-1}(v_{m-1},\mrm{d}v_m)\cdots \bpsi_1(v_1,\mrm{d}v_2)\,\pi(\mrm{d}v_1) >0,
$$
for every $m \geq 1$, then $\bm{P}$ has full support. In particular, if there exists $\epsilon > 0$ such that $(0,\epsilon)$ is contained in the support of $\pi$, and for each $v \in (0,\epsilon)$ and $j \geq 1$, $(0,\epsilon)$ is also contained in the support of $\bpsi_j(v,\cdot)$, then $\bm{P}$ has full support.
\end{prop}
In practice, this means that, as long as the Markov chain can be initialized at an arbitrarily small value and remain close to its initial state for any finite number of steps, the MSBp prior assigns positive mass to any distribution whose support is contained in that of $P_0$.

We now provide a general convergence result for MSBp. To this aim let us denote by $\wto$, $\dto$ and $\overset{wd}{\to}$ weak convergence, convergence in distribution, and weak convergence in distribution, respectively. See App.~\ref{app:theo:MSB_conv} 
for further background.

\begin{theo}\label{theo:MSB_conv}
	For $n \geq 1$, let $\bm{v}$ and $\bm{v}^{(n)}$ be Markov processes of  length variables with initial distributions $\pi$ and $\pi^{(n)}$, and collections of transition probability kernels $\bmpsi = (\psi_j)_{j=1}^{\infty}$ and $\bmpsi ^{(n)}= \big(\bpsi_j^{(n)}\big)_{j=1}^{\infty}$, respectively. Consider the corresponding MSBw, $\bm{w}$ and $\bm{w}^{(n)}$, with parameters $(\pi,\bmpsi)$ and $\l(\pi^{(n)},\bmpsi^{(n)}\r)$, and the proper MSBp over $(\X,\B_\X)$, $\bm{P}$ and $\bm{P}^{(n)}$, with parameters $\l(\pi,\bmpsi,P_0\r)$ and $\big(\pi^{(n)},\bmpsi^{(n)},P_0^{(n)}\big)$. If $\pi^{(n)} \wto \pi$, $P_0^{(n)} \wto P_0$, and for each $j \geq 1$ and $u_n \to u$ in $[0,1]$, $\bpsi_j^{(n)}(u_n,\cdot) \wto \bpsi_j(u,\cdot)$, as $n \to \infty$, then $\bm{v}^{(n)} \dto \bm{v}$, $\bm{w}^{(n)} \dto \bm{w}$, and $\bm{P}^{(n)} \dwto \bm{P}$.
\end{theo}

Under mild extra requirements over $(\X,\B_\X)$, $\bm{P}^{(n)} \dwto \bm{P}$ is equivalent to $\bm{P}^{(n)} \dto \bm{P}$; see App.~\ref{app:theo:MSB_conv}.
This result specializes nicely for many sub-families as we will show in Sections \ref{sec:BBSB} and \ref{sec:LMSB}.  In particular, we will explain  how to recover well-known models such as Dirichlet, Geometric and Pitman-Yor processes as limiting cases of MSBp. 

\subsection{Weights' orderings and clustering probabilities}\label{subsec:2}

One can always associate a species sampling sequence, $\bm{x} = (x_i)_{i=1}^{\infty}$ to a $\bm{P}$ in \eqref{eq:ssp}: conditionally on $\bm{P}$, its elements are sampled independently from $\bm{P}$, namely $x_i|\bm{P}\iid \bm{P}$ for any $i\ge 1$. 
Clearly, $\bm{x}$ is exchangeable, and, by de Finetti's theorem, $\bm{P}$ is the limit of the empirical distributions $n^{-1}\sum_{i=1}^{n}\delta_{x_i}$ as $n \to \infty$ a.s. Hence, the laws of $\bm{x}$ and $\bm{P}$ characterize each other. A random object that encodes important information about $\bm{x}$ and $\bm{P}$ is the random partition of $\{1,2,\ldots\}$ generated by the ties in $\bm{x}$. Namely, we can define the equivalence relation $i \bm{\sim}_{\bm{x}} j$ if and only if $x_i = x_j$, and for each $n \geq 1$, consider the random partition, $\Pi_n$, of $[n] = \{1,\ldots,n\}$,  generated by $\bm{\sim}_{\bm{x}}$. Using the terminology of \cite{P95} the collection $\bm{\Pi} = (\Pi_n)_{n=1}^{\infty}$ is consistent in the sense that $\Pi_n \vert_{[m]} = \{A \cap [m] \neq \emptyset: A \in \Pi_n\} = \Pi_m$, a.s.~for every $m < n$. Moreover, $\bm{\Pi}$ is exchangeable meaning that there exists a symmetric function $\Phi : \bigcup_{k=1}^{\infty} \N^k \to [0,1]$, such that
$$
\Prob[\Pi_n = \{A_1,\ldots,A_k\}] = \Phi(|A_1|,\ldots,|A_k|),
$$
for every $n \geq 1$ and each partition $\{A_1,\ldots,A_k\}$ of $[n]$. To emphasize the number of blocks $k$ and $n = \sum_{j=1}^{k} |A_j|$ we will often use the notation $\Phi^{(n)}_k$ instead of simply $\Phi$.
This function is called exchangeable partition probability function (EPPF) and plays an important role in several areas, including Population Genetics and Combinatorics; see \cite{P06} and references therein. 

The EPPF and the law of the weights, $\bm{w}$, are in one to one correspondence up to weights permutations. In fact, it follows from the definition of $\bm{\Pi}$ that:
\begin{equation}\label{eq:EPPF_sum}
\Phi_k^{(n)}(n_1,\ldots,n_k) = \sum_{(j_1,\ldots,j_k)}\Esp\l[\prod_{i=1}^k w_{j_i}^{n_i}\r],
\end{equation}
where the sum ranges over all vectors of size $k$ of distinct positive integers.  The almost sure connection between $\bm{w}$ and $\bm{\Pi}$ is made explicit by Kingman's representation theorem \citep{King72,King78}. It states that if $A_{n,1}^{\daw}$, \ldots $A_{n,k}^{\daw}$ are the blocks of $\Pi_n$ with non--increasing sizes, i.e. $|A_{n,j}^{\daw}| \geq |A_{n,j+1}^{\daw}|$, then $\lim_{n \to \infty} n^{-1}|A^\daw_{n,j}| = w^\daw_j$  where $w^\daw_1 \geq w^\daw_2 \geq \cdots$ are the weights of $\bm{P}$ in the decreasing order.  Alternatively, if $\bm{P}$ is proper and $\tilde{A}_{n,1},\ldots,\tilde{A}_{n,k}$ are the blocks of $\Pi_n$ in the least element order, so that $\min(\tilde{A}_{n,j}) \leq \min(\tilde{A}_{n,j+1})$, then $\lim_{n \to \infty} n^{-1}|\tilde{A}_{n,j}| = \tilde{w}_j := w_{\varrho_j}$, where  $(w_{\varrho_1},w_{\varrho_2},\ldots)$
are obtained by sampling without replacement from $(w_1,w_2,\ldots)$ with probabilities proportional to $(w_1,w_2,\ldots)$. Formally,
\begin{equation}\label{eq:sb_pick}
\Prob[\varrho_1 = j \mid \bm{w}]  = w_j, \quad	\Prob[\varrho_{k+1} = j \mid \bm{w},\bm{\varrho}_{[k]}]  = \frac{w_j}{1-\sum_{j=1}^k w_{\varrho_{j}}}\Ind\{j \not\in \bm{\varrho}_{[k]}\}, \quad k \geq 1,
\end{equation}
under the event $\sum_{j=1}^k w_{\varrho_{j}} < 1$ with  $\bm{\varrho}_{[k]} = (\varrho_1,\ldots,\varrho_k)$. Otherwise, if $\sum_{j=1}^k w_{\varrho_{j}} = 1$, we set $  \Prob[\varrho_{k+1} = j \mid \bm{w},\bm{\varrho}_{[k]}]  = \Ind\{j = \inf (\N\setminus\bm{\varrho}_{[k]})\}$. 
The collection $(w_{\varrho_j})_{j=1}^{\infty}$ is known as a \emph{size-biased permutation} of $(w_j)_{j=1}^{\infty}$ and we say that $\bm{w}$ is \emph{invariant under size-biased permutations}, or that it is in \emph{size-biased random order} if it satisfies 
$(w_j)_{j=1}^{\infty} \deq (w_{\varrho_j})_{j=1}^{\infty}$. 

In general, working with one ordering of the weights or another does not change the distribution of model, as illustrated by \eqref{eq:EPPF_sum}. Nonetheless, for historical and operational reasons the two aforementioned weights' rearrangements have gained relevance. Our next result gives necessary and sufficient conditions for a MSBw to be a.s. decreasing.
\begin{prop}\label{cor:MSB_dec_sup}
Let $\bm{w}$ be a MSBw  with parameters $(\pi,\bmpsi)$ and such that $w_j > 0$, for every $j \geq 1$, so that $t = \inf\{j \geq 1:v_j = 1\} = \infty$, a.s. Then $\bm{w}$ is decreasing a.s.~if and only if for every $j \geq 1$ the support of $\psi_j(v_j,\cdot\,)$ is contained in $[0,c(v_j)]$, $\pi_j$-a.s. where  $c(v) = 1\wedge v(1-v)^{-1}$ for $v \in (0,1)$.
\end{prop}

The characterization of almost sure decreasing weights $\bm{w}$ is beneficial for some computational algorithms in mixture models as highlighted, e.g., in \cite{MW15}. Nonetheless, the size-biased random order is arguably the most important weights' permutation in practice, as the EPPF  and some posterior or predictive distributions can be explicitly computed when the distribution of this weights' ordering is known. To explain this, let $\tilde{x}_1,\ldots,\tilde{x}_{K_n}$ denote the distinct values in $(x_i)_{i=1}^{n}$ in order of discovery so that $\tilde{x}_1 = x_1$ and $\tilde{x}_j = x_{m_j}$ with $m_j = \inf \{i > m_{j-1}: x_i \not\in \{x_1,\ldots,x_{i-1}\} \}$. Note that $\{i \leq n: x_i = \tilde{x}_j\} = \tilde{A}_{n,j}$, recalling that $\tilde{A}_{n,j}$ is the $j$th block of $\Pi_n$ in the least element order. 
Hence, if $\bm{P}$ is proper, the long-run proportion of $x_i$'s that coincide with $\tilde{x}_j$ is
$\tilde{w}_j = w_{{\varrho}_j}$, for $\bm{\varrho} = (\varrho_j)_{j=1}^{\infty}$ as in \eqref{eq:sb_pick}. It follows that given $\tilde{\bm{w}}$, one has a tractable conditional prediction rule:
\begin{equation}\label{eq:cond_pred_rule}
\Prob[x_{n+1} \in \cdot \mid x_1,\ldots,x_{n}, \tilde{\bm{w}}] = \sum_{j=1}^{K_n} \tilde{w}_j\delta_{\tilde{x}_j} + \l(1-\sum_{j=1}^{K_n}\tilde{w}_j\r)P_0, \quad  n \geq 1.
\end{equation}
Noting that $\bm{P} = \lim_{n \to \infty} n^{-1}\sum_{i=1}^{n} \delta_{x_i} = \sum_{j=1}^{\infty}\tilde{w}_j\delta_{\tilde{x}_j}$, \eqref{eq:cond_pred_rule} implies that $\bm{P}$ can be decomposed as  $\bm{P} = \bm{P}' + \bm{P}''$. The first component $\bm{P}' = \sum_{j=1}^{K_n} \tilde{w}_j\delta_{\tilde{x}_j}$ is finite-dimensional and captures the part of $\bm{P}$ that learns directly from $n$ data points $(x_i)_{i=1}^{n}$. The second component $\bm{P}'' = \sum_{j>K_n} \tilde{w}_j\delta_{x_j}$ only learns from the data indirectly through $\bm{P}'$. This fact is key for devising some \emph{conditional} sampling schemes such as the \emph{ordered allocation sampler} by \cite{DG23}. It also follows from \eqref{eq:cond_pred_rule} that for any partition $B$ of $[n]$,
$
\Prob[\Pi_n = B \mid \tilde{\bm{w}}] = \prod_{j=1}^k {\tilde{w}_j}^{n_j-1} \prod_{j=1}^{k-1} \l(1-\sum_{i=1}^j \tilde{w}_i\r),
$ 
where $n_j = |\tilde{B}_j|$ and $\tilde{B}_1,\ldots,\tilde{B}_k$ are the blocks of $B$ in the \emph{least element order}. Thus, one obtains the notable expression for the EPPF \citep[cf.][]{P95,P96,P96b}: 
\begin{equation}\label{eq:EPPF_sb}
\Phi_k^{(n)}(n_1,\ldots,n_k) = \Esp\l[\prod_{j=1}^k {\tilde{w}_j}^{n_j-1} \prod_{j=1}^{k-1} \l(1-\sum_{i=1}^j \tilde{w}_i\r)\r].
\end{equation}
Furthermore, by integrating out $\tilde{\bm{w}}$ in \eqref{eq:cond_pred_rule}, with respect to its conditional distribution given the data, one can also compute the (marginal) prediction rule $\Prob[x_{n+1} \in \cdot \mid x_1,\ldots,x_{n}]$. This in turn allows the description of the model as a generalized P\'olya urn scheme in the spirit of the seminal contributions of \cite{BM73} and \cite{P96}, as well as posterior inference by means of \emph{marginal} sampling schemes first proposed by \cite{EW95} and \cite{MacE94}.

In view of this, it is natural to ask whether there exist MSBw, other than the Pitman-Yor weights, that remain invariant under size-biased permutations. Given the relevance of this question, we devote the next section to addressing it.

\subsection{Pitman-Yor process characterization}\label{subsec:3}
A deep and celebrated characterization by \citet[Theorem 2]{P96b} essentially states that for stick-breaking weights $\bm{w}$ satisfying $\sum_{j=1}^{\infty} w_j = 1$,  $w_j > 0$, for every $j \geq 1$ a.s.~and featuring \textit{independent length variables},  $\bm{w}$ is invariant under size-biased permutations if and only if $v_j \sim \Be(1-\sigma,\theta+j\sigma)$ for some $\sigma \in [0,1)$ and $\theta > -\sigma$. These weights, first identified in \cite{PPY92}, are known as the size-biased permuted Pitman-Yor weights and the resulting stick-breaking process is the Pitman-Yor process \citep{P97}, also known as two parameter Poisson-Dirichlet process. We are able to derive a highly non-trivial and 
somehow surprising generalization of this result replacing the independence assumption for Markovianity: under mild conditions involving only the first two length variables or weights, the Pitman-Yor process is the only MSBp whose weights are invariant under size-biased permutations.

\begin{theo}\label{theo:MSB_sb_PY}
Let $\bm{w}$ be a MSBw 
such that $\sum_{j=1}^{\infty}w_j = 1$, and $w_j > 0$, for every $j \geq 1$, a.s. Let $\bm{v}$ be the length variables of $\bm{w}$ and say that at least one of following  
 holds: 
\begin{itemize}
\item[$(A)$]  The law of $\bm{v}_{[2]} = (v_1,v_2)$ and the Lebesgue measure over $([0,1]^2,\B_{[0,1]^2})$ are mutually absolutely continuous.
\item[$(B)$] The support of $\bm{w}_{[2]} = (w_1,w_2)$ is convex and there is a version of the conditional law $\Prob[v_3 \in \cdot \mid v_2]$ that is left continuous with respect to the weak topology.
\item[$(C)$] There exists $0 < \varepsilon < 1$ such that the set $\{(x_1,x_2) \in [0,1]^2: x_2 = \varepsilon\}$ is contained in the support of $\bm{v}_{[2]} = (v_1,v_2)$. Moreover, there exists a version, $\psi_2$, of the conditional law $\Prob[v_3 \in \cdot \mid v_2]$ that is continuous with respect to the weak topology, i.e. $\psi_2(u_n,\cdot) \wto \psi_2(u,\cdot)$  as $u_n \to u$ in the support of $v_2$.
\end{itemize}
Then $\bm{w}$ is invariant under size-biased permutations if and only if the length variables are independent with $v_j \sim \Be(1-\sigma, \theta+j\sigma)$ for some $\sigma \in [0,1)$ and $\theta > -\sigma$.
\end{theo}

This result is of independent interest, beyond Bayesian modeling, and its proof is technically challenging. App.~\ref{sec:app:MSB_sb}
is devoted to proving and discussing this remarkable characterization. Here we highlight the key argument behind this result by focusing on the conditional law of $v_3$ given $(w_1,w_2)$. In the proof of Corollary~7 in \cite{P96b} it is shown 
	that when the weights are invariant under size-biased permutations, there is a version, say $\mu$, of the mapping $(w_1,w_2) \mapsto \Prob[v_3 \in \cdot \mid w_1,w_2]$ that is symmetric, i.e. $\mu(\,\cdot\,;w_1,w_2) = \mu(\,\cdot\,;w_2,w_1)$. At the same time, 
	Markovianity of the length variables implies that 
	there is another version, $\nu$, of the same mapping $(w_1,w_2) \mapsto \Prob[v_3 \in \cdot \mid w_1,w_2]$ that only depends on the value of $v_2 = w_2/(1-w_1)$. In other words, for any set $A$ the function $(w_1,w_2) \mapsto \nu(A;w_1,w_2)$ is constant over the sets 
	$L_\varepsilon = \{(w_1,w_2): w_2/(1-w_1) = \varepsilon\}\cap \Delta_2$, where $\varepsilon \in [0,1]$ and $\Delta_2 = \{(w_1,w_2) \in [0,1]^2: w_1+w_2 \leq 1\}$. Since both $\mu$ and $\nu$ are 
	versions of the same conditional probability distribution, they must coincide (a.s.). 
	For the sake of illustration say $\mu = \nu$ over the whole space,  $\Delta_2$. Then, for any two points $\bm{w}' = (w'_1,w'_2)$ and $\bm{w}^* = (w^*_1,w^*_2)$ in the interior of $\Delta_2$, 
	there exist $\bm{z}' = (z_1,z_2)$ and $\bm{z}^* = (z_2,z_1)$  in $\Delta_2$ such that $\bm{w}',\bm{z}' \in L_{\varepsilon'}$ and $\bm{w}^*,\bm{z}^* \in L_{\varepsilon^*}$ for some $\varepsilon', \varepsilon^* \in (0,1)$  as illustrated in Figure \ref{fig:thm2_8}. Since $\nu$ is constant over $L_{\varepsilon'}$ and  $L_{\varepsilon^*}$ we get $\nu(\,\cdot\,;\bm{w}') = \nu(\,\cdot\,;\bm{z}' )$ and  $\nu(\,\cdot\;\bm{w}^*) = \nu(\,\cdot\,;\bm{z}^*)$. Similarly, the symmetry of $\mu$ yields $\mu(\,\cdot\,;\bm{z}' ) = \mu(\,\cdot\,;\bm{z}^*)$. Recalling that $\mu = \nu$, we obtain $\mu(\,\cdot\,;\bm{w}') = \mu(\,\cdot\,;\bm{w}^*)$, i.e., both $\mu$ and $\nu$ are constant. This means $v_3$ is independent the first two weights (or equivalently of the first two length variables). In general, we cannot assure $\mu=\nu$  all over $\Delta_2 $, but the conditions (\emph{A})--(\emph{C}) guarantee $\mu$ and $\nu$ coincide in sufficient points so as to apply a similar argument. Broadly speaking, the application of the aforementioned trick along consecutive shifts of the length variables $(v_1,v_2,\ldots) \mapsto (v_2,v_3,\ldots)$ results in the independence of all length variables, and so the result follows.

\begin{figure}
	\centering
	\includegraphics[width=0.5\textwidth]{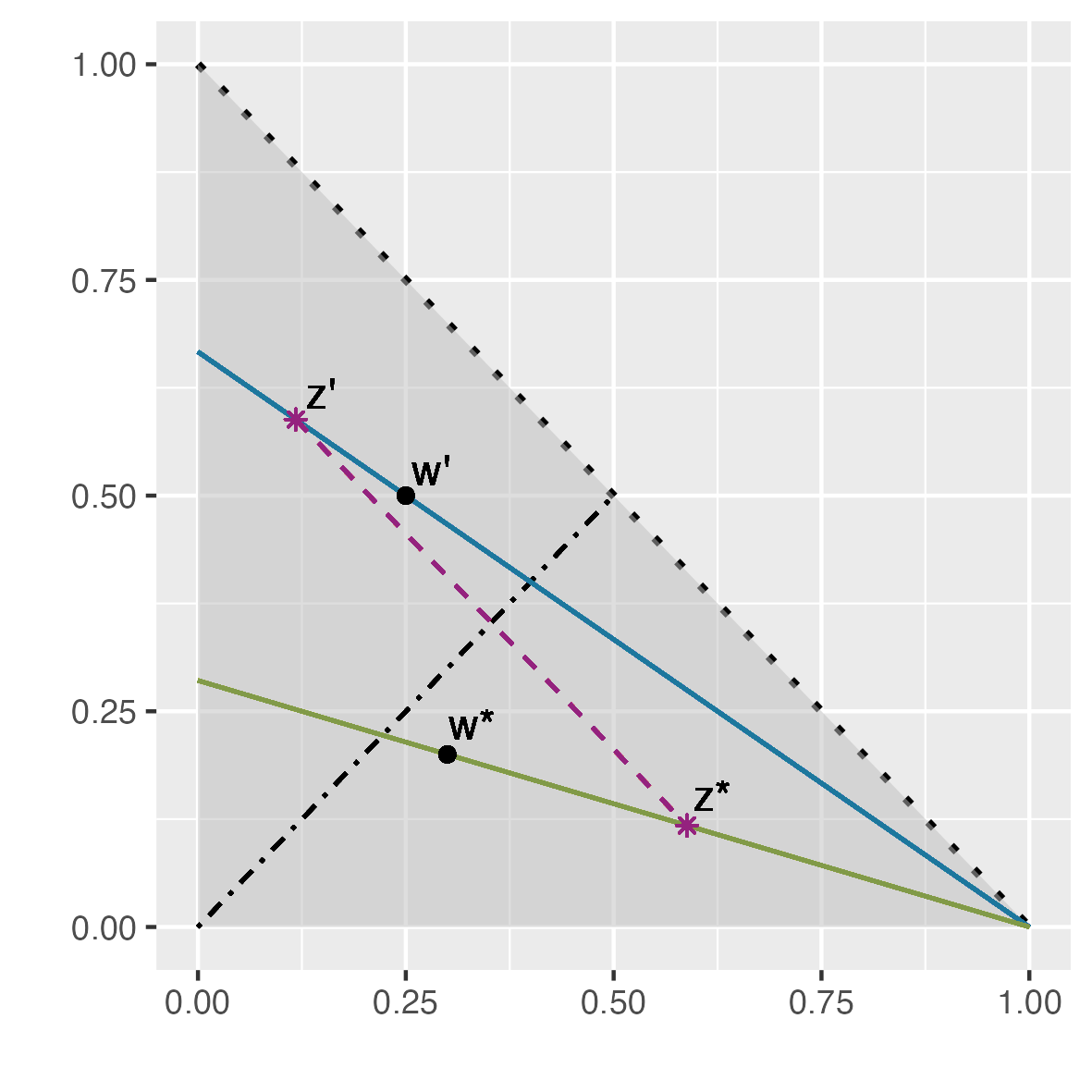}
	\begin{small} 
		\caption{\footnotesize{Illustration of Theorem \ref{theo:MSB_sb_PY}. The upper and lower solid lines stand for $L_{\varepsilon'}$ and $L_{\varepsilon^*}$, respectively. The red dashed line shows that $\bm{z}'$ is a reflection of $\bm{z}^*$ with respect to the identity function (dotted-dashed line).}}
		\label{fig:thm2_8}
	\end{small}
\end{figure}

\subsection{Posterior inference and simulation}\label{subsec:2_old}

For	general classes of species sampling processes, the explicit computation of the distribution of the size-biased weights and the EPPF is out of reach. By virtue of Theorem \ref{theo:MSB_sb_PY}, this is also the case for the vast majority of MSBp. However, the lack of explicit formulae does not hinder the investigation of prior and posterior properties. 
	A similar challenge arises for other well-known and tractable random probability measures, such as mixtures of the Dirichlet process. Even for these models, one must rely on sampling schemes to investigate the key features.
One of the strengths of 
MSBp is the possibility of easily generating $\tilde{w}_1,\ldots,\tilde{w}_j$ by sampling the MSBw $\bm{w}$ in the original order as well as the sequence $\bm{\varrho}$, which is defined in \eqref{eq:sb_pick}. 
Using the samples of $\tilde{w}_1,\ldots,\tilde{w}_j$ quantities of interest, such as  \eqref{eq:EPPF_sb} and other functional moments of $\bm{P}$,
can be numerically approximated.

The application of simulation-based approximation methods will be showcased through: (i) the distribution of the number of partition blocks $K_n = |\Pi_n|$ and (ii) the \emph{tie probability} $\tau_{p} = \Prob[x_i = x_j] = \Phi_1^{(2)} (2) = \Esp[\tilde{w}_1^2]$. Fundamental functional moments of the associated species sampling process depend on $\tau_p$. An example is 
$$
\Esp[\bm{P}(f)\bm{P}(g)] = \tau_p P_0(fg) + (1-\tau_p)P_0(f)P_0(g),
$$
for any pair $f,g: \X \to \R$ of bounded measurable functions, where $\bm{P}(f) = \int_{\X} \, f \, \mathrm{d}\bm{P}$ and similarly $P_0(f) = \int_{\X} \, f \, \mathrm{d}P_0$. From this, one can deduce $\Cov(\bm{P}(f),\bm{P}(g)) =\tau_p\{P_0(fg)-P_0(f)P_0(g)\}$ and $\Var(\bm{P}(f))=\tau_p\{P_0(f^2)-P_0^2(f)\}$. In particular, if $f=\Ind_A$ and $g=\Ind_B$ for any $A,B \in \B_{\X}$, 
one has
$$
		\Cov(\bm{P}(A),\bm{P}(B)) = \tau_p\,\left[P_0(A\cap B)-P_0(A)P_0(B)\right],
$$
and $\Var(\bm{P}(A)) = \tau_p \, P_0(A)[1-P_0(A)]$; see App.~\ref{app:funct_moments}. For any diffuse probability measure $P_0$ and bounded measurable $f:\X\to\R$, larger values of the tie probability $\tau_p$ lead to larger values of $\Var(\bm{P}(f))$. 
From a methodological perspective, this means that by increasing $\tau_p$ one obtains a less informative prior on $\bm{P}$. 
However, it should be noted that large values of $\tau_p$ have the opposite effect on 
the number of partition blocks 
$K_n = |\Pi_n|$. Indeed, the more likely two observations are to be clustered together, the fewer the number of clusters will be. This typically leads to 
smaller values of both $\Esp[K_n]$ and $\Var(K_n)$ as illustrated in the following sections.
When the EPPF and the distribution of $\tilde{w}_1$ are not available, 
$\tau_p$ can be approximated by averaging over different samples of $\tilde{w}_1$. In turn, to draw a sample of $\tilde{w}_1$ it suffices to sample $u \sim \mathsf{Unif}(0,1)$, the MSBw $w_1,\ldots,w_{\varrho_1}$ with $\varrho_1 = \min\{j: \sum_{i=1}^{j}w_i > u$\} and set $\tilde{w}_1 = w_{\varrho_1}$. Similarly, we can readily draw samples of $K_n = \max\{j: \sum_{i=1}^j G_i \leq n\}$ where $G_1,G_2,\ldots$ are conditionally independent given $(\tilde{w}_j)_{j=1}^{\infty}$, with  $G_1 = 1$ and $G_i \sim \mathsf{Geo}(1-\sum_{l=1}^{i-1} \tilde{w}_l)$, for $i > 1$; see App.~\ref{app:Kn_sum_geo}.

For general stick-breaking processes, numerical approximation of prior distributional features of $\bm{P}$ is feasible if one is able to sample from the joint distribution of the length variables $\bm{v}_{[m]} = (v_1,\ldots,v_m)$, for any $m \geq 1$. For MSBp, this simplifies to sampling from the initial distribution $\pi$ and the transition kernels $\psi_j(v,\cdot)$, for any $v \in [0,1]$ and $j \geq 1$.
As for posterior inference, 
it can 
be addressed via \textit{conditional algorithms} such as, e.g., slice \citep{W07}, retrospective \citep{PR08}, or order allocation \citep{DG23} sampling.
They rely on the use of latent allocation variables $\bm{d} = (d_i)_{i=1}^{\infty}$ with $d_i = j$ if and only if $x_i = \theta_j$, so that $x_i = \theta_{d_i}$. The ties among elements of $\bm{d}$ define the same partition structure, $\bm{\Pi}$, generated by the ties in $\bm{x}$ a.s. We have $d_i \mid \bm{w} \iid \sum_{j=1}^{\infty}w_j \delta_j$, and, due to the stick-breaking decomposition, the joint distribution of $(d_1,\ldots,d_n)$ given $\bm{v}$ is
\begin{equation}\label{eq:all_var_cond}
	\bm{p}(d_1,\ldots,d_n\mid \bm{v}) =  \prod_{j=1}^{\kappa} w_j^{a_j} = \prod_{j=1}^{\kappa} v_j^{a_j} (1-v_j)^{b_j},
\end{equation}
where $\kappa = \max_{i \leq n}d_i$, $a_j = |\{i \leq n: d_i = j\}|$ and $b_j = \sum_{l > j}a_l$. 
In order to adapt any conditional method
to a stick-breaking process, the key quantity is the posterior density
\begin{equation}\label{eq:v_post}
	\bm{p}(\bm{v}_{[m]}\mid d_1,\ldots,d_n) \propto \bm{p}(\bm{v}_{[m]})\prod_{j=1}^{\kappa} v_j^{a_j} (1-v_j)^{b_j},
\end{equation}
for every $m \geq 1$, where $a_j$ and $b_j$ are as in \eqref{eq:all_var_cond},  and $\bm{p}(\bm{v}_{[m]})$ is the prior density of $\bm{v}_{[m]}$. 
When sampling from \eqref{eq:v_post} is hard, one can instead work with the element-wise full conditional
\begin{equation*}\label{eq:v_full_cond}
	\bm{p}(v_j\mid \bm{v}_{-j},d_1,\ldots,d_n) \propto v_j^{a_j} (1-v_j)^{b_j} \bm{p}(v_j \mid \bm{v}_{-j}),
\end{equation*}
where $\bm{v}_{-j} = (v_1,\ldots,v_{j-1},v_{j+1},\ldots)$, for every $j \geq 1$.

For proper MSBp with initial distribution, $\pi$, and transition probability kernels, $\psi_j(v,\cdot)$, admitting densities (denoted by the same symbols), it follows that 
\begin{equation*}
\bm{p}(\bm{v}_{[m]}\mid d_1,\ldots,d_n) \propto \pi(v_1)v_1^{a_1}(1-v_1)^{b_1}\, \prod_{j=2}^{m}v_j^{a_j}(1-v_j)^{b_j}\bpsi_{j-1}(v_{j-1},v_j),
\end{equation*}
for every $m \geq 1$, with $a_j$ and $b_j$ as in \eqref{eq:all_var_cond}. We also have
\begin{equation}\label{eq:MSB_post_v1}
\bm{p}(\bm{v}_1\mid \bm{v}_{-1},d_1,\ldots,d_n) \propto v_1^{a_1}(1-v_1)^{b_1}\pi(v_1)
\bpsi_1(v_{1},v_{2}),
\end{equation}
and, for each $j \geq 2$,
\begin{equation}\label{eq:MSB_post_vj}
\bm{p}(\bm{v}_j\mid \bm{v}_{-j},d_1,\ldots,d_n) \propto v_j^{a_j}(1-v_j)^{b_j}\bpsi_{j-1}(v_{j-1},v_j)
\bpsi_j(v_{j},v_{j+1}).
\end{equation}
If tractable transitions are specified, as in Sections~\ref{sec:BBSB} and \ref{sec:LMSB}, it is straightforward to sample from $\bm{p}(\bm{v}_{[m]})$, $\bm{p}(\bm{v}_{[m]}\mid d_1,\ldots, d_n)$ or $\bm{p}(v_j\mid \bm{v}_{-j},d_1,\ldots,d_n)$ for MSBp.

Another key feature of sampling algorithms is that they must have a reasonable run time. Namely,  many algorithms require to sample weights $w_1,\ldots,w_J$, where $J = \min\{j \geq 1: \sum_{i=1}^{j} w_i > U\}$ and $U$ is some random threshold in $(0,1)$ that changes from one algorithm to another. Such thresholds appear in posterior sampling schemes such as the \emph{slice sampler} \citep{KGW11} and the \emph{ordered allocation sampler} \citep{DG23}. Note that $\sum_{i=1}^{J}w_i > U$ is equivalent to $\prod_{i=1}^{J}(1-v_i) = 1-\sum_{i=1}^{J}w_i < 1-U$. For length variables $v_i$ that place a lot of mass on values close to $0$, a larger $J$ (sometimes, orders of magnitude larger) is needed to satisfy the inequality, thereby increasing the computational burden. 
In extreme cases, where weights are heavy-tailed, for example if $v_i \sim \mathsf{Be}(\alpha,\beta_i)$, with $\beta_i$ very large, it may well happen that $v_i$ becomes so negligible that the algorithm will return a numerical value equal to $0$ and the condition $\prod_{i=1}^{J}(1-v_i)  < 1-U$ will not be met. In this case, the algorithm will not terminate in finite time.
In view of this, Sections \ref{sec:BBSB} and \ref{sec:LMSB} present examples of MSBp with flexible and tractable transitions $\psi_j$ and such that the marginal distributions of the length variables $\pi_j$ can be chosen in advance.
We next specialize these general notions to two extreme dependence regimes, which will subsequently serve as building blocks for more flexible subclasses.

%--------------------------------------------------------------------------------------------------------------------------------
\section{Classes of degenerate MSBp}\label{sec:examples}
%--------------------------------------------------------------------------------------------------------------------------------
The first subclasses of MSBp we consider correspond to the degenerate extreme cases: independent and completely dependent length variables.

\subsection{Independent stick-breaking processes}
The independent case arises by taking the transition probability kernels of the form
\[
\psi_j(v_j,B) = \Prob[v_{j+1}\in B\mid v_j] = \Prob[v_{j+1} \in B] = \pi_{j+1}(B),
\]
i.e. $\pi_{j+1}$ is both the transition probability kernel at time $j$ and the marginal distribution of $v_{j+1}$. The corresponding $\bm{w}$ and $\bm{P}$ are termed, respectively, \emph{independent stick-breaking weights sequence} (ISBw) with parameter $\bm{\pi} = (\pi_j)_{j=1}^{\infty}$ and \emph{independent stick-breaking process} (ISBp) with parameters $(\bm{\pi},P_0)$. This represents the class of stick-breaking processes most widely studied in the literature. It is well-known \citep[see e.g.][]{GvdV17} that for ISBw the \textit{properness} condition $\sum_{j=1}^{\infty}v_j = \infty$ in \eqref{eq:sum1_MSB} can be relaxed to 
\begin{equation} \label{eq:proper}
\sum_{j=1}^{\infty}\int_{[0,1]}v\pi_j(\mrm{d}v) = \sum_{j = 1}^{\infty} \Esp[v_j] = \infty.
\end{equation}
For iid length variables,
where $\pi_j = \pi$, for $j \geq 1$, $\sum_{j=1}^{\infty}w_j = 1$ if and only if $\pi \neq \delta_0$. The condition on $\pi$ matches that in Theorem \ref{theo:sMSB_1}(iii): iid length variables form a stationary Markov process with transition kernel $\psi(v,\cdot)=\pi$, whose unique invariant measure is $\pi$.
Moreover, if $(0,\varepsilon)$ belongs to the support of $\pi_j$ for some $\varepsilon > 0$ and every $j \geq 1$, the ISBp $\bm{P}$ has full support \citep{BO14}.   Moreover, the posterior distribution of the length variables given the allocation variables was first derived in \cite{IJ01}.

The most popular instance of ISBp is the Pitman-Yor process, introduced in the previous section: it is proper, which follows by noting that $\Esp[v_j]=(1-\sigma)/(1+\theta + (j-1)\sigma)$ thus satisfying \eqref{eq:proper}, and clearly has full support.
Furthermore, if $\sigma=0$, 
$v_j \iid \Be(1,\theta)$ and $\bm{P}$ becomes a Dirichlet process \citep{F73,S94}, the most widely used nonparametric model.
There are other ISBw that remain invariant under size-biased permutations; however, for all of these there exist $m \geq 1$ such that $w_j = 0$ a.s.~for every $j > m$, which means that the associated species sampling process will not have full support.

\subsection{Completely dependent stick-breaking processes}
The other extreme case 	corresponds to complete dependence. Consider first the special case of identical length variables $\bm{v}$, which can be seen as the counterpart to the iid case. The resulting stick-breaking process $\bm{P}$ is known as Geometric process \citep{FMW10}. Within the MSBp framework we recover it by noting that
$\bm{v} = (v,v,\ldots)$ is trivially a stationary Markov process with transition probability kernel $\bpsi(v,B) = \delta_v(B)$ for every $v \in [0,1]$ and $B \in \B_{[0,1]}$. For this transition it holds that
$
\pi(B) = \int_{[0,1]} \delta_v(B)\pi(\mrm{d}v)
$
for any  probability measure $\pi$ over $([0,1],\B_{[0,1]})$.  Hence, every probability measure, $\pi$, is invariant for $\bpsi$, which means that it is important to require $\pi(\{0\}) = \Prob[v = 0] = 0$ to attain properness, i.e., $\sum_{j=1}^{\infty} w_j =1$ a.s. Moreover, $t$, as defined in Proposition \ref{prop:MSB_supp_points}, is either infinite or equal to one a.s.~and $\Prob[t = 1] = \pi(\{1\})$: thus, we require $\pi(\{1\}) = 0$, to avoid $\bm{P} = \delta_{\theta_1}$ with positive probability. As for the full support of Geometric processes, it is achieved if $(0,\varepsilon)$ is contained in the support of $\pi$, for some $\varepsilon > 0$. For this process, the MSBw reduce to $w_j = v(1-v)^{j-1}$ and, thus, $\bm{w}$ is decreasing a.s. 
Also the posterior distribution given the allocation variables simplifies to
$
\bm{p}(v\mid d_1,\ldots,d_n) \propto v^{n}(1-v)^{\sum_{j=1}^mb_j}\,\pi(v).
$

Beyond the geometric process, a general approach to define a new class of nonparametric priors corresponding to the extreme case of complete dependence is as follows.
Given a collection of probability measures, $\bm{\pi} = (\pi_j)_{j=1}^{\infty}$, on $([0,1],\B_{[0,1]})$,  the goal is to define a completely dependent sequence of length variables, $\bm{v}$, with marginals $\bm{\pi}$. This can be easily achieved if the distribution function $F_j$ of $\pi_j$ is continuous and strictly increasing on $[0,1]$, thus assuring existence of the inverse, $F_j^{-1}$. In this case, we can define $\Upsilon_j : [0,1] \to [0,1]$ as
\begin{equation}\label{eq:Upsilon}
\Upsilon_j(v) = (F_{j+1}^{-1} \circ {F_j})(v) = F_{j+1}^{-1}({F_j}(v)),
\end{equation}
for every $j \geq 1$. At this stage simply choose $v_1 \sim \pi_1$ and recursively, for $j \geq 1$, set $v_{j+1} = \Upsilon_j(v_j)$. If  $v_j \sim \pi_j$ then $u_j = F_j(v_j) \sim \mathsf{Unif}(0,1)$ which yields $v_{j+1} = F^{-1}_{j+1}(u_j) \sim \pi_{j+1}$. Thus, $\bm{v}$ is completely determined by $v_1$, marginally $v_j \sim \pi_j$ as desired, and it defines a Markov process with initial distribution $\pi = \pi_1$ and transition probability kernels 
\begin{equation}\label{eq:trans_Upsilon}
\psi_{j}(v_j,\cdot) = \Prob[v_{j+1}\in \cdot \mid v_j] = \delta_{\Upsilon_j(v_j)}.
\end{equation}
We call the corresponding $\bm{w}$ and $\bm{P}$, respectively, \emph{(completely) dependent stick-breaking weights sequence} (CDSBw) with parameter $\bm{\pi} = (\pi_j)_{j=1}^{\infty}$ and \emph{(completely) dependent stick-breaking process} (CDSBp) with parameters $(\bm{\pi},P_0)$.

We first investigate the two key properties, namely properness and full support, of CDSBp.
The requirement for $F_j$ to be continuous and strictly increasing assures that $\pi_j$ is diffuse. Hence, in particular, $\pi(\{1\}) = \Prob[v_j = 1] = 0$, which implies $\Prob[\bigcup_{j=1}^{\infty}\{v_j = 1\}] = 0$. From Proposition \ref{prop:MSB_supp_points}, the respective stick-breaking weights $\bm{w}$ satisfy $w_j > 0$ a.s. Additionally, from \eqref{eq:sum1_MSB} we get $\sum_{j=1}^{\infty}w_j = 1$ if and only if
\begin{equation}\label{eq:sum_1_Upsilon}
\sum_{j=1}^{\infty}v_j = \sum_{j=1}^{\infty}\Upsilon^{[j]}(v) = \infty \quad \text{a.s.}
\end{equation}
with $v \sim \pi_1$,  $\Upsilon^{[j]} = \Upsilon_{j-1}\circ \cdots \circ \Upsilon_0 = F^{-1}_{j}\circ F_1$ and $\Upsilon_0$ standing for the identity function. 
The following result provides a sufficient and a necessary condition for \eqref{eq:sum_1_Upsilon} to hold. Additionally, it gives conditions on the distribution functions $F_j$ to assure $\bm{P}$ has full support and that the weights are decreasingly ordered analogous to Geometric weights.

\begin{theo}\label{theo:Upsilon}
Let $\bm{\Upsilon} = (\Upsilon_j)_{j=1}^{\infty}$ be as in \eqref{eq:Upsilon}. Consider the completely dependent Markov process, $\bm{v}$, with initial distribution $\pi = \pi_1$ and transition probability kernels as in \eqref{eq:trans_Upsilon}, and the corresponding CDSBw, $\bm{w}$, and CDSBp, $\bm{P}$.
\begin{enumerate}
\item[\emph{(i)}]
\begin{enumerate}
\item[\emph{a)}]If for some $\varepsilon > 0$ and any $v \in (0,\varepsilon)$ there exists $ J \geq 1$  and a divergent series $\sum_{j=1}^{\infty} a_j = \infty$ with $0 < a_{j+1} \leq a_j \leq 1$, such that
\[
F_j(v) \geq F_{j+1}(v \,a_{j+1}/a_{j}),
\]
for each $j \geq J$, then \eqref{eq:sum_1_Upsilon} holds, i.e. $\sum_{j=1}^{\infty}w_j = 1$ a.s.
\item[\emph{b)}] If $\sum_{j=1}^{\infty}\Esp[v_j] < \infty$ then $\sum_{j=1}^{\infty} w_j < 1$ a.s.
\end{enumerate} 
\item[\emph{(ii)}] If $F_j(v) \leq F_{j+1}(v)$ for each $v \in [0,1]$ and $j \geq 1$ then $\bm{w}$ is decreasing. Furthermore, under the same condition, if $\bm{P}$ is proper then $\bm{P}$ has full support.
\end{enumerate}
\end{theo}

When the marginals $\bm{\pi}$ are identical, i.e. $\pi_j = \pi$ for every $j \geq 1$, the conditions (i)a) and (ii) in  Theorem \ref{theo:Upsilon} are trivially satisfied, $\Upsilon_j$ becomes the identity function and we recover the identical length variables of a Geometric process. Other noteworthy cases of CDSBp which are proper and have full support are highlighted next. When $\pi_j = \Be(\alpha_j,1)$ with $1 \leq \alpha_{j+1} \leq \alpha_{j}$,  $\Upsilon_j(v) = v^{\alpha_j/\alpha_{j+1}}$ and $\Upsilon^{[j]}(v) = v^{\alpha_1/\alpha_{j}}$. Whilst the requirement $\alpha_{j} \geq 1$ implies $\Upsilon^{[j]}(v) \geq v^{\alpha_1}$ for $j \geq 1$, which in turn yields \eqref{eq:sum_1_Upsilon}, the requirement $\alpha_{j+1} \leq \alpha_{j}$ implies $F_j(v) \leq F_{j+1}(v)$ as in Theorem \ref{theo:Upsilon} (ii). Similarly, for $\pi_j = \Be(1,\beta_j)$ we get closed form expressions $\Upsilon_j(v) = 1-(1-v)^{\beta_j/\beta_{j+1}}$ and $\Upsilon^{[j]}(v) = 1-(1-v)^{\beta_1/\beta_{j}}$. In this case, if $\beta_j \leq \beta_{j+1}$, then $F_j(v) \leq F_{j+1}(v)$, as required by Theorem \ref{theo:Upsilon} (ii). Furthermore, if $\beta_j \leq 1$ for each $j \geq 1$, then $\Upsilon^{[j]}(v) \geq 1-(1-v)^{\beta_1}$, which yields \eqref{eq:sum_1_Upsilon}. If instead we choose $\beta_j = \theta + j\sigma$ for $0 \leq \sigma \leq 1$ and $\theta > -\sigma$, it can be shown that \eqref{eq:sum_1_Upsilon} also holds via Theorem \ref{theo:Upsilon} (i.a). More generally, despite the lack of a closed form expression for $\Upsilon_j$ and $\Upsilon^{[j]}$, one can prove that the choice $\pi_j = \Be(\alpha,\theta+j\sigma)$, with $0 < \alpha \leq 1$ and $\theta$ and $\sigma$ as above, also defines a proper CDSBp with full support; see Lemma 
\ref{lem:C}
in the Appendix. The case $\alpha = 1-\sigma$ and $\sigma < 1$ will be particularly important in the following sections as the marginals $\pi_j$ equal to those of a Pitman-Yor model, though the resulting CDSBp is clearly not a Pitman-Yor process. 

\begin{cor}\label{cor:CDSB_PY}
Let $\bm{P}$ be a CDSBp with parameters $\bm{\pi} = (\pi_j)_{j=1}^{\infty}$ where $\pi_j = \Be(1-\sigma,\theta+j\sigma)$ for $0 < \sigma < 1$ and $\theta > -\sigma$. Then $\bm{P}$ is proper and it has full support. Moreover the CDSBw, $\bm{w}$, of $\bm{P}$ satisfy $w_{j} > 0$ and $w_{j+1}\leq w_j$ for every $j \geq 1$, a.s.
\end{cor}
	The preceding constructions formalize two limiting dependence regimes: independent lengths, which give rise to Dirichlet and Pitman--Yor process priors, and completely dependent lengths, which yield geometric and CDSBp priors. The processes studied next interpolate between these extremes while preserving tractability.
 
%------------------------------------------------------------------------------------------------------------
\section{Beta Markov stick-breaking processes}\label{sec:BBSB}
%------------------------------------------------------------------------------------------------------------

In the previous section we have investigated degenerate classes of MSBp corresponding to the extreme cases of independence and complete dependence. 
It is therefore natural to look for a non-degenerate subclass of MSBp, which encompasses both dependence structures while preserving tractability. 

\subsection{Construction, properness and support}\label{subsec:BMSBP_1}

Consider $\pi_j = \Be(\alpha_j,\beta_j)$, for $j\ge 1$, with $\bm{\alpha} = (\alpha_j)_{j=1}^{\infty}$ and $\bm{\beta} = (\beta_j)_{j=1}^{\infty}$ two sequences of positive numbers and let $\Upsilon_j$ be as in \eqref{eq:Upsilon} with $F_j(v) = \mathcal{I}_v(\alpha_j,\beta_j)$, where $\cl{I}_x(a,b) = \int_0^x \Be(\mrm{d}u\mid a,b)$ is the regularized incomplete Beta function. Define the transition probability kernels  
\begin{equation}\label{eq:BB_trans}
{\psi_j}(v,\mathrm{d}u) = \sum_{z=0}^N  \Be(\mrm{d}u\mid \alpha_{j+1} + z, \beta_{j+1}+N-z) \, \Bin(z\mid N,\Upsilon_j(v)),
\end{equation}
for some $N \in \{0,1,\ldots\}$. We call $\bmpsi = (\psi_j)_{j=1}^{\infty}$  Beta-Binomial transitions with parameters $(N,\bm{\alpha},\bm{\beta})$. 
It is easy to check that for a Markov process $\bm{v}$  with initial distribution $\pi = \pi_1 = \Be(\alpha_1,\beta_1)$ and Beta-Binomial transitions $\bm{\psi}$ as in \eqref{eq:BB_trans} marginally $v_j \sim \Be(\alpha_j,\beta_j)$, for every $j \geq 1$; see App.~\ref{app:prop:BB_trans_marg}.
Consequently, we call  the corresponding $\bm{w}$ and $\bm{P}$, respectively, \emph{Beta Markov stick-breaking weights sequence} (BMSBw) with parameters $(N,\bm{\alpha},\bm{\beta})$ and \emph{Beta Markov stick-breaking processes} (BMSBp) with parameters $(N,\bm{\alpha},\bm{\beta},P_0)$. In practice, the parameter $N$ controls the degree of dependence: for small $N$ the BMSBp prior is close to independent regimes such as the Dirichlet or Pitman-Yor process, whereas for large $N$ it approaches completely dependent regimes such as the CDSBp or geometric case. 
The next result shows that BMSBp have full support and, under suitable conditions, are also proper.

\begin{theo}\label{theo:BMSBp}
Let $\bm{P}$ be any BMSBp with parameters $(N,\bm{\alpha},\bm{\beta},P_0)$ and $\bm{w} = (w_j)$ be the associated BMSBw. Then,
\begin{itemize}
\item[\emph{(i)}] If $\ \sum_{j=1}^{\infty} \alpha_j/(\alpha_j+\beta_j) = \infty \ $ and $ \ \limsup_{j \to \infty} (\alpha_j+\beta_j)^{-1} < \infty$, $ \ \bm{P}$ is proper.
\item[\emph{(ii)}] For every $j \geq 1$, $w_j > 0$ a.s.
\item[\emph{(iii)}] $\bm{P}$ has full support. 
\end{itemize}
\end{theo}

In summary, Theorem~\ref{theo:BMSBp} identifies the range of parameters for which the BMSBp defines a proper random probability measure, ensuring that the associated mixture models are well defined and suitable for practical Bayesian inference. For $\alpha_j = \alpha$ and $\beta_j = \beta$, for  $j \geq 1$, the conditions in Theorem \ref{theo:BMSBp} (i) are trivially satisfied and $\bm{P}$ is proper. In this case  
$\Upsilon_j(v) = v$,  simplifying \eqref{eq:BB_trans}, and we recover the models considered by \cite{GMN20} as special cases. The most interesting scenario arises when the marginal distributions of the length variables coincide with those of a Pitman-Yor model, i.e. $\alpha_j = 1-\sigma$ and $\beta_j = \theta + j\sigma$. We have that $\sum_{j=1}^{\infty} \alpha_j/(\alpha_j+\beta_j) = (1-\sigma)\sum_{j=1}^{\infty} (1+\theta+(j-1)\sigma)^{-1}  = \infty$ and $\lim_{j\to \infty} (1+\theta+(j-1)\sigma)^{-1} = 0$. Hence, the corresponding BMSBp is proper and has full support. With this in mind, the following result shows that the class of BMSBp achieves the goal of recovering 
the Pitman-Yor process and the CDSBp in Corollary \ref{cor:CDSB_PY} as limiting cases. For $n \in \{0,1,\ldots\}$, let $P_0^{(n)}$ be a diffuse probability measure over $(\X,\B_\X)$, let $N^{(n)} \in \{0,1,\ldots\}$ and consider some positive sequences $\bm{\alpha}^{(n)} = \big(\alpha^{(n)}_j\big)_{j=1}^{\infty}$ and $\bm{\beta}^{(n)} = \big(\beta^{(n)}_j\big)_{j=1}^{\infty}$ such that:
\begin{itemize}
    \item[(H0)] $N^{(0)} = 0$ and $N^{(n)} \to \infty$ as $n \to \infty$.
	\item[(H1)] $\alpha_j^{(n)} \to \alpha_j>0$ and $\beta_j^{(n)} \to \beta_j>0$ as $n \to \infty$, for any $j \geq 1$;
	\item[(H2)] the limiting sequences where $\bm{\alpha} = (\alpha_j)_{j=1}^{\infty}$ and $\bm{\beta} = (\beta_j)_{j=1}^{\infty}$ are such that $\bm{\Upsilon} = (\Upsilon_j)_{j=1}^{\infty}$ satisfies \eqref{eq:sum_1_Upsilon}, with $\Upsilon_j$ as in \eqref{eq:Upsilon} for $F_j(v) = \mathcal{I}_{v}(\alpha_j,\beta_j)$.
	\item[(H3)] $P_0^{(n)} \wto P_0$, as $n \to \infty$, for some diffuse probability measure, $P_0$, over $(\X,\B_\X)$.
\end{itemize}

\begin{theo}\label{cor:BMSBp_conv}
For each $n \in \{0,1,\ldots\}$ let $\bm{P}^{(n)}$ be a proper BMSBp
with parameters $\big(N^{(n)},\bm{\alpha}^{(n)},\bm{\beta}^{(n)},P_0^{(n)}\big)$ and $\bm{w}^{(n)}$ the associated BMSBw. Moreover, the parameters satisfy {\rm (H0)}--{\rm (H3)}. Then the following hold true
\begin{itemize}
\item[\emph{(i)}] For $n = 0$, 
$\bm{P}^{(0)}$ is an ISBp with parameters $\big(\bm{\pi}^{(0)},P^{(0)}_0\big)$ where $\bm{\pi}^{(0)} = \big(\pi^{(0)}_j\big)_{j=1}^{\infty}$ with $\pi^{(0)}_j = \Be\big(\alpha^{(0)}_j,\beta^{(0)}_j\big)$. In particular, if $\alpha^{(0)}_j = 1-\sigma$, and $\beta^{(0)}_j = \theta + j \sigma$, for $0 \leq \sigma < 1$ and $\theta > -\sigma$, $\bm{P}^{(0)}$ is a Pitman-Yor process. 
\item[\emph{(ii)}] 
If $P_0^{(n)} \wto P_0$, 
then $\bm{P}^{(n)} \dwto \bm{P}$ and $\bm{w}^{(n)} \dto \bm{w}$, as $n\to\infty$, for some CDSBp $\bm{P}$, with parameters $(\bm{\pi},P_0)$, and a CDSBw, $\bm{w}$, with parameter $\bm{\pi} = (\pi_j)_{j=1}^{\infty}$, where $\pi_j = \Be(\alpha_j,\beta_j)$. In particular, if $\alpha_j = 1-\sigma$ and $\beta_j = \theta+j\sigma$, for some $0 \leq \sigma < 1$ and $\theta > -\sigma$, $\bm{P}$ is as in Corollary \ref{cor:CDSB_PY} and $\bm{w}$ is decreasing.
\end{itemize}
\end{theo}

The key idea behind Theorem \ref{cor:BMSBp_conv} is that for a Markov chain of length variables, $\bm{v}$, with Beta-Binomial transitions $\Prob[v_{j+1} \in \cdot\mid v_j] = \psi_j(v_j,\cdot)$ as in \eqref{eq:BB_trans}, the parameters $\bm{\alpha}$ and $\bm{\beta}$ determine the marginal distributions $v_j \sim \Be(\alpha_j,\beta_j)$, while the parameter $N$ tunes the dependence among length variables. In fact, for $N = 0$, we recover independent length variables and, as $N \to \infty$, the length variables converge in distribution to a completely dependent sequence; see Lemma 
\ref{lem:BB_trans_conv}
in the Appendix. This appealing feature then allows us to recover any ISBp and CDSBp with the desired Beta marginals as limiting cases of a BMSBp.

\subsection{Weights' orderings and simulation} \label{subsec:BMSBP_2}
For this type of BMSBp, $N$ also modulates the stochastic ordering of the weights. 
When $\alpha_j = 1-\sigma$ and $\beta_j = \theta+j\sigma$, if $N = 0$ the weights $\bm{w}$ are in size-biased random order and satisfy 
$
\Esp\big[w_{j}\big] \geq \Esp\big[w_{j+1}\big],
$
for every $j \geq 1$. At the other extreme, as $N \to \infty$, we obtain decreasing weights. This suggests that for $N \in \{1,2,\ldots\}$ the ordering of the BMSBw, $\bm{w}$, is somewhere in between a size-biased random order and a decreasing rearrangement. It is worth pointing out that since the support of the Beta-Binomial transition is always the interval $[0,1]$, Proposition \ref{cor:MSB_dec_sup} assures that no BMSBw is decreasing a.s.~and the decreasing order is only achieved in the limit as $N \to \infty$.

Moreover, we can use Theorem \ref{theo:MSB_sb_PY} to establish, without having to impose even mild conditions, that the only proper BMSBw that is invariant under size-biased permutations are the Pitman-Yor weights. Indeed, by Theorem \ref{theo:BMSBp} we have that $w_j > 0$ a.s.,~for every $j \geq 1$, and that $\bm{P}$ is proper. For BMSBw, condition $(A)$ of Theorem \ref{theo:MSB_sb_PY}  is always met and so the result follows. 
Thus, the two notable weights' orderings, namely decreasing and size-biased permuted orderings, are only achieved in the limits $N = 0$ (if $\alpha_j = 1-\sigma$ and $\beta_j = \theta+j\sigma$) or $N \to \infty$. App.~\ref{sec:app:pr_w_dec}.
includes more details on the probability that weights are decreasing.

The clustering properties of stick-breaking random probability measures are available  only for a few noteworthy examples, including the Pitman--Yor process and the Dirichlet process. 
For any other stick-breaking process, with independent length variables, no analytical result is available. For the other extreme case of CDSBp there has been some progress in the literature \citep{DMP,HMW23}, though very little is known. 
Hence, it comes as no surprise 
that obtaining analytical results for MSBp, and its notable subclasses, is even more challenging. Nonetheless, Markovianity of the length variables has the merit of allowing straightforward 
sampling 
of $\tilde{w}_1$ and $K_n$ 
and it, thus, paves the way for performing also an empirical investigation. 
The key requirement is to be able to sample the length variables. To this aim we introduce a sequence of r.v.~$\bm{z} = (z_j)_{j=1}^{\infty}$ such that the joint density function of $\bm{v}_{[m]} = (v_j)_{j=1}^m$ and $\bm{z}_{[m]} = (z_j)_{j=1}^{m}$ is 
\[
\bm{p}(\bm{v}_{[m]},\bm{z}_{[m]}) = \prod_{j=1}^{m}\Be(v_{j}\mid \alpha'_j,\beta'_j)\Bin(z_j\mid N, \Upsilon_j(v_j)),
\]
for every $m \geq 1$, where $\alpha'_1 = \alpha_1$, $\beta'_1 = \beta_1$,  $\alpha'_j = \alpha_j+z_{j-1}$ and $\beta'_j = \beta_j +N -z_{j-1}$, for $j \geq 2$. In this way, integrating over $\bm{z}$ we 
recover the Markov process $\bm{v}$ with initial distribution $\Be(\alpha_1,\beta_1)$ and transition probability kernels $\bm{\psi}$ as in \eqref{eq:BB_trans}. 
In our framework, it allows us to easily draw samples of the length variables by first sampling $v_1 \sim \mathsf{Be}(\alpha'_1,\beta'_1)$ and sequentially for $j \geq 1$, $z_j\mid v_j \sim \mathsf{Bin}(N,\Upsilon_j(v_j))$ and $v_{j+1}\mid z_j \sim \mathsf{Be}(\alpha'_j,\beta'_j)$. As discussed below, the inclusion of $\bm{z}$ also allows us attain a posterior characterization of the length variables.

\begin{figure}
	\centering
	\includegraphics[width=1\textwidth]{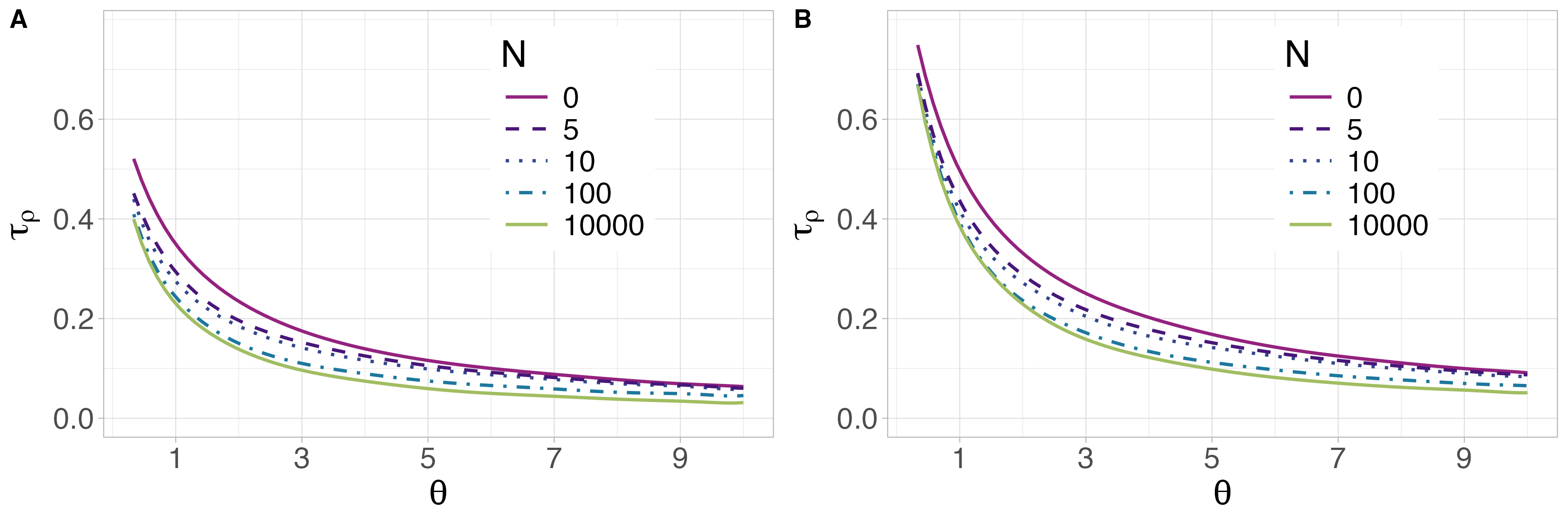}
	\begin{small} 
		\caption{\footnotesize{Tie probability $\tau_p$ as $\theta$ varies in $[0,10]$ for BMSBp for different 
		values of $N$ where $v_j \sim \Be(1-\sigma,\theta+j\sigma)$. Moreover, $\sigma = 0.3$ in panel A and $\sigma = 0$ in panel B. }}
		\label{fig:Tp_BMSB}
	\end{small}
\end{figure}

\begin{figure}
	\centering
	\includegraphics[width=0.9\textwidth]{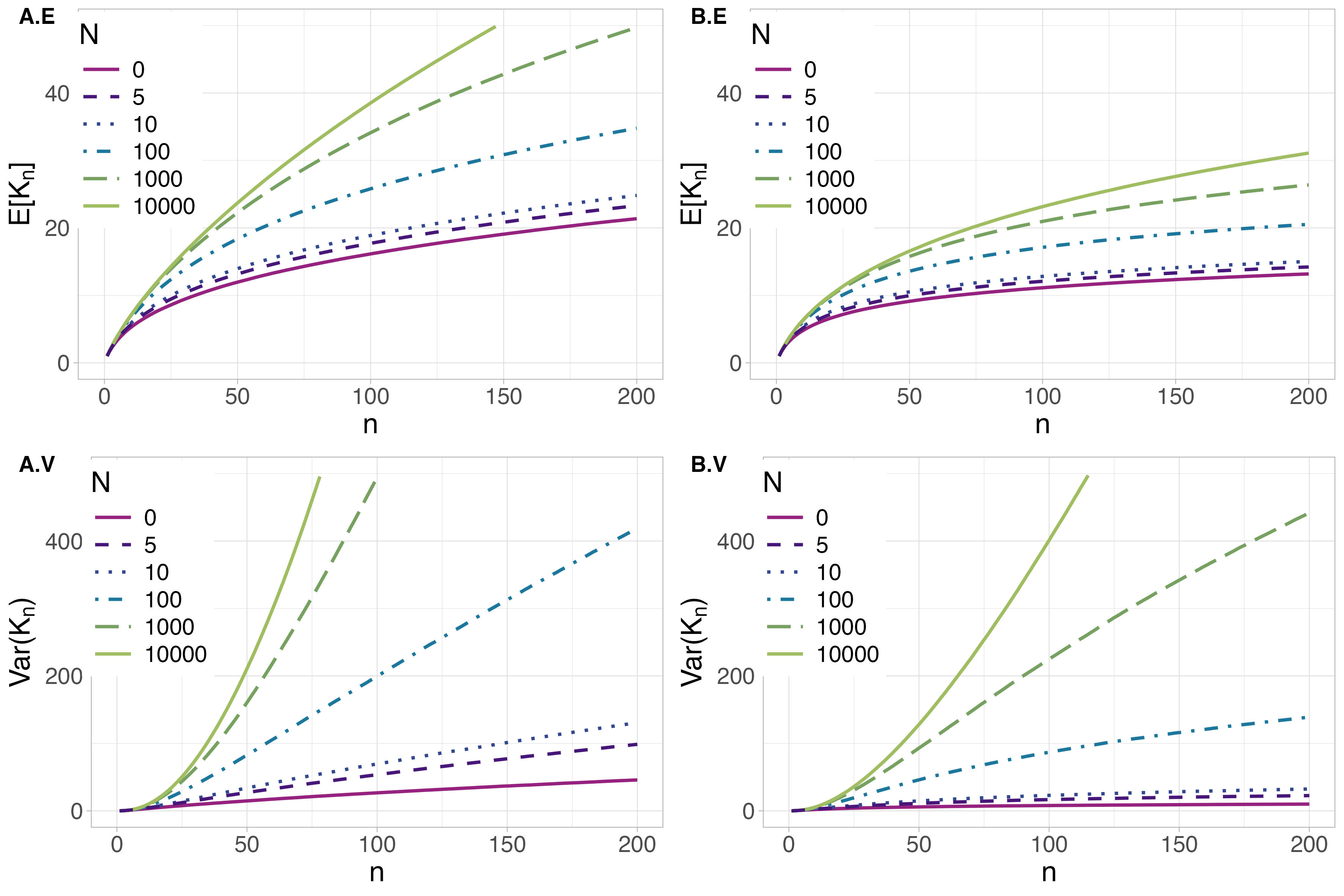}
	\begin{small} 
	\caption{\footnotesize{
		Plots of the mappings $n \mapsto \Esp[K_n]$ and $n \mapsto \Var(K_n)$ for BMSBp with distinct values of $N$ where $v_j \sim \Be(1-\sigma,\theta+j\sigma)$ with $\sigma = 0.3,\theta = 2$ (A) and $\sigma = 0,\theta = 3$ (B) . }}
		\label{fig:Kn_BMSB}
	\end{small}
\end{figure}

In Figure \ref{fig:Tp_BMSB} we display the tie probability $\tau_p$ for BMSBp with $v_j \sim \Be(1-\sigma,\theta+j\sigma)$ and $\sigma = 0.3$ (Figure A) and $\sigma = 0$ (Figure B), as $\theta$ and $N$ vary. In particular, for fixed $\theta$, smaller values of $N$ correspond to larger tie probabilities and, consequently, larger values of $\Var(\bm{P}(A))$. In contrast, Figure~\ref{fig:Kn_BMSB} displays the mappings $n \mapsto \Esp[K_n]$ and $n \mapsto \Var(K_n)$ for BMSBp with different choices of $N$, where $v_j \sim \Be(1-\sigma,\theta+j\sigma)$ with $\sigma = 0.3,\theta = 2$ (left column) and $\sigma = 0,\theta = 3$ (right column). Here we see an inverse effect, so larger values of $N$ are associated to larger values of $\Esp[K_n]$ and $\Var(K_n)$. This outcome was expected, as a larger number of clusters tends to form when the probability of any two observations being clustered together is low. Remarkably, changing the dependence of the length variables leads to subtle changes in $\tau_p$ and to drastically different behaviours of $\Esp[K_n]$ and $\Var(K_n)$. 

\noindent Another interesting insight follows by contrasting the BMSBp with parameters $N = 100$ and $\sigma = 0$ (dotted-dashed blue lines in right columns) against the Pitman-Yor model with parameter $\sigma = 0.3$ (solid purple line in left columns). In both models $n \mapsto \Esp[K_n]$, for $n \leq 200$, and $\tau_p$ behave somewhat similarly, with the strict BMSBp exhibiting smaller values of $\tau_p$ for smaller values of $\theta$. Regardless, $\Var(K_n)$ is notably higher for the BMSBp with $N = 100$ compared to Pitman-Yor model.

\subsection{Posterior inference} \label{subsec:BMSBP_3}

Now we focus on the fundamental tools for posterior inference, by tailoring the general results of Section \ref{sec:MSB} to BMSBps. In particular, we derive a posterior characterization of the length variables relying on an augmentation that incorporates $\bm{z}$, as discussed previously.

\begin{prop}\label{prop:BMSB_post}
Let $\bm{w} = (w_j)_{j=1}^{\infty}$ be a proper BMSBw with parameters $(N, \bm{\alpha},\bm{\beta})$, and $\bm{v} = (v_j)_{j=1}^{\infty}$ the  Markov process for the length variables. Consider the allocation variables $(d_1,\ldots,d_n)$, $d_i \mid \bm{w} \iid \sum_{j=1}^{\infty}w_j \delta_j$. Then, there exists a sequence $\bm{z} = (z_j)_{j=1}^{\infty}$ such that
\begin{align*}
\bm{p}(\bm{z}_{[m]}\mid \bm{v}_{[m]}, d_1,\ldots,d_n) \propto \prod_{j=1}^{m-1}\binom{N}{z_j}&\frac{[\Upsilon_j(v_j)v_{j+1}]^{z_j}[(1-\Upsilon_j(v_j))(1-v_{j+1})]^{N-z_j}}{(\alpha_{j+1})_{z_j}(\beta_{j+1})_{N-z_j}}\\
&\quad \quad \quad \quad \quad \quad\quad \quad \quad \quad \quad \times  \Bin(z_m\mid N,\Upsilon_m(v_m)),
\end{align*}
and 
\begin{equation*}
\bm{p}(\bm{v}_{[m]}\mid\bm{z}_{[m]}, d_1,\ldots,d_n)  \propto \prod_{j=1}^{m} v_j^{\alpha'_j+a_j-1}(1-v_j)^{\beta'_j+b_j-1}[\Upsilon_j(v_j)]^{z_j}[(1-\Upsilon_j(v_j))]^{N-z_j}.
\end{equation*}
with  $a_j, b_j$ as in \eqref{eq:all_var_cond}, $\alpha'_1 = \alpha_1$, $\beta'_1 = \beta_1$,  $\alpha'_j = \alpha_j+z_{j-1}$ and $\beta'_j = \beta_j +N -z_{j-1}$, for $j \geq 2$.
\end{prop}

The previous result allows the implementation of Bayesian nonparametric models with BMSBp priors. Indeed, the model augmentation through the sequence $\bm{z}$ makes it simple to adapt a conditional Gibbs sampling method. Specifically, at each iteration, we can first update $\bm{z}_{[m]}$ by sampling from $\bm{p}(\bm{z}_{[m]}\mid \bm{v}_{[m]}, d_1,\ldots,d_n)$. This is straightforward because the $z_j$ are conditionally independent, given $\bm{v}_{[m]}$, and their support is the finite set $\{0,\ldots,N\}$. Separately, we can update $\bm{v}_{[m]}$ by sampling from $\bm{p}(\bm{v}_{[m]}\mid\bm{z}_{[m]}, d_1,\ldots,d_n)$. Since the $v_j$ are also conditionally independent, given $\bm{z}_{[m]}$, the length variables are easily updated (e.g. via Adaptive Rejection Metropolis Sampling, ARMS). 
In particular, if $\bm{v}$ is stationary, i.e. $\alpha_j = \alpha$ and $\beta_j = \beta$ for $j \geq 1$, the full conditional of $\bm{v}_{[m]}$ simplifies to
\begin{equation*}
\bm{p}(\bm{v}_{[m]}\mid\bm{z}_{[m]}, d_1,\ldots,d_n)  = \prod_{j=1}^{m}\Be(v_{j}\mid \alpha'_j+z_j+a_j,\beta'_j+N-z_j+b_j),
\end{equation*}
where $\alpha'_j$ and $\beta'_j$ are as in Proposition \ref{prop:BMSB_post}. 

From an inferential perspective 
one would like to learn $N$ from the data, as this variable determines the dependence of length variables. To this aim, we can assign a prior distribution $\bm{p}(N)$ to $N$, so that Proposition \ref{prop:BMSB_post} holds, given $N$. Then, we need to compute the posterior distribution of $N$ and update it in the algorithm. Given $(\bm{v}_{[m]},\bm{z}_{[m]})$, with $m \geq \max_{i\leq n}d_i$, $N$ is conditionally independent of $(d_1,\ldots,d_n)$, yielding the posterior distribution
\begin{equation}\label{eq:BMSB_post_N}
\bm{p}(N \mid \bm{v}_{[m]},\bm{z}_{[m]}) \propto \bm{p}(N)\prod_{j=1}^{m}\Be(v_{j}\mid \alpha'_j,\beta'_j)\Bin(z_j\mid N, \Upsilon_j(v_j)).
\end{equation}
The methodological implications of its behaviour are clear-cut: if $\bm{p}(N \mid \bm{v}_{[m]},\bm{z}_{[m]})$ favours values close to $0$, a BMSBp model with nearly independent length variables (such as a Dirichlet or Pitman-Yor model) is preferred over one with strongly dependent length variables; alternatively, if $\bm{p}(N \mid \bm{v}_{[m]},\bm{z}_{[m]})$ is concentrated on larger values, a model with strongly dependent length variables is more suitable. In practice, how easy it is to sample from \eqref{eq:BMSB_post_N} depends on the specific prior assigned to $N$ and, in some cases, this task might be challenging. The noteworthy class we construct in the next section does not suffer from such an issue and allows learning of the dependence parameter in a straightforward way.

%--------------------------------------------------------------------------------------------------------------------------------
\section{Lazy Markov stick-breaking processes}\label{sec:LMSB}
%--------------------------------------------------------------------------------------------------------------------------------

We now introduce another notable subclass of MSBp, which we term 
\textit{lazy Markov stick-breaking processes}
due to the structure of the transition probability. A first advantage over the previously considered class of BMSBp is represented by the fact that the marginal distribution of the length variables is not restricted to the beta distribution.

\subsection{Construction, properness and support} \label{subsec:LMSBP_1}
Let $\bm{\pi} = (\pi_j)_{j=1}^{\infty}$ be a collection of diffuse probability measures over $([0,1],\B_{[0,1]})$. For simplicity, we assume throughout that the distribution functions $F_j$ of $\pi_j$ are continuous and strictly increasing in $[0,1]$. Consider Markovian length variables, $\bm{v}$, with initial distribution $\pi = \pi_1$ and transition probability kernels
\begin{equation}\label{eq:SS_trans}
	\psi_j(v_j,\cdot\,) = \Prob[v_{j+1} \in \cdot\mid v_j] = \rho\,\delta_{\Upsilon_j(v_j)} + (1-\rho)\pi_{j+1}, 
\end{equation}
with $\rho\in [0,1]$ and $\Upsilon_j$ as in \eqref{eq:Upsilon}. It is easy to check that $v_j \sim \pi_j$; see App.~\ref{app:prop:SS_trans_marg}. We call the transition probability kernels $\bm{\psi} = (\psi_j)_{j=1}^{\infty}$  in \eqref{eq:SS_trans} \emph{lazy transitions} with parameters $(\rho,\bm{\pi})$. The corresponding $\bm{w}$ and $\bm{P}$ are termed, respectively, \emph{Lazy Markov stick-breaking weights sequence} (LMSBw) with parameters $(\rho,\bm{\pi})$ and \emph{Lazy Markov stick-breaking process} (LMSBp) with parameters $(\rho,\bm{\pi},P_0)$.
 
\noindent Analogously to $N$ for the BMSBp, the parameter $\rho$ controls how close the LMSBp is to either an independent or a completely dependent regime.
	Values of $\rho$ close to zero yield a behaviour similar to independent priors such as Dirichlet or Pitman-Yor process, whereas values close to one push the LMSBp towards the CDSBp/Geometric regime.

The next result shows full support and, under suitable conditions, properness of LMSBp. 
 \begin{theo}\label{theo:LMSBp}
Let $\bm{P}$ be a LMSBp with parameters $(\rho,\bm{\pi},P_0)$ and $\bm{w}$ the associated LMSBw. Then, the following hold true
\begin{itemize}
\item[\emph{(i)}]
(a) If $\ \sum_{j=1}^{\infty}\int_{[0,1]}v \pi_j(\mrm{d}v) = \infty \ $, $\bm{P}$ is proper for $\rho < 1$.
(b) If $\ \bm{\Upsilon} = (\Upsilon_j)_{j=1}^{\infty} \ $ in \eqref{eq:Upsilon} satisfies the condition \eqref{eq:sum_1_Upsilon}, then $\bm{P}$ is proper.
\item[\emph{(ii)}] For every $j \geq 1$, one has $w_j > 0$ a.s.
\item[\emph{(iii)}] $\bm{P}$ has full support.
\end{itemize}
\end{theo}

Note that if $\bm{v}$ is a completely dependent Markov process with $v_1 \sim \pi_1$ and $v_{j+1} = \Upsilon_j(v_j)$, the condition in Theorem \ref{theo:LMSBp} (i.b) becomes $\sum_{j=1}^{\infty}v_j = \infty$. If this holds, we also get $\sum_{j=1}^{\infty}\!\Esp[v_j] \!=\! \sum_{j=1}^{\infty}\!\!\int_{[0,1]}\!v \pi_j(\mrm{d}v)\!=\! \infty$. In this sense, (i.b) on $\bm{\pi}$ is stronger than (i.a).

Theorem \ref{theo:LMSBp} allows one to establish properness for many instances of LMSBp in a straightforward way. First consider the stationary case: $\pi_j = \pi$ for every $j \geq 1$, $\Upsilon_j$ becomes the identity function and the transitions in \eqref{eq:SS_trans} reduce to 
\begin{equation}\label{eq:sSS_trans}
	\psi(v_j,\cdot\,) = \Prob[v_{j+1} \in \cdot\mid v_j] = \rho\,\delta_{v_j} + (1-\rho)\pi.
\end{equation}
Condition \eqref{eq:sum_1_Upsilon} trivially holds as well as (i.b) yielding properness. 
A further subclass satisfying (i.b) arises for $\pi_j = \mathrm{Be}(1-\sigma, \theta + j\sigma)$, as shown in Corollary \ref{cor:CDSB_PY}. In this instance, setting $\rho = 0$ recovers the Pitman--Yor process, while $\rho = 1$ yields a CDSBp. The next result shows that both subclasses can be approximated by non--degenerate LMSBp.

\begin{theo}\label{cor:LMSB_conv}
Let $\bm{P}$ and $\bm{P}^{(n)}$ be proper LMSBp
with parameters $(\rho,\bm{\pi},P_0)$  and $(\rho^{(n)},\bm{\pi}^{(n)},P_0^{(n)})$, respectively. Moreover, $\bm{w}$ and $\bm{w}^{(n)}$ are the associated LMSBw. If $\rho^{(n)} \to \rho$ in $[0,1]$, $P_0^{(n)} \wto P_0$ and  $\pi^{(n)}_j \wto \pi_j$, for $j \geq 1$, then $\bm{P}^{(n)} \dwto \bm{P}$ and $\bm{w}^{(n)} \dto \bm{w}$. In particular,
\begin{itemize}
\item[\emph{(i)}] If $\rho = 0$, $\bm{P}$ is an ISBp with parameters  $(\bm{\pi},P_0)$. Moreover, if $\pi_j = \Be(1-\sigma,\theta+j\sigma)$, for some $0 \leq \sigma <1$ and $\theta  > -\sigma$, $\bm{P}$ is a Pitman-Yor process and $\bm{w}$ is invariant under size-biased permutations.
\item[\emph{(ii)}] If  $\rho = 1$, $\bm{P}$ is CDSBp with parameters $(\bm{\pi},P_0)$. In particular, if $\pi_j = \Be(1-\sigma,\theta+j\sigma)$, for some $0 \leq \sigma <1$ and $\theta  > -\sigma$, $\bm{P}$ is as in Corollary \ref{cor:CDSB_PY}, and $\bm{w}$ is decreasing.
\end{itemize}
\end{theo}

The previous Theorem has interesting implications, both theoretical and methodological. To better understand these, for each $0 \leq \rho \leq 1$ and $0 \leq \sigma < 1$, let $\bm{P}^{(\rho,\sigma)}$ be a LMSBp with parameters $(\rho,\bm{\pi},P_0)$ where $\pi_j = \Be(1-\sigma,\theta+j\sigma)$ with $\theta > 0$ fixed. Then $\bm{P}^{(0,0)}$ is Dirichlet process, $\bm{P}^{(1,0)}$ is a Geometric process with length variable $v \sim \Be(1,\theta)$, $\bm{P}^{(0,\sigma)}$ is a Pitman-Yor process and $\bm{P}^{(1,\sigma)}$ is a CDSBp as in Corollary \ref{cor:CDSB_PY}. With this notation at hand, Theorem \ref{cor:LMSB_conv} implies that for any $(\rho_n,\sigma_n) \to (\rho,\sigma)$ in $[0,1]\times[0,1)$ we have that $\mathcal{L}\big(\bm{P}^{(\rho_n,\sigma_n)}\big) \to \mathcal{L}\big(\bm{P}^{(\rho,\sigma)}\big)$, w.r.t. the weak topology, with $\mathcal{L}(\bm{P})$ denoting the law of $\bm{P}$. Hence, the mapping $(\rho,\sigma) \mapsto \mathcal{L}\big(\bm{P}^{(\rho,\sigma)}\big)$ is continuous on $[0,1]\times[0,1)$. Since the spaces where $\rho$ and $\sigma$ take values are intervals, this allows us to continuously warp any of these processes into any other by simply tuning $\rho$ and $\sigma$. 
This showcases another theoretical advantage of LMSBp w.r.t. BMSBp, extreme cases such as Dirichlet, Pitman-Yor or Geometric processes, can be arbitrarily approximated by non-trivial LMSBp whereas this is not true for BMSBp. 

\subsection{Weights' orderings and simulation} \label{subsec:LMSBP_2}
Theorem \ref{cor:LMSB_conv} is also useful to understand the ordering of the LMSBw, $\bm{w}^{(\rho,\sigma)}$ associated to  $\bm{P}^{(\rho,\sigma)}$. Given $\bm{w}^{(0,\sigma)}$ is in size-biased random order for any value of $\sigma$ and $\bm{w}^{(1,\sigma)}$ is decreasingly ordered, this suggests that the ordering of $\bm{w}^{(\rho,\sigma)}$ is somewhere in between a size-biased permutation and a decreasing rearrangement. The larger $\rho$ is, the more likely it is that elements of $\bm{w}^{(\rho,\sigma)}$  are decreasingly ordered; see App.~\ref{sec:app:pr_w_dec}. In fact, for general LMSBw, the decreasing order can only be met when $\rho = 1$ (though not all LMSBw with $\rho = 1$ are decreasing). This is a consequence of Proposition \ref{cor:MSB_dec_sup}, Theorem \ref{theo:LMSBp}, and the fact that for $\rho < 1$ the support of $\psi_j(v,\cdot)$ and that of $\pi_j$ is $[0,1]$ (because $F_j$ is assumed to be continuous and strictly increasing). Furthermore, we can use Theorem \ref{theo:MSB_sb_PY} to establish, without having to impose even mild conditions, that the only proper LMSBw ($\rho < 1$) that is invariant under size-biased permutations corresponds to the Pitman-Yor weights. By Theorem \ref{theo:LMSBp} we have that $w_j > 0$ a.s.~for every $j \geq 1$. Moreover, the lazy transition \eqref{eq:SS_trans} satisfies condition $(C)$ of Theorem~\ref{theo:MSB_sb_PY} if $\rho < 1$. Indeed, the supports of $\pi_1$ and $\psi_1(v,\cdot)$, $v \in [0,1]$, coincide with $[0,1]$ so the first part of condition $(C)$ is met.  The second part is a consequence of the fact that the mappings $v \mapsto \Upsilon_2(v)$ and $u \mapsto \delta_u$ are both continuous, the latter w.r.t. the weak topology; see App.~\ref{app:lem:SS_trans_conv}. The result then follows.

Regarding the clustering properties, we can conduct a more straightforward empirical investigation compared to the analysis in Section~\ref{subsec:BMSBP_2} on the BMSBp.
         Here, when sampling the length variables 
	one either has $v_{j+1} = \Upsilon_j(v_j)$ 
	or samples $v_{j+1} \sim \pi_{j+1}$, independently of $v_j$. 
Overall, we see that allowing for non-trivial dependence among length variables has a relevant impact on the resulting random probability measure. 
More specifically, 
let $\tau_p^{(\rho)}$ be the tie probability of a LMSBp with marginal distributions of the length variables given by $v_j \sim \Be(1-\sigma,\theta+j\sigma)$. Since the CDSBw are decreasing and the marginals of the length variables are fixed we get $\tau_p^{(0)} = \Esp[v_1] = \Esp[w_1] \geq \Esp[\tilde{w}_1] = \tau_p^{(1)}$
where $w_1 = v_1 \sim \Be(1-\sigma,\theta+\sigma)$ is the first CDSBw, and $\tilde{w}_1$ the corresponding first size-biased weight. Thus, the ISBp will always have larger values of $\Var(\bm{P}(A))$, while larger values of $\Esp[K_n]$ and $\Var(K_n)$ are associated with the CDSBp. One can, then, rely on the specification of $\rho$ in the LMSBp to tune 
the dependence among 
length variables and meet the appropriate trade-off between the 
values of $\Var(\bm{P}(A))$ and 
of $\Esp[K_n]$ and $\Var(K_n)$. This is illustrated in Figures \ref{fig:Tp_LMSB} and \ref{fig:Kn_LMSB}, and also in Figures  \ref{fig:Tp_BMSB} and \ref{fig:Kn_BMSB} for the BMSBp counterpart. 
Importantly, LMSBp  and BMSBp can achieve a good balance between  $\Var(\bm{P}(A))$ and $\Var(K_n)$ without modifying the marginal distributions of the length variables. As discussed in Section \ref{subsec:2} this is extremely beneficial from a computational perspective, 
as the burden of sampling from certain prior and posterior distributions typically increases when $\Esp[v_j]$ decreases rapidly, which is exactly what happens when $\sigma \to 1$ in the Pitman-Yor model.

\begin{figure}
	\centering
	\includegraphics[width=1\textwidth]{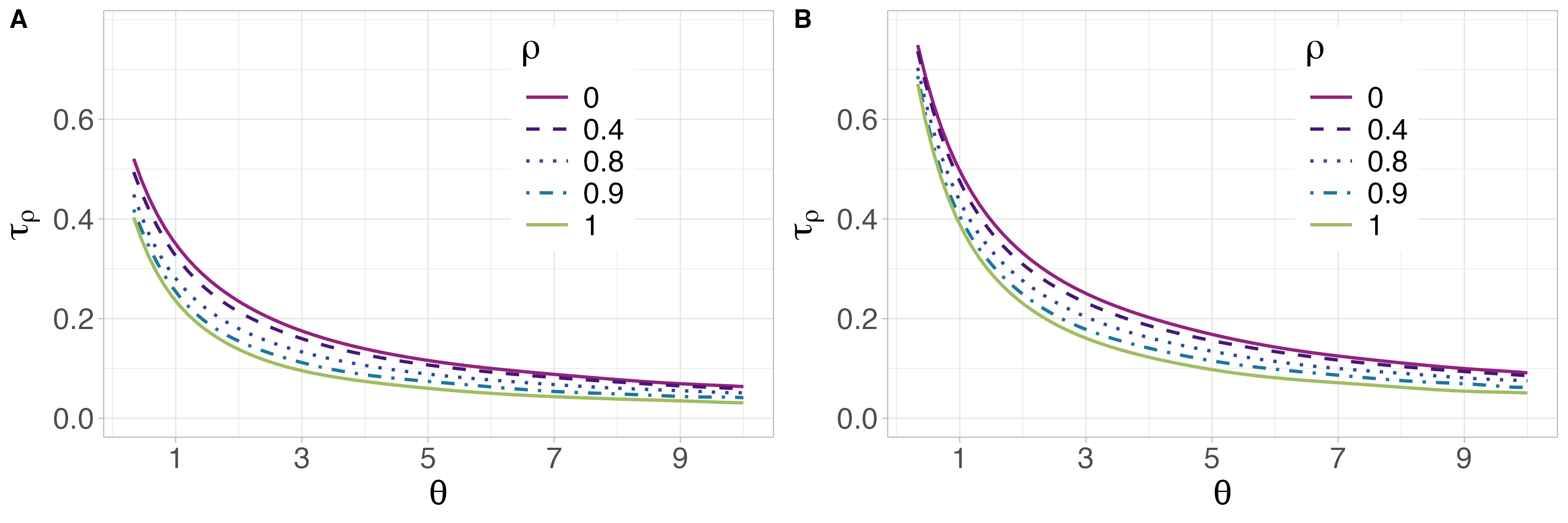}
	\begin{small} 
		\caption{\footnotesize{Tie probability $\tau_p$ as $\theta$ varies in $[0,10]$ for LMSBp with distinct values of $\rho$ where $v_j \sim \Be(1-\sigma,\theta+j\sigma)$ with $\sigma = 0.3$ (A) and $\sigma = 0$ (B) .}}
		\label{fig:Tp_LMSB}
	\end{small}
\end{figure}

\begin{figure}
	\centering
	\includegraphics[width=1\textwidth]{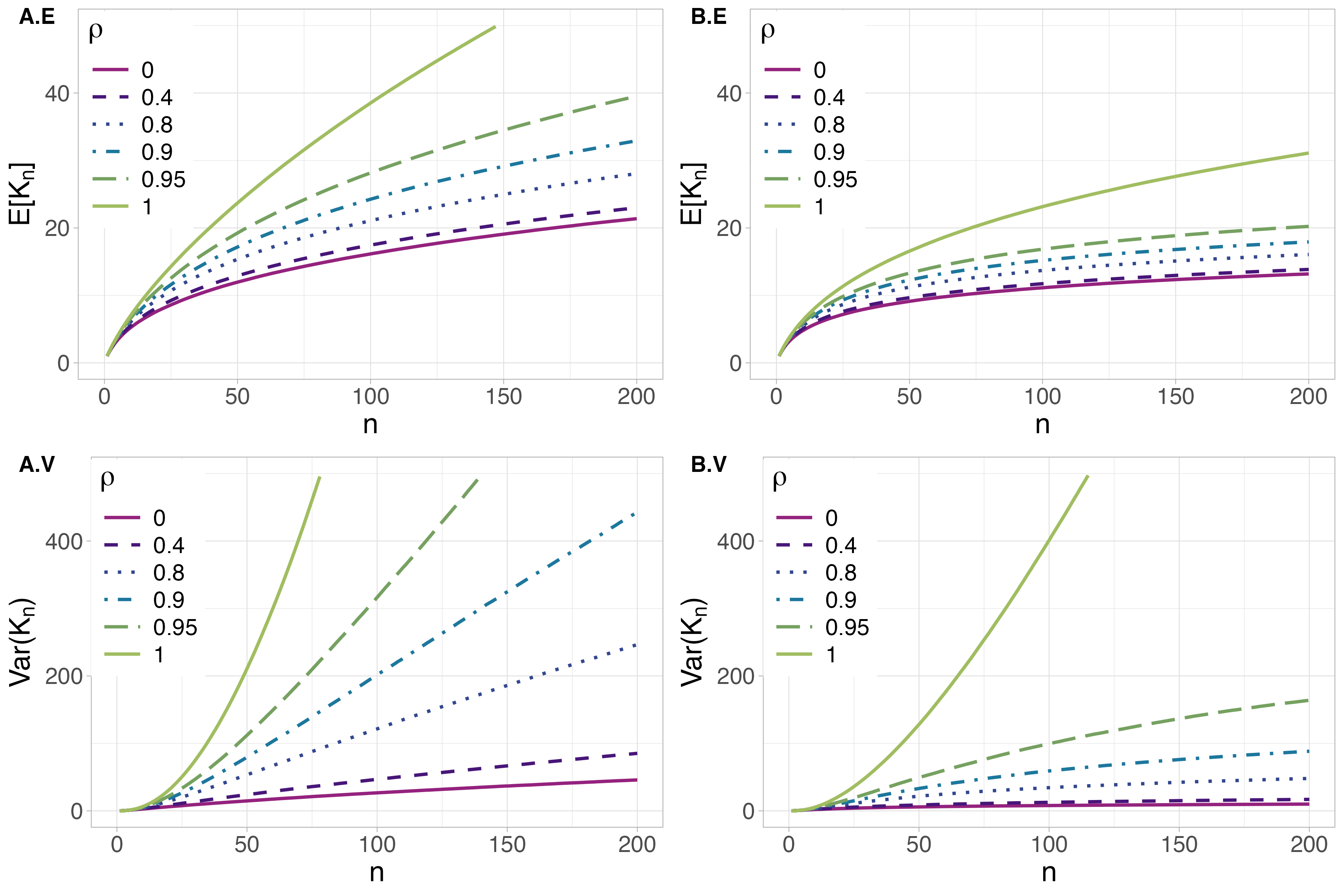}
	\begin{small} 
		\caption{\footnotesize{Displays of the mappings $n \mapsto \Esp[K_n]$ and $n \mapsto \Var(K_n)$ for LMSBp with distinct values of $\rho$ where $v_j \sim \Be(1-\sigma,\theta+j\sigma)$ with $\sigma = 0.3,\theta = 2$ (A) and $\sigma = 0,\theta = 3$ (B) .  }}
		\label{fig:Kn_LMSB}
	\end{small}
\end{figure}

\subsection{Posterior inference} \label{subsec:LMSBP_3}
Here we provide a posterior characterization of the length variables with Beta marginals, as this is the key ingredient of any sampling scheme. To this end, let $\alpha'_j = \alpha_j+a_j$, $\beta'_j = \beta_j + b_j$, with $a_j = |\{i \leq n: d_i = j\}|$ and $b_j = \sum_{l > j}a_l$. 

\begin{prop}\label{prop:LMSB_post}
Let $\bm{w} = (w_j)_{j=1}^{\infty}$ be a proper LMSBw with parameters $(\rho,\bm{\pi})$, where $\pi_j = \Be(\alpha_j,\beta_j)$, and $\bm{v} = (v_j)_{j=1}^{\infty}$ be the underlying Markov process of length variables of $\bm{w}$. Consider the allocation variables $(d_1,\ldots,d_n)$ satisfying $d_i \mid \bm{w} \iid \sum_{j=1}^{\infty}w_j \delta_j$. Then, 
\[
\bm{p}(v_1\mid \bm{v}_{-1},d_1,\ldots,d_n) = \rho_1 \,\delta_{\Upsilon_1^{-1}(v_2)}(v_1) + (1-\rho_1)\Be(v_1\mid \alpha'_1,\beta'_1).
\]
For $j > 1$, 
$
\bm{p}(v_j\mid \bm{v}_{-j},d_1,\ldots,d_n)=\delta_{\Upsilon_{j-1}(v_{j-1})}(v_j),
$
if $\Upsilon_{j-1}(v_{j-1}) = \Upsilon_j^{-1}(v_{j+1})$, whereas 
\[
\bm{p}(v_j\mid \bm{v}_{-j},d_1,\ldots,d_n)=
\rho_{j,1}\,\delta_{\Upsilon_{j-1}(v_{j-1})}(v_j)+ \rho_{j,2}\,\delta_{\Upsilon_{j}^{-1}(v_{j+1})}(v_j)+ \rho_{j,3}\,\Be(v_j\mid \alpha'_j,\beta'_j)
\]
if $\Upsilon_{j-1}(v_{j-1}) \neq \Upsilon_j^{-1}(v_{j+1})$, for some $\rho_{j,1},\rho_{j,2},\rho_{j,3},\rho_{1} \in [0,1]$
\end{prop}

The probabilities $\rho_{j,1},\rho_{j,2},\rho_{j,3},\rho_{1}$ are explicitly computed in App.~\ref{app:prop:LMSB_post}. Leveraging this result, we can adapt a conditional Gibbs sampler to implement LMSBp by updating one length variable at a time. Thanks to the Markov property this is straightforward: each length variable $v_j$ either is sampled independently from a Beta distribution with updated parameters or equals one between $\Upsilon_{j-1}(v_{j-1})$ and  $\Upsilon_{j}^{-1}(v_{j+1}) = F_{j}^{-1}F_{j+1}(v_{j+1})$. Although theoretically valid, this updating strategy can lead to a slow mixing of the Gibbs sampler. To see this define $t_1 = 1$ and inductively for $i \geq 2$, $t_i = \inf\{j > t_{i-1}: v_j \neq \Upsilon_{j-1}(v_{j-1})\}$, so that $\bm{v}^* = (v^*_1,v^*_2,\ldots)$ with $v^*_i = v_{t_i}$ are  the length variables that were sampled independently of the previous ones, and for $t_{i} < j < t_{i+1}$ we have $v_j = (\Upsilon_{j-1}\circ \cdots \circ \Upsilon_{t_i})(v^*_i) = F^{-1}_{j}(F_{t_i}(v^*_i))$.  Thus, for every $m \geq 1$, $\bm{v}_{[m]} = (v_j)_{j=1}^{m}$ is completely determined by $\bm{t}_{[r_m]} = (t_i)_{i=1}^{r_m}$ and $\bm{v}^{*}_{[r_m]} = (v^*_i)_{i=1}^{r_m}$ where $r_m = \max\{i \geq 1: t_i \leq m\}$, and vice versa. The drawback of updating $\bm{v}_{[m]}$ by sampling one entry at a time, as Proposition \ref{prop:LMSB_post} explains, is that $v^*_{t_i}$ only learns from the allocation variables through $a_{t_i}$ and $b_{t_i}$, while ideally it should incorporate the information provided by $a_{j}$ and $b_j$, for $t_i < j < t_{i+1}$. More worrisome, when $\rho \approx 1$ the values of $\bm{v}^*_{[r_m]}$ will rarely be updated. We overcome this issue by adding a step where $\bm{v}^*_{[r_m]}$ is re-sampled from $\bm{p}(\bm{v}^{*}_{[r_m]}\mid \bm{t}_{[r_m]}, d_1,\ldots,d_n)$. In order to compute this quantity first note that
$$
\bm{p}(\bm{v}^{*}_{[r_m]},\bm{t}_{[r_m]}) = \rho^{m-r_{m}}(1-\rho)^{r_{m}-1}\prod_{i=1}^{r_m}\pi_{t_i}(v^*_i).
$$
Multiplying this equation by \eqref{eq:all_var_cond}  we get
\begin{equation}\label{eq:LMSB_post_v*}
\bm{p}(\bm{v}^{*}_{[r_m]}\mid  \bm{t}_{[r_m]}, d_1,\ldots,d_n) \propto \pi_{t_i}(v^*_i) \prod_{i=1}^{r_m} (v^*_i)^{a_{t_i}}(1-v^*_i)^{b_{t_i}} \prod_{t_i < j < t_{i+1}} v_j^{a_j}(1-v_j)^{b_j},
\end{equation}
with $a_j$ and $b_j$ as in Proposition \ref{prop:LMSB_post} and $v_j = F^{-1}_{j}(F_{t_i}(v^*_i))$ for $t_{i} < j < t_{i+1}$. Hence we can update elements of $\bm{v}^*_{[r_m]}$ independently of each other, e.g.  using ARMS. In the particular case where the length variables are stationary and $\pi_j = \pi = \Be(\alpha,\beta)$, \eqref{eq:LMSB_post_v*} becomes
\begin{equation*}
\bm{p}(\bm{v}^{*}_{[r_m]}\mid \bm{t}_{[r_m]}, d_1,\ldots,d_n) = \prod_{i=1}^{r_m} \Be(v^*_i\mid \alpha+A_i,\beta+B_i),
\end{equation*}
where $A_i = \sum_{j=t_i}^{t_{i+1}-1} a_j$, $B_i = \sum_{j=t_i}^{t_{i+1}-1} b_j$. 

Regardless of whether the length variables are stationary or not, we can  assign a prior $h$ to $\rho$ so that the model itself can learn from the data whether nearly independent or strongly dependent length variables are more suitable and thus exploit Theorem \ref{cor:LMSB_conv} in practice. In this case, the aforementioned sampling scheme holds conditioning on $\rho$ and we also have $\rho$ conditionally independent of $(d_1,\ldots,d_n)$ given $\bm{v}_{[m]}$ for some $m \geq \max_{i \leq n} d_i$ with 
\[
h(\rho\mid \bm{v}_{[m]}) \propto h(\rho)\, \rho^{m-r_{m}}(1-\rho)^{r_{m}-1},
\]
where $r_m = \max\{i \geq 1: t_i \leq m\}$. In particular, if $h=\Be(a,b)$, we get $h(\rho\mid \bm{v}_{[m]}) = \Be\l(\rho\mi a + m-r_m, b+r_m-1\r)$. Clearly, the larger $r_m$ is, the more mass $h(\rho\mid \bm{v}_{[m]})$ assigns to values close to zero, suggesting less dependent length variables are preferred. In contrast, for smaller values of $r_m$, $h(\rho\mid \bm{v}_{[m]})$ assigns more mass to values closer to $1$, meaning that a LMSBp model with more dependent length variables suits the data better. 
See also \cite{NMJLM12} and  \cite{GS2011} for other, alternative, uses of the Beta-Binomial and lazy transitions in Bayesian nonparametric modeling. 

%--------------------------------------------------------------------------------------------------------------------------------
\section{A simulation study for mixture models}\label{sec:illustration}
%--------------------------------------------------------------------------------------------------------------------------------

The most popular application of discrete nonparametric priors is in the context of mixture models. This line of research dates back to \cite{Lo84}, where mixtures of Dirichlet processes were introduced. There is now a huge literature concerning the proposal and study of alternatives to the Dirichlet process as mixing measures, their asymptotic properties and computational implementation. See e.g. \cite{HHMW10,MQJH,GvdV17} and references therein.

In mixture models, the observations  $\bm{y}=(y_i)_{i=1}^n$ take values in a complete and separable metric space
$\mathbb{Y}$ with Borel $\sigma$-algebra $\B_\mathbb{Y}$, and are assumed to be conditionally iid 
from a random probability density function 
	$
	\bm{Q}(y) = \int \varphi(y\mid x) \bm{P}(dx),
	$
where $\varphi(\,\cdot\,\mid x)$ is a probability density over $\mathbb{Y}$, for each fixed value $x$ in 
$\mathbb{X}$, and the mixing distribution $\bm{P}$ is a random probability measure over $(\mathbb{X},\B_\mathbb{X})$.  Assigning a prior distribution to $\bm{Q}$ can be done indirectly by assigning a prior distribution to $\bm{P}$. Interest lies in discrete random probability measures $\bm{P} = \sum_{j=1}^{\infty}w_j\delta_{\theta_j}$, since they allow for simultaneous probabilistic clustering in addition to density estimation, and Dirichlet, Pitman-Yor or Geometric processes are popular choices. More generally, one can endow $\bm{P}$ with a MSBp prior, in which case we refer to $\bm{Q}$ as a MSBp mixture.
Closed form posterior inference is impossible even in the simplest cases and one has to rely on suitable  Markov Chain Monte Carlo algorithms. Further details on the posterior sampling algorithms for MSBp mixtures are provided in App.~\ref{sec:app:illust_mix}.

We illustrate the behaviour of MSBp mixtures by means of a simulation study. We consider $100$ datasets $\bm{y}^{(1)},\ldots,\bm{y}^{(100)}$, each one consisting of $150$ observations that are iid 
from a mixture of eight Gaussian distributions, 
$$
\bm{Q}^* = \sum_{j=1}^{8}w_j \mathsf{N}(\mu_j,\sigma_j^2) \qquad \text{s.t.} \ \  w_j \propto 0.1(0.9)^{j-1},
$$
with means $(\mu_j)_{j=1}^{8} = (-13.1, -7.2, -5.5, -2.7, 2.2, 4.3, 8.9, 9.7)$,  and standard deviations $(\sigma_j)_{j=1}^{8} = ( 1, 0.5, 0.3,1.3, 0.5, 0.5, 0.9, 0.4)$. 
This configuration, which combines well-separated and partially overlapping components, provides a demanding testbed for evaluating how accurately the different priors recover the multimodal structure of the true density; see Figure~\ref{fig:dtv_minmax}.
We assess the performance of 
the mixture models $\bm{Q}_1$\ldots,$\bm{Q}_4$ based on four different MSBp: 
(1) a Dirichlet process (DP); (2) a LMSBp with $\rho \sim \mathsf{Unif}(0,1)$; (3) a BMSBp with random parameter $N$; (4) a Geometric process (GP) mixture. All priors share the same marginals for the length variables $v_j \sim \Be(1,\theta)$. We consider four options for $\theta$: $(\theta_1,\ldots,\theta_4) = (1.6,1.4,0.7,0.4)$. Each $\theta_j$, for $j=1, \ldots, 4$, is calibrated so that, for model $\bm{Q}_j$, the prior expected number of clusters satisfies $\Esp[K_{150}] \approx 8$, which is the true number of components. Each of the four models is then evaluated under all four $\theta$ values for a total of sixteen experimental settings. For all cases we assume a Gaussian kernel $\varphi( y \mid x_1,x_2) = \mathsf{N}(x_1,x_2^{-1})$ and a base measure $P_0(x_1,x_2) = \mathsf{N}(x_1 \mid \mu_0, (\lambda_0 x_2)^{-1}) \mathsf{Ga}(x_2 \mid a_0,b_0) $ with $\mu_0 = n^{-1}\sum_{i=1}^{n} y_i^{(l)}$, $\lambda_0 = 1/100$,  and $a_0 =  b_0 = 0.5$.

For each sample we compare the posterior expectation $\widehat{\bm{Q}}_j^{(l)} = \Esp[\bm{Q}_j\mid \bm{y}^{(l)}]$ with the true distribution $\bm{Q}^*$  in terms of the total variation distance:
\[
d_{TV}(\widehat{\bm{Q}}_j^{(l)},\bm{Q}^*) = \sup_{A \in \B_{\R}} |\widehat{\bm{Q}}_j^{(l)}(A)-\bm{Q}^*(A)| = \frac{1}{2} \int_\R |\widehat{\bm{Q}}_j^{(l)}(y)-\bm{Q}^*(y)| dy,
\]
$l \in [100]$,  $j \in [4]$. Here, $\widehat{\bm{Q}}_j^{(l)}$ and $\bm{Q}^*$ denote both the probability measure and the density function. Figure \ref{fig:dtv_minmax} illustrates two runs where $d_{TV}(\widehat{\bm{Q}}_j^{(l)},\bm{Q}^*)$ was small (A) and large (B). 
 
In Table \ref{tab:dtv} we report the mean (and standard deviation) of $\{d_{TV}(\widehat{\bm{Q}}_j^{(l)},\bm{Q}^*)\}_{l=1}^{100}$ for 
$j=1, \ldots, 4$ and the four distinct values of $\theta$. Figure \ref{fig:dtv_box} displays the corresponding boxplots. We observe that, for each fixed value of $\theta$, the BMSBp achieves smaller values of the distance, closely followed by the LMSBp. This pattern persists even when $\theta$ was chosen to center a different model on the true number of mixture components. For smaller values of $\theta$, the GP shows smaller distances than the DP, whereas the opposite holds for larger values. 
A comparison of performance across models that are centered a priori on the true number of components shows that LMSBp outperforms the other alternatives, yielding the smallest average TV distance from the truth. The remaining prior specifications can be ranked, still in terms of the TV distance, as: BMSBp, DP and GP. See the bold numbers along the antidiagonal of Table \ref{tab:dtv} and the black boxplots in Figure \ref{fig:dtv_box}.

\begin{center}
\begin{table*}
\centering
\caption{\footnotesize{Average distance (and standard deviation) between the estimated and true densities for all MSBp mixtures. Bold text highlights the model for which $\theta$ was chosen so that $\Esp[K_{n}]= 8$}}
\label{tab:dtv}
\begin{tabular}{@{}lrrrc@{}}
\hline
$\theta$ & \multicolumn{1}{c}{DP}
& \multicolumn{1}{c}{LMSBp} & \multicolumn{1}{c}{BMSBp}
& \multicolumn{1}{c}{GP}  \\
\hline
$0.4$    & 0.191  (0.030) & 0.171  (0.027) & 0.169  (0.025)  & \textbf{0.183  (0.025)}  \\
$0.7$    & 0.183  (0.029) & 0.173  (0.027) & \textbf{0.170  (0.028)}  & 0.182  (0.026)  \\
$1.4$    & 0.173  (0.025) & \textbf{0.168  (0.025)} & 0.166 (0.024)  & 0.180  (0.024)  \\
$1.6$    & \textbf{0.173  (0.027)} & 0.169  (0.026) & 0.167  (0.026)  & 0.180  (0.023)  \\
\hline
\end{tabular}
\end{table*}
\end{center}

\begin{figure}
	\centering
	\includegraphics[width=1\textwidth]{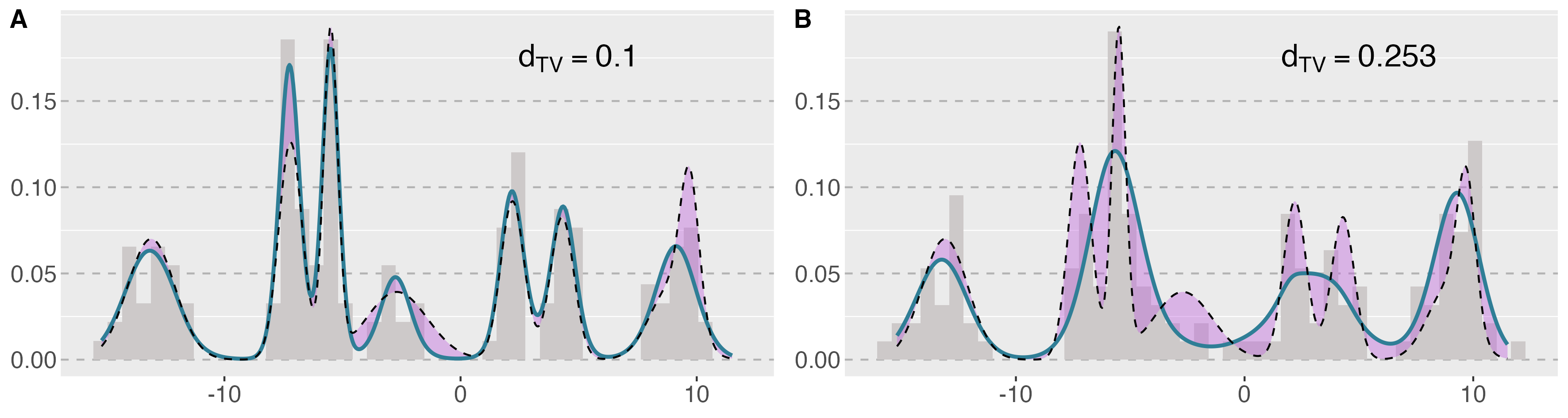}
	\begin{tiny} 
		\caption{\footnotesize{Estimated densities $\widehat{\bm{Q}}$ (solid blue lines), true density $\bm{Q}^*$ (dashed black lines) and histogram of the datasets (shaded gray area) that correspond to the two different runs where $d_{TV}(\widehat{\bm{Q}},\bm{Q}^{*})$ (shaded purple area) was small (A) and large (B).}}\label{fig:dtv_minmax}
	\end{tiny}
\end{figure}

\begin{figure}
	\centering
	\includegraphics[width=1\textwidth]{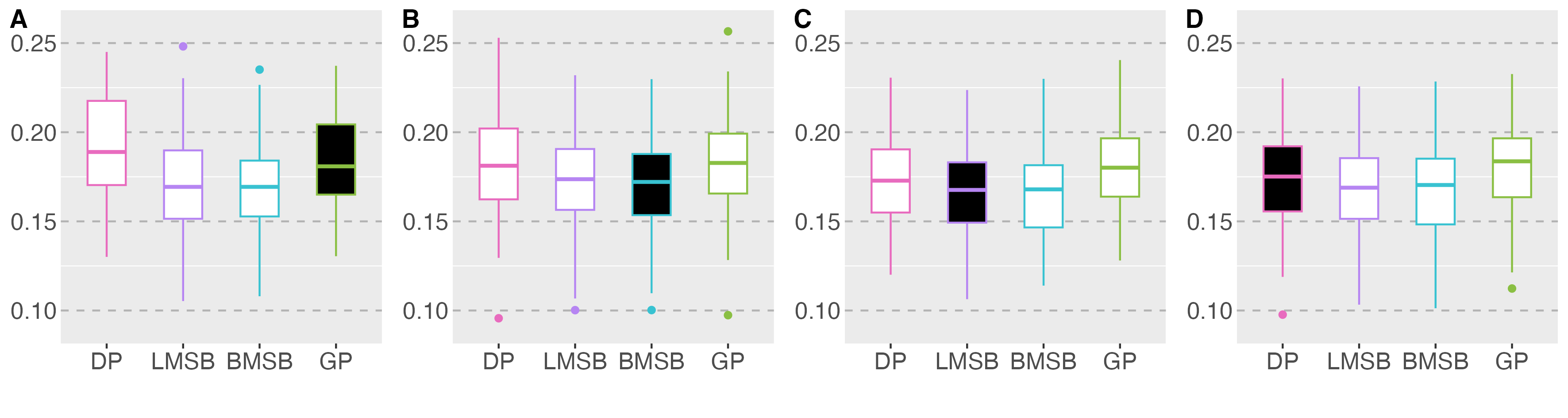}
	\begin{tiny} 
		\caption{\footnotesize{Boxplots of the TV distance between the estimated and true densities for all MSBp mixtures and the four values of $\theta = 0.4,0.7,1.4,1.6$ (panels A,B,C,D, respectively). The black boxplots highlight the model for which $\theta$ was chosen so that $\Esp[K_{n}]= 8$.}}\label{fig:dtv_box}
	\end{tiny}
\end{figure}

\begin{figure}
	\centering
	\includegraphics[width=1\textwidth]{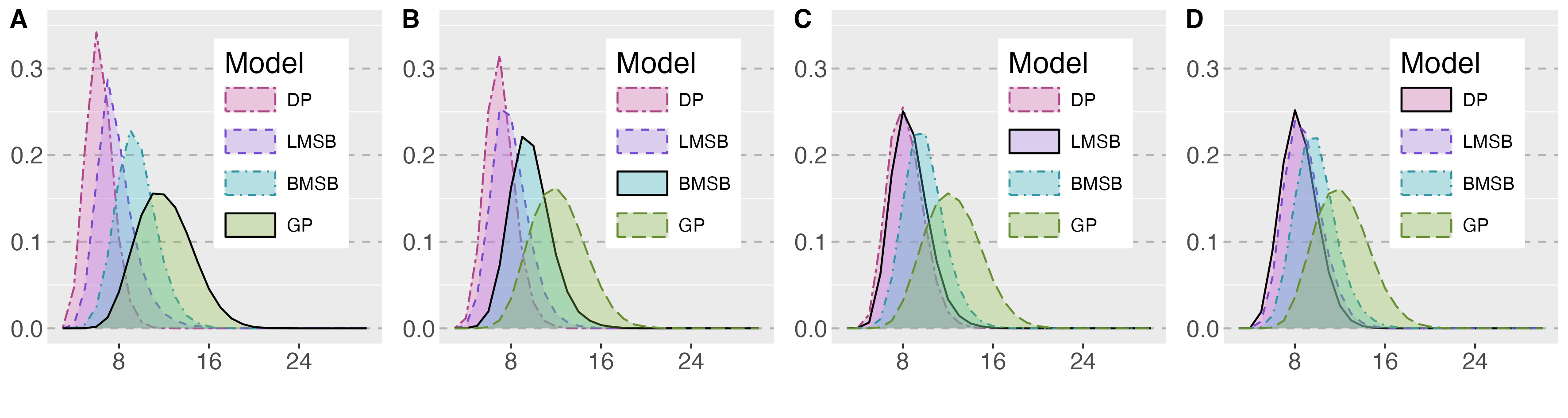}
	\begin{tiny} 
		\caption{\footnotesize{Averaged posterior distribution of $K_n$ over the 100 runs $\sum_{l=1}^{100} \Prob[K_n \in \cdot \mid \bm{y}^{(l)}] /100$, for all MSBp mixtures  and the four values of $\theta = 0.4,0.7,1.4,1.6$ (panels A,B,C,D, respectively). The solid black line highlights the model for which $\theta$ was chosen so that $\Esp[K_{n}]= 8$}}\label{fig:dtv_kn}
	\end{tiny}
\end{figure}

To complement the analysis, Figure \ref{fig:dtv_kn} displays the averaged posterior distribution of the number of observed components $K_n$, namely
$
\sum_{l=1}^{100} \Prob[K_n \in \cdot \mid \bm{y}^{(l)}] /100,
$
for $n = 150$ and for each of the models.
For the larger values of $\theta$ (panels C and D) the posterior distributions of $K_n$ corresponding to the DP and the LMSBp behave similarly, and place most mass around the true number of components. In contrast, for smaller values of $\theta$, the DP often underestimates the number of components by a small amount, and the LMSBp is the model that places more mass around the true number of components. In all cases, the BMSBp is also concentrated around 8, though it frequently overestimates the number of components. The GP notably overestimates the number of components regardless of the value of $\theta$. 
To sum up, Table \ref{tab:dtv}, together with Figures \ref{fig:dtv_box} and \ref{fig:dtv_kn},  suggest that the performance of MSBp mixtures is more stable and less sensitive to small changes in the marginal distributions of $v_j$, thereby highlighting the benefits of incorporating dependence among length variables as in LMSBp and BMSBp mixtures. In practice, where the number of true components is typically unknown, this characteristic can lead to a more robust posterior inference.

In App.~\ref{sec:app:illust_mix_2} we report an additional simulation study for multivariate MSBp mixtures, considering several choices of $\pi_j$, including $\Be(1-\sigma,\theta+j\sigma)$ with $\sigma > 0$. There we also estimate the induced clustering structures and examine the posterior behavior of key tuning parameters. Our findings provide strong evidence favoring MSBp models over simpler, less flexible alternatives. Future work will systematically investigate these methodological implications.

%--------------------------------------------------------------------------------------------------------------------------------
\section{Concluding remarks}\label{sec:conclusion}
%--------------------------------------------------------------------------------------------------------------------------------

In this paper we have focused on MSBp with marginal distributions of the length variables, $\pi_j$, that are absolutely continuous with respect to the Lebesgue measure. This assumption ensures that $\Upsilon_j$ in \eqref{eq:Upsilon} has appealing properties, such as monotonicity and continuity, which are crucial for the convergence results for the notable subclasses of BMSBp and LMSBp established in Theorems \ref{cor:BMSBp_conv} and \ref{cor:LMSB_conv}, respectively. Despite this, for LMSBp with stationary length variables, i.e. $\pi_j = \pi$, the functions $\Upsilon_j$ are all equal to the identity function, which means that the diffuseness assumption on $\pi$ can be easily relaxed so as to consider transition probability kernels as in \eqref{eq:sSS_trans} endowing $\pi$ with an atom at $1$. For instance, one could consider $\pi = (1-q)\Be(\alpha,\beta) + q\delta_1$, so that
\[
\psi(v,\cdot) = \rho\,\delta_{v} + (1-\rho)q\,\delta_1+ (1-\rho)(1-q)\Be(\alpha,\beta)
\]
Leveraging the general results in Section \ref{sec:MSB} it is then straightforward to prove that the corresponding MSBp will be proper, have full support, and that the number of atoms will be $t = \{j \geq 1: v_j = 1\} < \infty$ a.s. In fact, $\bm{P} = \sum_{j=1}^{t}w_j\delta_{\theta_j}$ where $t \sim \mathsf{Geo}(1-\rho)q$. Furthermore, by letting $q \to 0$ we can approximate (in distribution) LMSBp with infinitely many support points such as Dirichlet and Geometric processes. Although we do not further pursue this direction here, the theoretical results in Section \ref{sec:MSB} also cover this scenario. In particular, we highlight that the number of support points $t$ is trivially a stopping time with respect to the Markovian length variables. This could help to avoid the well-known challenges, first highlighted in \cite{Ric:Gre:97}, related to the implementation of a species sampling process with a random model dimension.

%--------------------------------------------------------------------------------------------------------------------------------
\section*{Acknowledgements}
%--------------------------------------------------------------------------------------------------------------------------------
 M.F.~Gil--Leyva and R.~Mena gratefully acknowledge the support of PAPIIT grants IA101124 and IT100524. A.~Lijoi and I.~Pr\"unster were partially supported by the European Union - Next Generation EU PRIN-PNRR (project P2022H5WZ9).

%--------------------------------------------------------------------------------------------------------------------------------
\newpage
\appendix
\numberwithin{equation}{section}
%--------------------------------------------------------------------------------------------------------------------------------

\section{Posterior algorithms and further distributional properties}\label{sec:app:illust}

\smallskip

This section includes details of the posterior sampling scheme, which are not included in the main manuscript, and an investigation of other relevant distributional properties of  
Markov stick-breaking processes.

\subsection{Posterior sampling algorithm}\label{sec:app:illust_mix}

As mentioned in Section 
 \ref{sec:illustration}
in the main paper, in mixture models, the observations  $\bm{y}=(y_i)_{i=1}^n$ taking values in a complete and separable metric space $\mathbb{Y}$, with Borel $\sigma$-algebra $\B_\mathbb{Y}$, 
 are assumed to be conditionally iid
 from a random probability density function
	\[
	\bm{Q}(y) = \int \varphi(y\mid x) \bm{P}(dx),
	\]
where $\varphi(\,\cdot\,\mid x)$ is a probability density kernel over $\mathbb{Y}$, for each fixed value $x$ in 
$\mathbb{X}$, and the mixing distribution $\bm{P}$ is a random probability measure on $(\mathbb{X},\B_\mathbb{X})$. When the mixing distribution is discrete, $\bm{P} = \sum_{j=1}^{\infty}w_j\delta_{\theta_j}$, we have
\begin{equation}\label{eq:mix_model}
	\bm{Q}(y) = \sum_{j=1}^{\infty}w_j \, \varphi(y\mid \theta_j). 
	\end{equation}
Closed form posterior inference is impossible even in the simplest cases and one has to rely on suitable  Markov Chain Monte Carlo algorithms. These are typically classified into marginal and conditional schemes. The former \comillas{marginalize out} the infinite-dimensional $\bm{P}$ and work with the resulting random partition distribution or predictive scheme. The latter requires to sample the random probability measure itself. For stick-breaking processes in general, and even with iid or independent weights, conditional sampling schemes are the algorithms of choice. This is then obviously true also for MSBp. Among the conditional approaches, the most popular are the retrospective \citep{PR08} and slice \citep{W07} samplers with a recent new addition represented by the ordered allocation sampler (OAS) by \cite{DG23}. Here we develop a posterior sampling scheme based on OAS, due to its desirable mixing properties. However, thanks to the availability to sample from the posterior distributions of the length variables one could also derive a retrospective or slice samplers proceeding along similar lines.

To explain how the OAS works first note that modelling $\bm{y}_{[n]} = (y_1,\ldots,y_n)$ as conditionally iid from \eqref{eq:mix_model} is equivalent to assume that $y_i\mid x_i \overset{\rm ind}{\sim} \varphi(\cdot \mid x_i)$, $i \geq 1$, where $\bm{x}_{[n]} =(x_1,\ldots,x_n)$ are iid sampled from $\bm{P}$, i.e. $x_i \mid \bm{P} \iid \bm{P}$, $i \geq 1$. With the introduction of the first $n$ elements $\bm{x}_{[n]}$ of the species sampling sequence $\bm{x}$, 
we can then define the distinct values $\tilde{\bm{x}}_{[K_n]} = (\tilde{x}_1,\ldots,\tilde{x}_{K_n})$ that $\bm{x}_{[n]}$ exhibits in order of appearance (so that $\tilde{x}_1 = x_1$, $\tilde{x}_2 = x_l$ iff $l = \inf\{i \geq 2: x_i \neq x_1\}$, and so on) as well as the ordered allocation variables $\bm{d}_{[n]} = (d_1,\ldots,d_n)$ given by $d_i = j$ if and only if $x_i = \tilde{x}_j$. As explained by \cite{P96} \citep[see also][]{DG23} it turns out that $\tilde{x}_j = \theta_{\varrho_j}$ for some random permutation of $\N$, $\bm{\varrho} = (\varrho_j)_{j=1}^{\infty}$, independent of $\bm{x}$ and given by 
\eqref{eq:sb_pick}.

In particular, $\tilde{x}_1,\tilde{x}_2,\ldots$ remain iid from the base measure $P_0$, and independent of $\bm{\alpha}$ and the weights, $\bm{w}$. Moreover, the long-run proportion of elements in $\bm{x} = (x_i)_{i=1}^{\infty}$ that coincide with $\tilde{x}_j$ is precisely
\[
\tilde{w}_j = \lim_{n \to \infty} \frac{|\{i \leq n: x_i = \tilde{x}_j\}|}{n} = w_{\varrho_j}.
\]
In other words $\tilde{\bm{x}}_{[K_n]}$ and $\tilde{\bm{w}}_{[K_n]} = (\tilde{w}_1,\ldots,\tilde{w}_{K_n})$ are the (observed) random component parameters and weights of \eqref{eq:mix_model} in the order in which they were discovered by $\bm{y}_{[n]}$. With this considerations taken into account, one can work with the augmented likelihood
\begin{equation*}\label{eq:like}
\bm{p}(\bm{y}_{[n]},\bm{d}_{[n]},\bm{\varrho}_{[K_n]}\mid \tilde{\bm{x}},\bm{w}) = \prod_{j=1}^{K_n} w_{\varrho_j}^{n_j} \prod_{i \in D_j}\varphi(y_i \mid \tilde{x}_{d_i})\Ind_{\mathcal{D}}\Ind_{\mathcal{A}},
\end{equation*}
where $D_j = \{i\leq n: d_i = j\}$, $n_j = |D_j|$, $\mathcal{D}$ is the event that $(D_1,\ldots,D_{K_n})$ forms an ordered partition of $[n] = \{1,\ldots,n\}$ with blocks in the least element order, i.e. $\min(D_j) \leq \min(D_{j+1})$, and $\mathcal{A}$ is the event that $\varrho_l \neq \varrho_j$ for $j \neq l$. The OAS proceeds by sampling from the full conditional distributions of $\bm{d}_{[n]},\bm{\varrho}_{[K_n]},\tilde{\bm{x}}$ and $\bm{w}$ \citep[see details in][]{DG23}. In particular, if the stick-breaking representation of $\bm{w}$, $w_j = v_j \prod_{i < j}(1-v_i)$, is available the weights, $\bm{w}$, can be updated via the length variables, $\bm{v}$. To do so it is enough to note that
\[
\prod_{j=1}^{K_n} w_{\varrho_j}^{n_j} = \prod_{j=1}^{\kappa} v_{j}^{a_j}(1-v_j)^{b_j},
\]
where $\kappa = \max(\bm{\varrho}_{[K_n]})$, $a_j = \sum_{j=1}^{K_n}n_l\Ind_{\{\varrho_l = j\}}$ is the number of data points associated to $j$th component, $w_j g(\cdot\mid \theta_j)$, of the mixture in the original order, and $b_j = \sum_{l>j}a_l$. This yields the full conditional distribution of $\bm{v}_{[m]} = (v_1,\ldots,v_m)$:
\[
\bm{p}(\bm{v}_{[m]}\mid \cdots) \propto \bm{p}(\bm{v}_{[m]})\prod_{j=1}^{\kappa} v_{j}^{a_j}(1-v_j)^{b_j},
\]
for $m \geq 1$. Now, in order to adapt the OAS to implement mixture models with BMSBp and LMSBp mixing priors, it is enough to sample from $\bm{p}(\bm{v}_{[m]}\mid \cdots)$, for $m \geq \kappa$, as explained in Propositions \ref{prop:BMSB_post} and  \ref{prop:LMSB_post}.

After running the OAS for $S$ iterations, once the burn-in period has elapsed, one obtains a sample of size $S$ from $\bm{p}(\bm{d}_{[n]},\bm{\varrho}_{[K_n]},\tilde{\bm{x}}_{[K_n]},\bm{w}_{[\kappa]}\mid \bm{y}_{[n]})$, where $K_n = \max(\bm{d}_{[n]})$, and $\kappa = \max(\bm{\varrho}_{[K_n]})$. Using this sample one can then estimate the density of the data at $y$ through
\[
\widehat{\bm{Q}}(y) = \frac{1}{S}\sum_{s=1}^{S}\l\{\sum_{j=1}^{K_n^{(s)}}  \tilde{w}_j^{(s)}\varphi\bigg(y \,\bigg|\, \tilde{x}_j^{(s)}\bigg)+ \l(1-\sum_{j=1}^{K_n^{(s)}}  \tilde{w}_j^{(s)}\r)\int \varphi(y\mid x)P_0(dx)\r\},
\]
where $\bullet^{(s)}$ is the sample $s$th sample of $\bullet$, and $\tilde{w}_j^{(s)} = w^{(s)}_{\varrho_j^{(s)}}$. Alternatively, we can estimate the density through
\[
\widehat{\bm{Q}}(y) = \frac{1}{S}\sum_{s=1}^{S}\l\{\frac{1}{n}\sum_{j=1}^{K_n^{(s)}}  n_j^{(s)}\varphi\bigg(y \,\bigg|\, \tilde{x}_j^{(s)}\bigg)\r\},
\]
where $n_j^{(s)} = |\{i \leq n: d^{(s)}_i = j\}|$. 
Using the samples of the ordered allocation variables, we can also estimate the clustering of data points according to the mixture component they were sampled from. To do so, one can approximate the estimator 
\[
\hat{\Pi} = \underset{\Pi^*}{\arg\min} \,\,\Esp[L(\Pi^{*},\Pi)\mid \bm{y}_{[n]}]
\]
where $\Pi = \Pi_{[n]} = \{D_1,\ldots,D_{K_n}\}$ is the partition of $[n]$ induced by $\bm{d}$ (equiv. $\bm{x}$) and $L$ is a suitable loss function such as \emph{Binder's} loss or the \emph{variation of information} \citep[cf.][]{WG18}. For a fixed partition $\Pi^*$, the objective function $\Esp[L(\Pi^{*},\Pi)\mid \bm{y}_{[n]} ]$ can be approximated using the MCMC samples of $\bm{d}$. However, even for a moderately large number of data points, it is practically impossible to minimize over the space of partitions due to its large cardinality. Alternatively, one can minimize over a smaller set such as the partitions visited by the MCMC.

\subsection{A simulation study for multivariate mixture models}\label{sec:app:illust_mix_2}

We now focus on a simulation study for multivariate mixture models. To showcase the performance of MSBp in multivariate mixture modeling, here we devise an experiment where we estimate the density and clusters of a bivariate dataset, $\bm{y}_{[n]}=(y_1, \ldots, y_n)$, containing $n = 395$ points sampled from a mixture of seven bivariate Gaussian distributions; the true density and the sampled data with their cluster memberships are displayed in Figure \ref{fig:paw_true}.
\begin{figure}
	\centering
	\includegraphics[width=1\textwidth]{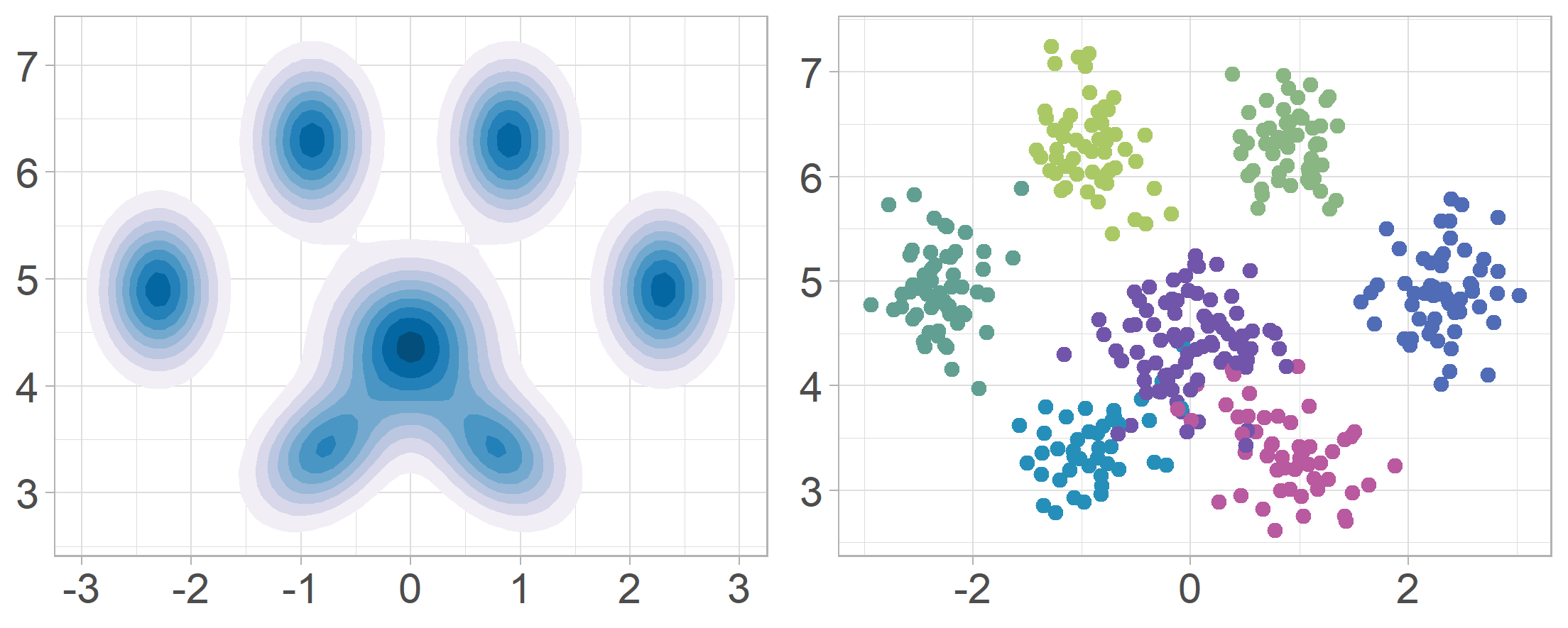}
	\begin{small} 
		\caption{True density (left panel) and sampled data with cluster memberships (right panel).  The density ranges from $0$ (lightest) to $0.024$ (darkest shade), and different colors indicate the different clusters.}
		\label{fig:paw_true}
	\end{small}
\end{figure}
We compare $12$ different MSBp mixtures with Gaussian kernel and the same base measure $P_0$. Specifically, in all cases $\varphi(y\mid \theta) = \mathsf{N}_2(y\mid \mu,\Sigma)$ and  $P_0(\mu,\Sigma) = \mathsf{N}_2(\mu \mid \mu_0,\lambda_0^{-1}\Sigma)\mathsf{IW}(\Sigma\mid \nu_0,\Psi_0)$,  where $\mathsf{IW}$ stands for an Inverse-Wishart distribution. The hyperparameters are set to $\mu_0 = n^{-1}\sum_{i=1}^n y_i$, $\lambda_0 = 10^{-1}$, $\nu_0 = 2$ and $\Psi_0$ equal to the $2$-dimensional identity matrix. The $12$ mixture models differ in terms of the assigned MSBw prior on $\bm{w}$. Namely, we consider:
\begin{enumerate}[(a)]
	\item Three ISBp with parameters $\bm{\pi} = (\pi_j)_{j=1}^{\infty}$ with: (1) $\pi_j = \Be(1,\theta)$; (2) $\pi_j = \Be(1-\sigma,\theta+j\sigma)$; (3) $\pi_j = \Be(1+\gamma/j,1)$, where $\sigma = 0.3$, $\theta = 1$ and $\gamma = 5$. Thus, (1) and (2) are, respectively, a Dirichlet and a Pitman-Yor process.
	\item Three LMSBp with parameters $\bm{\pi}$ as in (a), and in all cases prior on the tuning parameter given by $\rho \sim \Un(0,1)$.
	\item Three BMSBp with parameters $\bm{\alpha}$ and $\bm{\beta}$ given by: (1) $\alpha_j =1$ and $\beta_j = \theta$; (2) $\alpha_j = 1-\sigma$ and $\beta_j = \theta+j\sigma$; (3) $\alpha_j = 1+\gamma/j$ and $\beta_j = 1$, with $\sigma$, $\theta$ and $\gamma$ as in (a). In all cases, the prior on the tuning parameter is $N \sim \Un(\{1,\ldots,200\})$.
	\item Three CDSBp with parameters $\bm{\pi} = (\pi_j)_{j=1}^{\infty}$ as in (a). Thus, the first CDSBp corresponds to a Geometric process.
\end{enumerate}
In other words we considered stick-breaking models with three different types of marginal distributions, $\bm{\pi}$, of the length variables, i.e. $v_j \sim \pi_j$, and for each class of marginals we have considered four different dependence schemes between length variables. Importantly, we also assigned a prior to the tuning parameters of LMSBp and BMSBp, cases (b) and (c), in order to learn the weights' structure favoured by the data.

The estimated densities and clustering structure are shown, respectively, in Figures \ref{fig:paw_dens} and \ref{fig:paw_cl} for all $12$ models. In this simulation study, LMSBp and BMSBp mixtures recover the true density and clustering structures more accurately than their  limiting cases represented by ISBp and CDSBp mixtures. Recall that ISBp include the popular Dirichlet and Pitman-Yor mixtures as special cases and thus this represents a first evidence of the inferential gain in moving beyond these standard choices. In this experiment, the LMSBp and the BMSBp with $\Be(1,\theta)$ and $\Be(1-\sigma,\theta+j\sigma)$ marginals excel at density estimation. Nonetheless, the other models also perform well, except the CDSBp with $\Be(1-\sigma,\theta+j\sigma)$ and with $\Be(1+\gamma/j,1)$ marginals. In terms of clustering structure estimation, the LMSBp with $\Be(1,\theta)$ and $\Be(1-\sigma,\theta+j\sigma)$, as well all three BMSBp, are best at recovering the seven true clusters of data points. For the three ISBp mixing priors two clusters are merged and for the CDSBp models some clusters are merged and others are split. 
\begin{figure}
	\centering
	\includegraphics[width=1\textwidth]{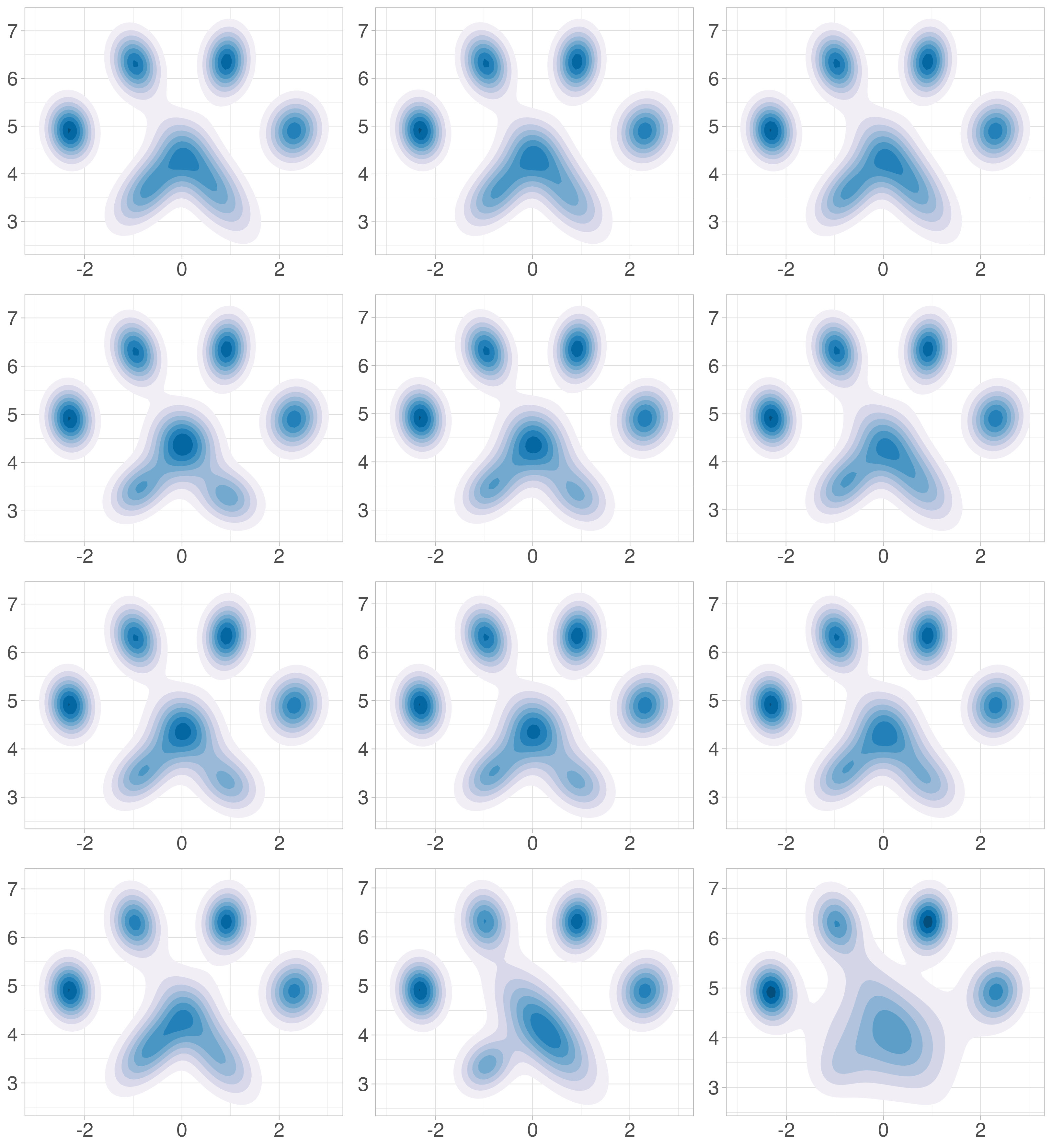}
	\begin{small} 
		\caption{Estimated densities with mixtures based on, respectively, ISBp (first row), LMSBp with random $\rho$ (second row), BMSBp with random $N$ (third row) and CDSBp (fourth row). The marginal distributions of the length variables are, respectively, $\pi_j = \Be(1,\theta)$ (first column), $\pi_j = \Be(1-\sigma,\theta+j\sigma)$ (second column) and $\pi_j = \Be(1+\gamma/j,1)$ (third column). All figures are in the same color scale and the density ranges from $0$ (lightest) to $0.024$ (darkest shade).}\label{fig:paw_dens}
	\end{small}
\end{figure}
\begin{figure}
	\centering
	\includegraphics[width=1\textwidth]{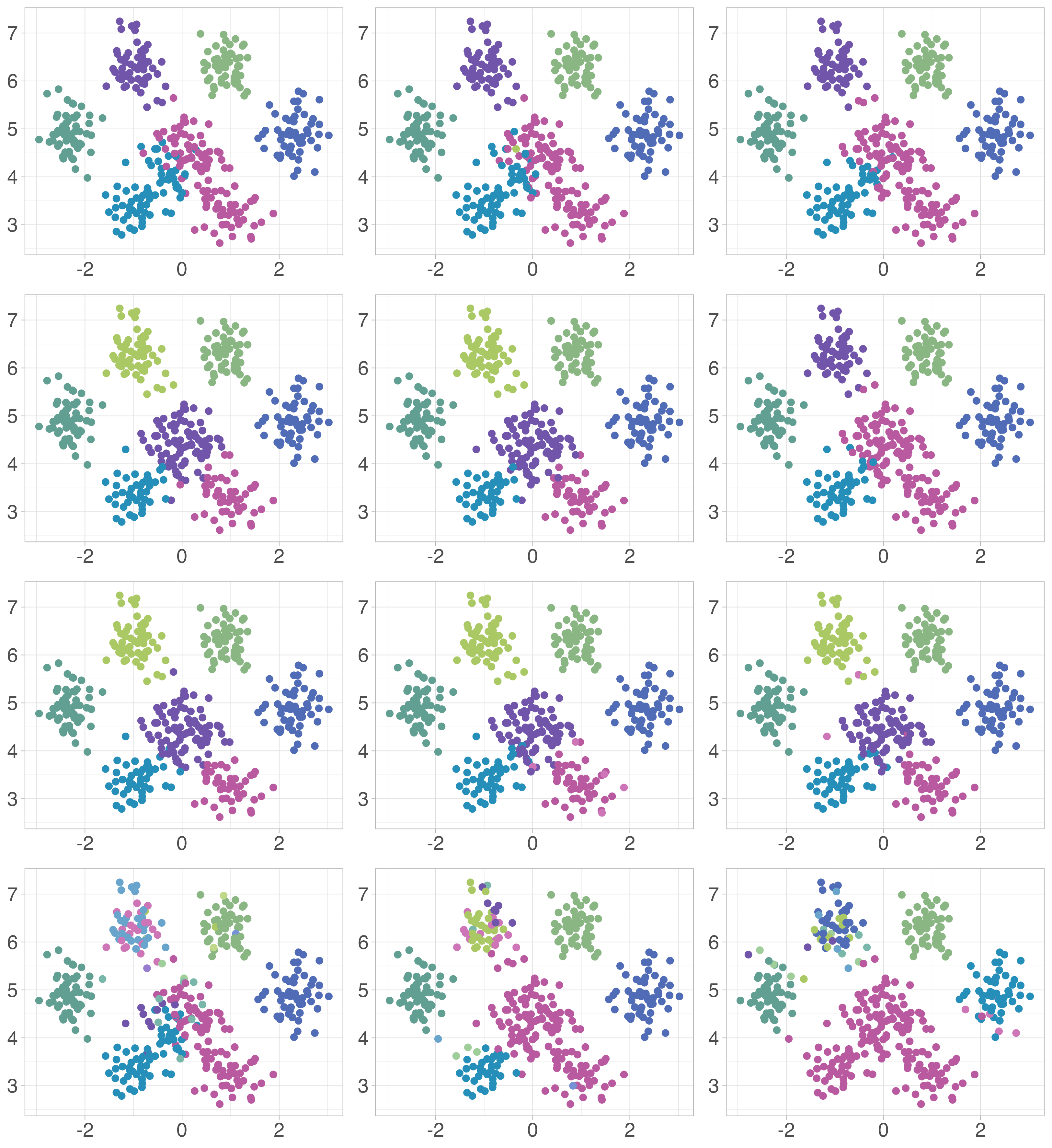}
	\begin{small} 
		\caption{Estimated clusters of data points with mixtures based on, respectively, ISBp (first row), LMSBp with random $\rho$ (second row), BMSBp with random $N$ (third row) and CDSBp (fourth row). The marginal distributions of the length variables are, respectively, $\pi_j = \Be(1,\theta)$ (first column), $\pi_j = \Be(1-\sigma,\theta+j\sigma)$ (second column) and $\pi_j = \Be(1+\gamma/j,1)$ (third column).}\label{fig:paw_cl}
	\end{small}
\end{figure}

We now analyze the posterior distribution of the number of mixture components that generated the data. Indeed, a key advantage of mixtures driven by discrete random probability measures is that the number of mixture components is not fixed and can be estimated. In our experiment the true value is $7$. Figure \ref{fig:paw_Kn} shows that the only models leading to a posterior mode of $7$ are the LMSBp mixtures with $\Be(1,\theta)$. Also the Pitman-Yor model (ISBp with $\Be(1-\sigma,\theta+j\sigma)$ marginals) assigns high probability to $7$ components. The other two ISBp mixtures and the LMSBp mixture with $\Be(1+\gamma/j,1)$ marginals, instead, lead to a posterior mode of $6$. In contrast, BMSBp mixtures tend to slightly overestimate the number of mixture components leading to modes in $8$ and $9$. As far as CDSBp mixtures are concerned, their estimates are completely off the mark: in all three models the posterior distribution of the number of mixture components only assigns positive mass to values larger than $9$.
\begin{figure}
	\centering
	\includegraphics[width=1\textwidth]{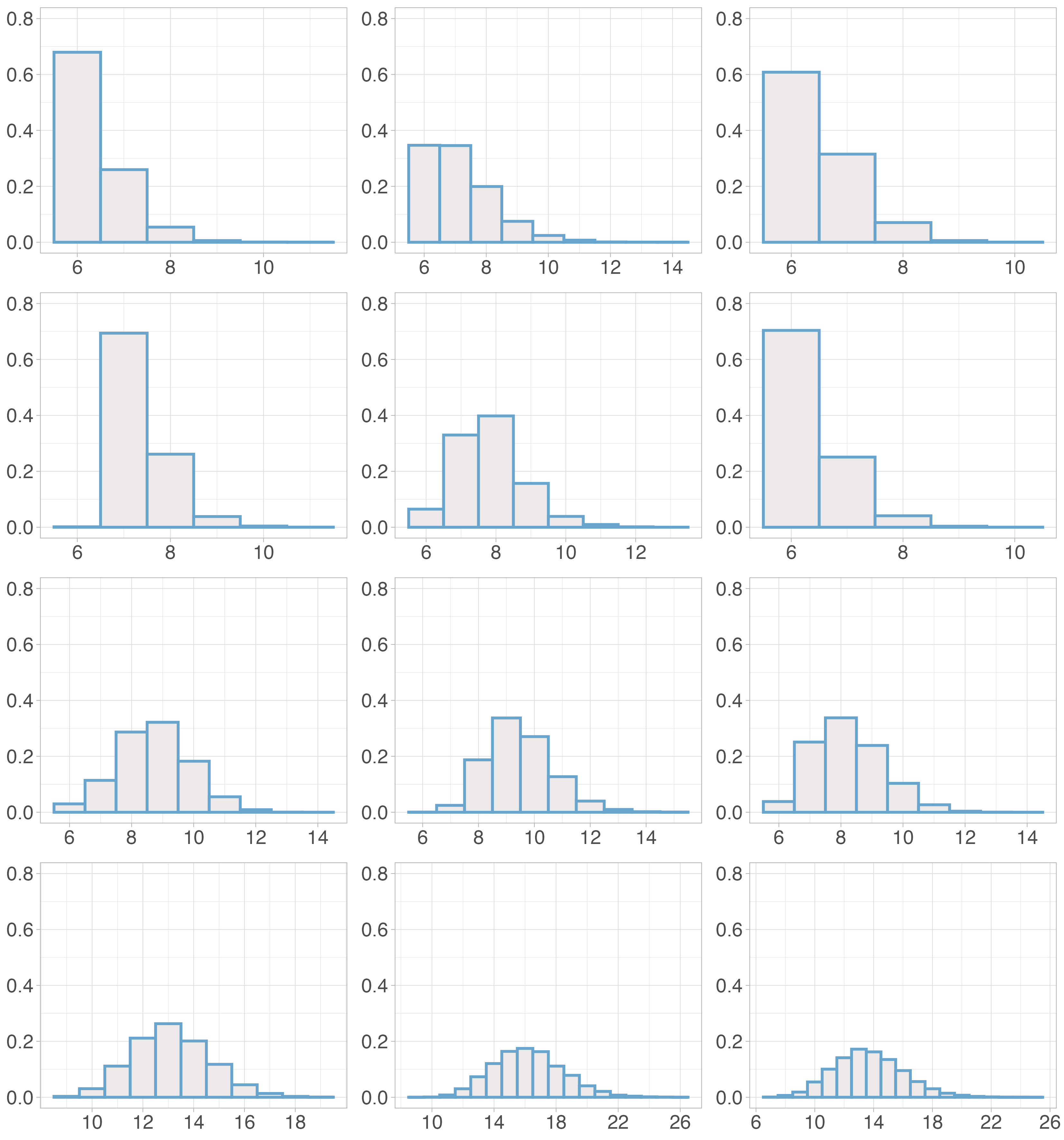}
	\begin{small} 
		\caption{Posterior distribution of the number of clusters corresponding to mixture models based on, respectively, ISBp (first row), LMSBp with random $\rho$ (second row), BMSBp with random $N$ (third row) and CDSBp (fourth row). The marginal distributions of the length variables are, respectively, $\pi_j = \Be(1,\theta)$ (first column), $\pi_j = \Be(1-\sigma,\theta+j\sigma)$ (second column) and $\pi_j = \Be(1+\gamma/j,1)$ (third column).}\label{fig:paw_Kn}
	\end{small}
\end{figure}

The results can be explained by the study in Sections 
\ref{subsec:BMSBP_2} 
and 
\ref{subsec:LMSBP_2} 
where we examined the effect that the dependence between length variables has on $\Var(\bm{P}(A))$, $\Esp[K_n]$ and $\Var(K_n)$. We observed that the ISBp models have the smaller values of $\Esp[K_n]$ and $\Var(K_n)$, which may have prevented them from recovering the true number of clusters in this experiment. This also affected slightly the density estimation in spite of the fact that the ISBp models enjoy the larger values of $\Var(\bm{P}(A))$. On the other extreme  the CDSBp feature considerably larger values of both $\Esp[K_n]$ and $\Var(K_n)$, which lead them to overestimate $K_n$ considerably. Furthermore, such a strong dependence between length variables decreases the degrees of freedom of  $\bm{v}$ to $1$ and this explains the deficiencies observed in the CDSBp mixtures. Now, for any LMSBp and BMSBp with random $\rho < 1$ and $N < \infty$ the degrees of freedom of $\bm{v}$ are infinite as it occurs with the ISBp. Moreover, by placing a prior on $\rho$ and $N$ the models are allowed to learn the dependence of the length variables and hence achieve a satisfactory compromise between $\Var(\bm{P}(A))$, $\Esp[K_n]$ and $\Var(K_n)$. 
Figure \ref{fig:paw_r} illustrates this, here we see that the posterior distribution of the dependence parameter, $\rho$, favours distinct values for distinct marginals choices.

\begin{figure}
	\centering
	\includegraphics[scale=0.7]{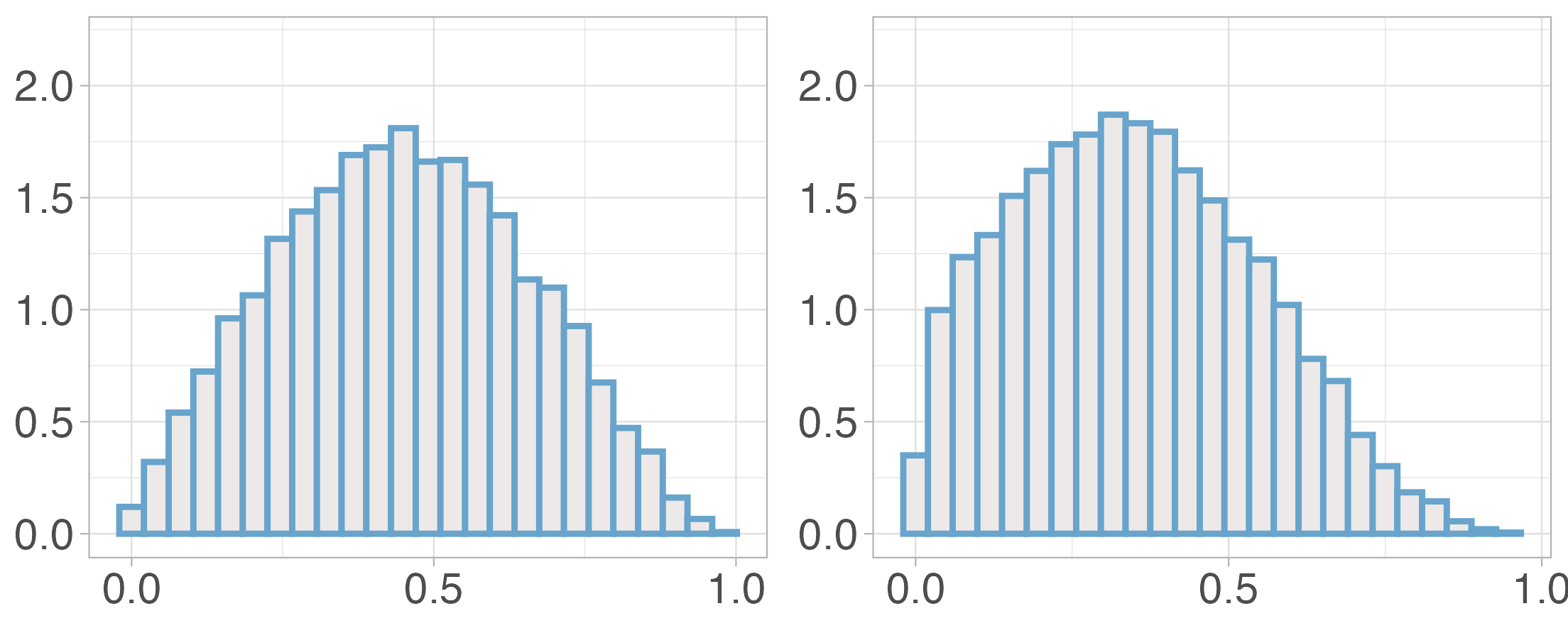}
	\begin{small} 
		\caption{Posterior distribution of $\rho$ for LMSBp with  $\pi_j = \Be(1,\theta)$ marginals (first column) and with $\pi_j = \Be(1-\sigma,\theta+j\sigma)$ marginals (second column).}\label{fig:paw_r}
	\end{small}
\end{figure}

\subsection{Moments of length variables and distribution of allocation variables}\label{sec:app:lv_moments}

In this section we focus on the computation of mixed moments of length variables such as
\[
\Esp\l[\prod_{j=1}^{\kappa} v_j^{a_j} (1-v_j)^{b_j}\r] .
\]
where $a_j$ and $b_j$ are positive constants for $j \geq 1$. This expressions appear in several quantities of interest. For example, using 
\eqref{eq:sb_pick}
and monotone convergence theorem, the tie probability turns out to be
\[
\tau_p = \Esp[\tilde{w}_1] = \Esp[\Esp[\tilde{w}_1\mid \bm{w}]] = \Esp\bigg[\sum_{j=1}^{\infty} w_j^2\bigg] = \sum_{j=1}^{\infty} \Esp[w_j^2] = \sum_{j=1}^{\infty} \Esp\l[v_j^{2} \prod_{l=1}^{j-1} (1-v_l)^{2}\r].
\]
Similarly, the evaluation of mixed moments of length variables is helpful when investigating the joint distribution of allocations variables. By taking expectations in 
\eqref{eq:all_var_cond} 
we obtain the joint mass probability function
\begin{equation}\label{eq:all_var_marg}
	\bm{p}(d_1,\ldots,d_n) =  \Esp\l[\prod_{j=1}^{\kappa} w_j^{a_j}\r] = \Esp\l[\prod_{j=1}^{\kappa} v_j^{a_j} (1-v_j)^{b_j}\r].
\end{equation}
where $\kappa = \max_{i \leq n}d_i$, $a_j = |\{i \leq n: d_i = j\}|$ and $b_j = \sum_{l > j}a_l$. While this joint distribution does not have a direct interpretation is the species sampling processes setting, it is
relevant in situations 
where the stick-breaking decomposition of weights arises.

For general MSBp we get \begin{equation*}
\Esp\l[\prod_{j=1}^{\kappa} v_j^{a_j} (1-v_j)^{b_j}\r] = \int_{[0,1]^{\kappa}}\prod_{j=1}^{\kappa}v_j^{a_j}(1-v_j)^{b_j}\,\bpsi_{\kappa-1}(v_{\kappa-1},\mrm{d}v_\kappa)\cdots \bpsi_1(v_1,\mrm{d}v_2)\,\pi(\mrm{d}v_1).
\end{equation*}
The following propositions provide these moments for the noteworthy subclasses of LMSBp and BMSBp with stationary length variables.
\begin{prop}\label{prop:sBMSB_post}
Let $\bm{v} = (v_j)_{j=1}^{\infty}$ be the length variables of a proper BMSBp with parameters $(N, \bm{\alpha},\bm{\beta})$ s.t. $\alpha_j = \alpha$ and $\beta_j = \beta$, for $j \geq 1$. Then, 
\[
\Esp\l[\prod_{j=1}^{\kappa} v_j^{a_j} (1-v_j)^{b_j}\r] = \sum_{z_1 = 1}^{N} \ldots \sum_{z_\kappa = 1}^N \l\{\prod_{j=1}^{\kappa}\binom{N}{z_j}\frac{(\alpha'_j)_{a_j + z_j}(\beta'_j)_{b_j+N-z_j}}{(\alpha'_j+\beta'_j)_{a_j+b_j+N}}\r\}
\]
where $\alpha'_1 = \alpha$, $\beta'_1 = \beta$,  $\alpha'_j = \alpha+z_{j-1}$ and $\beta'_j = \beta +N -z_{j-1}$, for $j \geq 2$. 
\end{prop}

\begin{proof}
We can introduce a sequence of r.v.~$\bm{z} = (z_j)_{j=1}^{\infty}$ such that the joint density function of $\bm{v}_{[m]} = (v_j)_{j=1}^m$ and $\bm{z}_{[m]} = (z_j)_{j=1}^{m}$ is 
\begin{equation*}
\bm{p}(\bm{v}_{[m]},\bm{z}_{[m]}) = \prod_{j=1}^{m}\Be(v_{j}\mid \alpha'_j,\beta'_j)\Bin(z_j\mid N, v_j),
\end{equation*}
for every $m \geq 1$, where $\alpha'_1 = \alpha$, $\beta'_1 = \beta$, $\alpha'_j = \alpha+z_{j-1}$ and $\beta'_j = \beta +N -z_{j-1}$, for $j \geq 2$. In this way, integrating over $\bm{z}$ we recover the Markov process $\bm{v}$ with initial distribution $\Be(\alpha,\beta)$ and transition probability kernel $\bpsi$ given by the suitably simplified version of 
\eqref{eq:BB_trans}. 
Moreover, integrating over $\bm{v}$ we find that $\bm{z}$ is also a Markov process with initial and stationary distribution
\[
\Prob[z_1 = z] = \binom{N}{z}\frac{(\alpha)_z(\beta)_{N-z}}{(\alpha+\beta)_N},
\]
and one step ahead transition probabilities
\[
\Prob[z_{j+1} = z\mid z_j] = \binom{N}{z}\frac{(\alpha + z_j)_z(\beta + N -z_j)_{N-z}}{(\alpha+\beta+N)_N}.
\]
By taking into consideration $\bm{z}$, we get that the elements of $\bm{v}$ are conditionally independent and Beta distributed. Indeed, 
\[
\bm{p}(\bm{v}_{[m]}\mid\bm{z}_{[m]}) = \Be(v_1\mid\alpha+z_1,\beta+N-z_1)\prod_{j=1}^{m-1}\Be(v_{j+1}\mid \alpha +z_j+z_{j+1},\beta+2N-z_j-z_{j+1}).
\]
This way, from \eqref{eq:all_var_marg} we can easily compute
\begin{align*}
\Esp\l[\prod_{j=1}^{\kappa} v_j^{a_j} (1-v_j)^{b_j}\r] & = \Esp\l[\Esp\l[\prod_{j=1}^{\kappa}v_j^{a_j}(1-v_j)^{b_j}\mi \bm{z}_{[\kappa]}\r]\r]\\
& = \Esp\l[\prod_{j=1}^{\kappa}\frac{(\alpha'_j+z_j)_{a_j}(\beta'_j+N-z_j)_{b_j}}{(\alpha'_j+\beta'_j + N)_{a_j+b_j}}\r]\\
& = \sum_{z_1 = 0}^{N}\cdots\sum_{z_{\kappa} = 0}^{N}\l\{\prod_{j=1}^{\kappa}\frac{(\alpha'_j+z_j)_{a_j}(\beta'_j+N-z_j)_{b_j}}{(\alpha'_j+\beta'_j + N)_{a_j+b_j}}\times\binom{N}{z_j}\frac{(\alpha'_j)_{z_j}(\beta'_j)_{N-z_j}}{(\alpha'_j+\beta'_j)_{N}}\r\}\\
& = \sum_{z_1 = 0}^{N}\cdots\sum_{z_{\kappa} = 0}^{N}\l\{\prod_{j=1}^{\kappa}\binom{N}{z_j}\frac{(\alpha'_j)_{a_j+z_j}(\beta'_j)_{b_j+N-z_j}}{(\alpha'_j+\beta'_j)_{a_j+b_j+N}}\r\}
\end{align*}
where $\alpha'_j$, $\beta'_j$ are as above.
\end{proof}

As for stationary length variables of a LMSBp we have the following result.
\begin{prop}\label{prop:sLMSB_post}
Let $\bm{v} = (v_j)_{j=1}^{\infty}$ be the length variables of a proper LMSBp with parameters $(\rho,\bm{\pi})$, where $\pi_j = \pi = \Be(\alpha,\beta)$.  Then, 
\[
\Esp\l[\prod_{j=1}^{\kappa} v_j^{a_j} (1-v_j)^{b_j}\r] = \sum_{(t_0,\ldots,t_r)}\rho^{\kappa-r}(1-\rho)^{r-1}\prod_{i=1}^{r}\dfrac{(\alpha)_{A_i} \, (\beta)_{B_i}}{(\alpha+\beta)_{A_i+B_i}}.
\]
where $A_i = \sum_{j=t_i}^{t_{i+1}-1} a_j$, $B_i = \sum_{j=t_i}^{t_{i+1}-1} b_j$, and the sum ranges over all $(t_1,\ldots,t_r)$ such that $t_1 = 1 < t_2 < \cdots < t_r \leq \kappa$ with $t_{r+1} = \kappa+1$. 
\end{prop}

\begin{proof}
Define $t_1 = 1$ and inductively for $i \geq 2$, $t_i = \inf\{j > t_{i-1}: v_j \neq v_{j-1}\}$, so that $\bm{v}^* = (v^*_1,v^*_2,\ldots)$ with $v^*_i = v_{t_i}$ are  the length variables that were sampled independently of the previous ones, and for $t_{i} < j < t_{i+1}$ we have $v_j = v^*_i$. Conditioning on $\bm{t} = (t_i)_{i=1}^{\infty}$, it is easy to see that  elements in $\bm{v}^*$ are iid from $\pi = \Be(\alpha,\beta)$. Thus 
\begin{align*}
\Esp\l[\prod_{j=1}^{\kappa} v_j^{a_j} (1-v_j)^{b_j}\r]  & = \Esp\l[\Esp\l[\prod_{j=1}^{\kappa}v_j^{a_j}(1-v_j)^{b_j}\mi \bm{t}\r]\r]\\
& = \Esp\l[\prod_{i=1}^{r_{\kappa}}\int_{[0,1]}(v)^{A_i}(1-v)^{B_i}\pi(\mrm{d}v)\r]\\
\end{align*}
where $A_i = \sum_{j=t_i}^{t_{i+1}-1} a_j$, $B_i = \sum_{j=t_i}^{t_{i+1}-1} b_j$, $r_{\kappa} = \max\{i \geq 1: t_i \leq \kappa\}$ and setting $t_{r_{\kappa}+1} = \kappa +1$. It also follows from the definition of $\bm{t}$ that
\[
\bm{p}(t_1,\ldots,t_{r_\kappa}) = \rho^{\kappa-r_{\kappa}}(1-\rho)^{r_{\kappa}-1},
\]
which yields
\begin{equation}\label{eq:all_var_LMSB1}
\Esp\l[\prod_{j=1}^{\kappa} v_j^{a_j} (1-v_j)^{b_j}\r]   = \sum_{(t_1,\ldots,t_{r_{\kappa}})}\rho^{\kappa-r_{\kappa}}(1-\rho)^{r_{\kappa}-1}\prod_{i=1}^{r_{\kappa}}\int_{[0,1]}(v)^{A_i}(1-v)^{B_i}\pi(\mrm{d}v).
\end{equation}
where the sum ranges over all sequences $(t_1,\ldots,t_{r_{\kappa}})$ such that $1 = t_1, < t_2 < \cdots < t_{r_\kappa} \leq \kappa$. In particular, as $\pi = \Be(\alpha,\beta)$, the integrals in \eqref{eq:all_var_LMSB1} simplify as
\[
\int_{[0,1]}(v)^{A_i}(1-v)^{B_i}\pi(\mrm{d}v) = \frac{\Gamma(\alpha+\beta)\Gamma(\alpha+A_i)\Gamma(\beta+B_i)}{\Gamma(\alpha+\beta+A_i+B_i)\Gamma(\alpha)\Gamma(\beta)} = \frac{(\alpha)_{A_i}(\beta)_{B_i}}{(\alpha+\beta)_{A_i+B_i}}.
\]
\end{proof}

\subsection{Probability of decreasing weights}\label{sec:app:pr_w_dec}

For MSBw, Theorems 
\ref{theo:MSB_conv}, 
\ref{cor:BMSBp_conv} 
and 
\ref{cor:LMSB_conv} 
suggest that, when the marginals of the length variables are $v_j \sim \Be(1-\sigma,\theta +j\sigma)$ and $\Cov(v_j,v_{j+1}) > 0$, the ordering of the weights $\bm{w}$ is somewhere in between the size-biased random order and the decreasing arrangement. This motivates the computation of 
$\Prob\l[w_{j+1} \leq w_j\r]$, for any $j$, as their values can be interpreted as a sort of measure of closeness 
of the actual order of MSBw to the size-biased random order  or the decreasing one. In this section we compute $\Prob\l[w_{j+1} \leq w_j\r]$ for general MSBw and see how the formula simplifies for BMSBw and LMSBw.

\begin{prop}\label{prop:dec_Ew}
Let $\bm{w}$ be a MSBw  with parameters $(\pi,\bmpsi)$.
\begin{itemize}
\item[\emph{(i)}] If the length variables, $\bm{v}$, are stationary, then $\Esp[w_{j+1}] \leq \Esp[w_j]$, for every $j \geq 1$.
\item[\emph{(ii)}] If the marginal distributions, $\pi_j$, given by 
\eqref{eq:vj_marg} 
satisfy $\pi_j(\{0\}) = 0$, then
\[
\Prob\l[w_{j+1} \leq w_j\r] = \Prob[t > j]\int_{(0,1)}\int_{(0,c(v)]}\bpsi_j(v,\mrm{d}u)\pi_j(\mrm{d}v)
+ \Prob[t \leq j],
\]
for $j \geq 1$, where $t$ is as in Proposition 
\ref{prop:MSB_supp_points} 
and $c(v) = 1\wedge v(1-v)^{-1}$ for $v \in (0,1)$.
\end{itemize}
\end{prop}

\begin{proof}
(i): Since the process $\bm{v}$ is stationary, $(v_1,\ldots,v_j)$ is equal in distribution to $(v_2,\ldots,v_{j+1})$ for every $j \geq 1$, which implies
\[
\Esp[w_{j}] = \Esp\l[v_j\prod_{i=1}^{j-1}(1-v_i)\r] = \Esp\l[v_{j+1}\prod_{i=2}^{j}(1-v_i)\r] \geq \Esp\l[v_{j+1}\prod_{i=1}^{j}(1-v_i)\r] = \Esp[w_{j+1}].
\]
(ii): By Theorem
\ref{theo:sMSB_1}
 (ii) we know $w_j = 0$ if and only if $j > t = \inf\{j \geq 1: v_j = 1\}$, a.s. Hence, under the event $\{t \leq j\}$, we know  $w_{j+1} = 0$ which yields $w_{j+1} \leq w_j$ holds. On the other hand  under the event $\{t > j\}$ we get $v_j < 1$, thus it follows from the stick-breaking decomposition that $w_{j+1} \leq w_j$ if and only if $v_{j+1} \leq v_j(1-v_j)^{-1}$. Putting everything together we obtain
\begin{align*}
\Prob[w_{j+1}\leq w_j] & = \Prob[t > j]\Prob[v_{j+1} \leq v_j(1-v_j)^{-1}\mid t > j] + \Prob[t \leq j] \\
& = \Prob[t > j]\int_{(0,1)}\int_{(0,c(v)]}\bpsi_j(v,\mrm{d}u)\psi_j(\mrm{d}v) + \Prob[t \leq j]
\end{align*}
where $c(v) = 1 \wedge v(1-v)^{-1}$.
\end{proof}

For BMSBw, the probability that consecutive weights are decreasingly ordered is derived in the following proposition.
\begin{prop}\label{prop:BMSBw_dec_w}
Let $\bm{w}$ be an BMSBw with parameters $(N,\bm{\alpha},\bm{\beta})$ where $\alpha_j = 1-\sigma$ and $\beta_j = \theta+j\sigma$, for every $j \geq 1$ and  $0 \leq \sigma < 1$ and $\theta > -\sigma$. Then
\[
\Prob[w_{j+1}\leq w_j] = \sum_{z= 0}^{N} \binom{N}{z}\Esp\big[\cl{I}_{c(v)}(\alpha_j+z,\beta_j+N-z)\Upsilon_j(v)^{z}(1-\Upsilon_j(v))^{N-z}\big]
\]
where $v \sim \Be(\alpha_j,\beta_j)$, $c$ is as in Proposition \ref{prop:dec_Ew} and, $\cl{I}_x(a,b)$ is the regularized incomplete  Beta function. In particular, if $\sigma = 0$
\[
\Prob[w_{j+1}\leq w_j] = \sum_{z= 0}^{N} \binom{N}{z}\frac{z!\theta}{(\theta+N-z)_{z+1}}\Esp\big[\cl{I}_{c(v)}(1+z,\theta+N-z)\big].
\]
\end{prop}

\begin{proof}
Being that $\pi_j = \Be(\alpha_j,\beta_j)$ is non-atomic for every $j \geq 1$, by Propositions 
\ref{prop:MSB_supp_points}
and \ref{prop:dec_Ew} we get
\[
\Prob[w_{j+1}\leq w_j] = \int_{(0,1)}\int_{(0,c(v)]}\bpsi_j(v,\mrm{d}u)\pi_j(\mrm{d}v) 
\]
where $\psi_j$ is the Beta-Binomial transition at time $j$ and  $c$ is as in Proposition \ref{prop:dec_Ew}. Hence,
\begin{align*}
\Prob[w_{j+1}\leq w_j] & = \int_0^1\int_0^{c(v)} \sum_{z=0}^{N} \Be(\mrm{d}u\mid \alpha_{j+1} +z,\beta_{j+1}+N-z)\Bin(z\mid N,\Upsilon_j(v)) \Be(\mrm{d}v\mid \alpha_j,\beta_j)\\
& = \sum_{z=0}^{N}\binom{N}{z}\int_{0}^{1} \cl{I}_{c(v)}(\alpha_{j+1}+z,\beta_{j+1}+N-z)\Upsilon_j(v)^z  (\Upsilon_j(v))^{N-z}\Be(\mrm{d}v\mid \alpha_j,\beta_j)\\
& = \sum_{z=0}^{N}\binom{N}{z}\Esp\l[\cl{I}_{c(v)}(\alpha_{j+1}+z,\beta_{j+1}+N-z)\Upsilon_j(v)^z (\Upsilon_j(v))^{N-z}\r]
\end{align*}
with $v \sim \Be(\alpha_j,\beta_j)$, $\cl{I}_x(a,b) = \int_{[0,x]}\Be(\mrm{d}v\mid a,b)$ and $\Upsilon_j(v) = \cl{I}_{\cl{I}_v(\alpha_j,\beta_j)}^{-1}(\alpha_{j+1},\beta_{j+1})$. In particular if $\alpha_j = \alpha$ and $\beta_j = \beta$, $\Upsilon_j$ is the identity function and we get
\begin{align*}
\Prob[w_{j+1}\leq w_j] & = \sum_{z=0}^{N}\binom{N}{z}\frac{\Gamma(\alpha+\beta)}{\Gamma(\alpha)\Gamma(\beta)}\int_{0}^{1} \cl{I}_{c(v)}(\alpha+z,\beta+N-z)v^{\alpha+z-1}(1-v)^{\beta+N-z-1}\mrm{d}v\\
& = \sum_{z=0}^{N}\binom{N}{z}\frac{(\alpha)_z(\beta)_{N-z}}{(\alpha+\beta)_N}\Esp\l[\cl{I}_{c(v)}(\alpha+z,\beta+N-z)\r],
\end{align*}
where $(x)_n = \prod_{i=0}^{n-1}(x+i)$. Furthermore, if $\alpha=1$ we obtain
\[
\Prob[w_{j+1}\leq w_j] = \sum_{z=0}^{N} \binom{N}{z} \frac{z! \beta}{(\beta+N-z)_{z+1}}\Esp\l[\cl{I}_{c(v)}(1+z,\beta+N-z)\r],
\]
\end{proof}

Next we compute $\Prob[w_{j+1}\leq w_j]$ for LMSBw.
\begin{prop}\label{prop:LMSBw_dec_w}
Let $\bm{w}$ be a LMSBw with parameters $(\rho,\bm{\pi})$, where $\pi_j = \Be(1-\sigma,\theta+j\sigma)$. Then
\[
\Prob[w_{j+1}\leq w_j] = \rho + (1-\rho)\Esp[\mathcal{I}_{c(v_j)}(1-\sigma,\theta+(j+1)\sigma)]
\]
where, as before, $\mathcal{I}_x(a,b)$ is the regularized incomplete Beta function, $v_j \sim \Be(1-\sigma,\theta+j\sigma)$ and $c$ is as in Proposition \ref{prop:dec_Ew}.
\end{prop}

\begin{proof}
From Proposition \ref{prop:dec_Ew} we know
\begin{equation*}\label{eq:LMSB_w_dec_proof}
\Prob[w_{j+1}\leq w_j] = \Prob[t > j]\int_{(0,1)}\int_{(0,c(v)]}\bpsi_j(v,\mrm{d}u)\pi_j(\mrm{d}v) + \Prob[t \leq j].
\end{equation*}
Since $\pi_j$ is diffuse for every $j \geq 1$ from Propostion 
\ref{prop:MSB_supp_points}
we get $t = \infty$ a.s. Thus $\Prob[t > j] = 1$, and we obtain
\begin{align*}
\Prob[w_{j+1}\leq w_j] & = \int_{(0,1)}\int_{(0,c(v)]}\bpsi_j(v,\mrm{d}u)\pi_j(\mrm{d}v)\\
& = \int_{(0,1)}\l\{\rho\,\delta_{\Upsilon_j(v)}((0,c(v)]) + (1-\rho) \int_{(0,c(v)]}\pi_{j+1}(\mrm{d}u)\r\}\pi_j(\mrm{d}v)\\
\end{align*}
with $c(v) = 1 \wedge v(1-v)^{-1}$. Now, in the proof of Corollary 
\ref{cor:CDSB_PY}
 we noted
\[
F_{j}(v) = \cl{I}_v(1-\sigma,\theta+j\sigma) \leq \cl{I}_v(1-\sigma,\theta+(j+1)\sigma) = F_{j+1}(v). 
\]
which yields $\Upsilon_j(v) =  F_{j+1}^{-1}(F_{j}(v)) \leq v \leq 1 \wedge v(1-v)^{-1} = c(v)$. That is $\delta_{\Upsilon_j(v)}((0,c(v)]) = 1$ for every $v \in (0,1)$ and we get
\[
\Prob[w_{j+1}\leq w_j] = \rho + (1-\rho)\int_{[0,1]}F_{j+1}(c(v))\, \pi_j(\mrm{d}v) = \rho + (1-\rho) \Esp[F_{j+1}((c(v))]
\]
with $v \sim \pi_j = \Be(1-\sigma,\theta+j\sigma)$.
\end{proof}

For the LMSBw in Proposition \ref{prop:LMSBw_dec_w}, $\Prob[w_{j+1} \leq w_j]$  is a linear combination between $1$ and a quantity completely determined by $\sigma,\theta$ and $j$, and $\rho$ measures how close to $1$ this probability is. In particular, $\Prob[w_{j+1} \leq w_j]$ increases as $\rho $ gets closer to $1$. By comparing this result to the BMSBw case (see Proposition \ref{prop:BMSBw_dec_w}) it is apparent that in the BMSB model  $\Prob[w_{j+1} \leq w_j]$ is a much more involved function of the tuning parameter $N$. In fact, even if $\sigma = 0$, it is quite hard to prove that the probability in question increases as $N$ does. This represents another theoretical advantage of LMSBp over BMSBp.

\medskip

\section{Proof of the results in Section \ref{sec:MSB}}\label{sec:app:results}

\smallskip

\subsection{Preliminaries}\label{sec:app:Markov}

In general for a fixed probability kernel $\bpsi$, it is possible that there exist no invariant probability measures, that there exists exactly one invariant probability measure or that there are infinitely many invariant probability measures. It is easy to see that if the probability measures $\lambda_1$ and $\lambda_2$ are both invariant for $\bpsi$, then for every $t \in [0,1]$, the mixture, $t\lambda_1+(1-t)\lambda_2$, is also invariant for $\bpsi$, hence the set of invariant probability measures with respect to $\bpsi$, denoted by $\cl{I}[\bpsi]$, is convex. We say that a probability measure $\lambda \in \cl{I}[\bpsi]$ is $\bpsi$-ergodic if it is extremal in $\cl{I}[\bpsi]$, that is, it can not be decomposed as a mixture of other probability measures $\lambda_i$ belonging to $\cl{I}[\bpsi]$. It can be shown that two  $\bpsi$-ergodic invariant probability measures, $\lambda$ and $\nu$ are either identical or mutually singular (meaning that the exist disjoint $A, B \in \B_{[0,1]}$ such that $A \cup B = [0,1]$, $\lambda(A) = 0$, and $\nu(B) = 0$). With this concepts in mind we can recall the ergodic theorem for Markov processes

\begin{theo}[Birkhoff; Ergodic theorem for stationary Markov processes]\label{theo:ergodic}
	Consider an stationary Markov process, $(v_j)_{j = 1}^{\infty}$, with and one-step transition probability kernel, $\bpsi$. Then
	\begin{equation*}\label{eq:ergodic0}
		\lim_{n \to \infty} \frac{1}{n} \sum_{j = 1}^nf(v_j) = \Esp[f(v_1)\mid \cl{I}],
	\end{equation*}
	where $\cl{I}$ is the invariant $\sigma$-algebra. In particular if the initial and stationary distribution, $\pi$, is $\bpsi$-ergodic, then
	\begin{equation}\label{eq:ergodic}
		\lim_{n \to \infty} \frac{1}{n} \sum_{j = 1}^n f(v_j) = \Esp[f(v_1)],
	\end{equation}
	a.s.~for every measurable function. Moreover, if $\cl{I}[\bpsi] = \{\pi\}$, \eqref{eq:ergodic} holds.
\end{theo}

A detailed review of Markov processes can be found in the work by \cite{Fell68,K02}.

\subsection{Proof of Theorem \ref{theo:sMSB_1}}\label{app:theo:sMSB_1}

\begin{proof}
First recall that $\sum_{j=1}^{\infty}w_j = 1$ a.s.~if and only if one of the conditions in \eqref{eq:sum1_MSB} 
holds.

(i): This is the same condition on the right side of 
\eqref{eq:sum1_MSB}.

(ii): By the Ergodic theorem for Markov processes (cf. Theorem \ref{theo:ergodic}) we know
\[
\lim_{n \to \infty} n^{-1}\sum_{j =1}^n v_j = \Esp[v_1\mid \cl{I}],
\]
where $\cl{I}$ is the invariant $\sigma$-algebra. If $\Prob[v_1 = 0] = \pi(\{0\}) = 0$, we get $v_1 > 0$ a.s., which yields $\Esp[v_1\mid \cl{I}] > 0$ a.s. Otherwise, if $\Prob[\Esp[v_1\mid \cl{I}] \leq 0] > 0$, from the definition of conditional expectation, we find that
\[
0 < \int_{\{\Esp[v_1\mid \cl{I}] \leq 0\}} v_1 \, \mrm{d}\Prob = \int_{\{\Esp[v_1\mid \cl{I}] \leq 0\}} \Esp[v_1\mid \cl{I}]\, \mrm{d}\Prob \leq 0,
\]
which is contradiction. Thus  $\Esp[v_1\mid \cl{I}] > 0$ a.s.~and we obtain the almost sure limit, $\lim_{n \to \infty} \sum_{j \leq n}v_j = \infty$, which is the  condition on the left side of 
\eqref{eq:sum1_MSB}.

(iii): First note that if $\pi = \delta_0$, then $\Prob[v_j = 0] = 1$, and trivially $\sum_{j \geq 1}w_j = 0$ a.s. Otherwise, if $\pi \neq \delta_0$, by the Ergodic theorem for Markov processes (cf. Theorem \ref{theo:ergodic}), and being that $\pi$ is $\psi$-ergodic, we get
\[
\lim_{n \to \infty} n^{-1}\sum_{j =1}^nv_j = \Esp[v_1] > 0.
\]
This implies $\lim_{n \to \infty} \sum_{j \leq n}v_j = \infty$ a.s., which is the first condition in 
\eqref{eq:sum1_MSB}.
\end{proof}

\subsection{Proof of Proposition 
\ref{prop:MSB_supp_points}}\label{app:prop:MSB_supp_points}

Define $t = \inf\{j \geq 1: v_j = 1\}$. Note that $w_j = v_j\prod_{i<j}(1-v_i) = 0$ if and only if $v_j = 0$, or $v_i = 1$ for some $i < j$. In turn, this occurs if and only if $v_j = 0$ or $t < j$. Thus, under the event $\bigcap_{j = 1}^{\infty}\{v_j \neq 0\}$, we get $w_j = 0$ if and only if $t < j$. Now, being that $\pi_j(\{0\}) = \Prob[v_j = 0] =0$, we get
\[
\Prob\l[\bigcap_{j \geq 1}\{v_j \neq 0\}\r] = 1,
\]
which yields $w_j > 0$ if and only $j \leq t$, a.s.~or in other words, $t = |\{j: w_j > 0\}|$, a.s. If we further have that $\pi_j(\{1\}) = \Prob[v_j = 1] = 0$, we obtain
\[
\Prob[t = \infty] = \Prob\l[\bigcap_{j \geq 1}\{v_j < 1\}\r] = 1,
\]
which, in particular, yields $0 \leq \Prob[w_j = 0] \leq \Prob[v_j = 0]+ \Prob[t < \infty] = 0$, for each $j \geq 1$, and we can conclude
\[
\Prob\l[\bigcap_{j \geq 1}\{w_j > 0\}\r] = 1.
\]

\subsection{Proof of Proposition 
\ref{prop:MSB_full_supp}}\label{sub:app:prop:MSB_supp_points}

\begin{proof}[Proof of Proposition 
\ref{prop:MSB_full_supp}]

Fix $\varepsilon > 0$, by hypothesis, there exists $\delta > 0$ such that
\[
\Prob\l[\bigcap_{j=1}^m (\delta < v_j < \varepsilon)\r] = \int_{\delta}^{\varepsilon}\int_{\delta}^{\varepsilon} \cdots \int_{\delta}^{\varepsilon}  \bpsi_{m-1}(v_{m-1},\mrm{d}v_m) \cdots \bpsi_1(v_1,\mrm{d}v_2)\pi(\mrm{d}v_1) > 0,
\]
for every $n \geq 1$. The first statement follows from the proof of Proposition 7 by \cite{BO14}. We now turn to prove the remaining part of the assertion, so say that there exist $0 < \epsilon < 1$ such that $(0,\epsilon)$ is contained in the support of $\pi$ and $\bpsi_j(v,\cdot)$, for every $v \in (0,\epsilon)$ and $j \geq 1$. Fix $\varepsilon > 0$ and define $\varepsilon' = \min\{\varepsilon,\epsilon\} >0$ and $\delta = \varepsilon'/2 > 0$. Then $(\delta,\varepsilon') \subset (0,\epsilon)$ and we get
\[
\pi((\delta,\varepsilon')) > 0 \quad \text{and} \quad \int_{\delta}^{\varepsilon'} \bpsi_j(v,\mrm{d}u) = \bpsi(v,(\delta,\varepsilon')) > 0,
\]
for every $v \in (\delta,\varepsilon')$. By a simple induction argument, this yields
\begin{align*}
\Prob\l[\bigcap_{j=1}^m (\delta < v_j < \varepsilon)\r] 
\geq \int_{\delta}^{\varepsilon'}\int_{\delta}^{\varepsilon'} \cdots \int_{\delta}^{\varepsilon'}  \bpsi_{m-1}(v_{m-1},\mrm{d}v_m) \cdots \bpsi_1(v_1,\mrm{d}v_2)\pi(\mrm{d}v_1) > 0
\end{align*}
for every $m \geq 1$, and the proof follows from Proposition 7 of \cite{BO14}. See also \cite{GvdV17}.
\end{proof}

\subsection{Proof of Theorem 
\ref{theo:MSB_conv}}\label{app:theo:MSB_conv} 

Before proving Theorem 
\ref{theo:MSB_conv},
we introduce some notation and definitions. Given a sequence of probability measures $(P_n)_{n \geq 1}$,  $P_n$ is said to converge weakly to $P$, denoted by $P_n \wto P$, whenever $P_n(f) = \int_\X f \mrm{d}P_n\to \int_\X f \mrm{d}P = P(f)$ for every continuous and bounded function $f:\X \to \R$. This is equivalent to $x_n \dto x$ when $x_n \sim P_n$ and $x \sim P$. Similarly, the random probability measures $\bm{P},\bm{P}_1,\bm{P}_2, \ldots$, are said to converge weakly a.s.~when the r.v.~$\bm{P}_n(f) = \int_\X f \mrm{d}\bm{P}_n\to \int_\X f \mrm{d}\bm{P} =  \bm{P}(f)$ a.s.~for every continuous and bounded function $f:\X \to \R$. If instead $\bm{P}_n(f) \dto \bm{P}(f)$, for every $f$ as before, we say $\bm{P}_n$ converges weakly in distribution to $\bm{P}$, denoted by $\bm{P}_n \dwto \bm{P}$. In particular if $\mathbb{X}$ is also a locally compact second countable Hausdorff space $(lcscH)$, $\bm{P}_n \dwto \bm{P}$ is equivalent to $\bm{P}_n \dto \bm{P}$ \citep[cf. Theorem 16.16 and Theorem 4.19 by][respectively]{K02,K17}. Further details on convergence of (random) probability measures can be found in \cite{B68} and \cite{K02,K17}.

Next we state lemmas that will aid the proof of Theorem 
\ref{theo:MSB_conv}.

\begin{lem}\label{lem:dL_cont_map}
Let
$\Delta_{\infty}$ denote the infinite dimensional simplex. 
The mapping from $\X^{\infty} \times \Delta_{\infty}$ into $\cl{P}(\X)$ defined by $[(x_1,x_2,\ldots),(w_1,w_2,\ldots)] \mapsto \sum_{j=1}^{\infty}w_j\delta_{x_j}$ is continuous with respect to the product and weak topologies.
\end{lem}

\begin{proof}[Proof of Lemma \ref{lem:dL_cont_map}]

Let $w = (w_1,w_2,\ldots)$, $w^{(n)}= \big(w^{(n)}_1,w^{(n)}_2,\ldots\big)$ be elements of the infinite dimensional simplex, $\Delta_{\infty}$, and  $x = (x_1,x_2,\ldots)$, $x^{(n)}=\big(x^{(n)}_1,x^{(n)}_2,\ldots\big)$, be elements of $\X^{\infty}$, such that $w_j^{(n)} \to w_j$ and  $x^{(n)}_j \to x_j$, for every $j \geq 1$. Define $P^{(n)} = \sum_{j =1}^{\infty}w^{(n)}_j\delta_{x^{(n)}_j}$ and $P = \sum_{j =1}^{\infty}w_j\delta_{x_j}$. Fix a continuous and bounded function $f: \X \to \R$. Then
\[
w^{(n)}_j\delta_{x^{(n)}_j}(f) = w^{(n)}_jf\l(x^{(n)}_j\r)\to w_jf(x_j) = w_j\delta_{x_j}(f),
\]
for every $j \geq 1$. Since $f$ is bounded, there exist $M$ such that $|f|\leq M$, hence 
\[
\l|w^{(n)}_j\delta_{x^{(n)}_j}(f)\r| = w^{(n)}_j\l|f\l(x^{(n)}_j\r)\r|\leq w^{(n)}_jM,
\]
for every $n \geq 1$, and $j \geq 1$. Evidently, $Mw^{(n)}_j \to Mw_j$, and $\sum_{j=1}^{\infty}Mw^{(n)}_j = M = \sum_{j \geq 1}Mw_j$. Thus, by the general Lebesgue dominated convergence theorem, we obtain
\[
P^{(n)}(f) = \sum_{j= 1}^{\infty}w^{(n)}_j f\l(x^{(n)}_j\r) \to  \sum_{j =1}^{\infty}w_jf(x_j) = P(f),
\]
that is $P^{(n)} \wto P$.
\end{proof}

\begin{lem}[Coupling]\label{lem:coup}
Let $\bm{x},\bm{x}_1,\bm{x}_2,\ldots$ be random elements taking values in $\X$ such that $\bm{x}_n \dto \bm{x}$. Then there exist a probability space and some random elements with $\bm{z}_n \deq \bm{x}_n$, $\bm{z} \deq \bm{x}$, and $\bm{z}_n \to \bm{z}$ a.s.
\end{lem}

\noindent Assume also $(\Y,\B_{\Y})$ to be a complete and separable metric space with corresponding Borel $\sigma$--field.
\begin{lem}[Continuous mappings]\label{lem:cont_map}
Let $\bm{x},\bm{x}_1,\bm{x}_2,\ldots$ be random elements taking values in $\X$ such that $\bm{x}_n \dto \bm{x}$, as $n \to \infty$, and consider some measurable mapping $f,f_1,f_2,\ldots:\X \to \Y$ satisfying $f_n(x_n) \to f(x)$ for every $x_n \to x$ in $\mathbb{X}$. Then $f_n(\bm{x}_n) \dto f(\bm{x})$.
\end{lem}

The proof of Lemmas  \ref{lem:coup} and \ref{lem:cont_map} can be found in \cite{K02} (Theorems 4.30 and 4.27, respectively).

\begin{lem}\label{lem:joint_conv_rcd}
Consider some random elements $\bm{x},\bm{x}_1,\bm{x}_2,\ldots$ and $\bm{y},\bm{y}_1,\bm{y}_2,\ldots$ taking values in $\mathbb{X}$ and $\mathbb{Y}$, respectively. Let $\pi$ be the distribution of $\bm{x}$ and $\pi_n$ that of $\bm{x}_n$. Also consider some regular versions of the conditional distributions $\bpsi(\bm{x},\cdot) = \Prob[\bm{y} \in \cdot\mid \bm{x}]$ and $\bpsi_n(\bm{x}_n,\cdot) = \Prob[\bm{y}_n \in \cdot \mid \bm{x}_n]$. If $\pi_n \wto \pi$ and for each $x_n \to x$, $\bpsi_n(x_n,\cdot) \wto \bpsi(x,\cdot)$, then $(\bm{x}_n,\bm{y}_n) \dto (\bm{x},\bm{y})$.
\end{lem}

\begin{proof}[Proof of Lemma \ref{lem:joint_conv_rcd}]
Let $g:\mathbb{X}\times\mathbb{Y} \to \R$ be a continuous and bounded function. Define the measurable mappings $f,f_1,f_2,\ldots:\X \to \R$ by
\[
f(x) = \int_{\mathbb{Y}} g(x,y)\bpsi(x,\mrm{d}y) \quad \text{ and } \quad f_n(x) = \int_{\mathbb{Y}} g(x,y)\bpsi_n(x,\mrm{d}y).
\]
First we will prove that 
\begin{equation}\label{eq:fn(sn)tof(s)}
f_n(x_n) \to f(x) \quad \text{as} \quad x_n \to x.
\end{equation}
So let $x_n \to x$ in $\mathbb{X}$, and choose some random elements $\bm{z} \sim \bpsi(x,\cdot)$ and $\bm{z}_n \sim \bpsi_n(x_n,\cdot)$, this way, $\bm{z}_n \dto \bm{z}$, by hypothesis. As $g$ is continuous we have that $h_n(y_n) = g(x_n,y_n) \to g(x,y) = h(y)$ for every $y_n \to y$ in $\mathbb{Y}$. Hence, Lemma \ref{lem:cont_map} yields $h_n(\bm{z}_n) \dto h(\bm{z})$. Since $g$ is bounded, the random variables $h_n(\bm{z}_n)$ and $h(\bm{z})$ are uniformly bounded. Thus, by Lemma \ref{lem:coup} and dominated Lebesgue convergence theorem, we obtain $f_n(x_n) = \Esp[h_n(\bm{z}_n)] \to \Esp[h(\bm{z})] = f(x)$. This proves \eqref{eq:fn(sn)tof(s)}, which together with the hypothesis, $\pi_n \wto \pi$, and Lemma \ref{lem:cont_map} yield $f_n(\bm{x}_n) \dto f(\bm{x})$. Using once again the fact that $g$ is bounded, Lemma \ref{lem:coup} and dominated convergence theorem we get
\[
\int_{\mathbb{X}}\int_{\mathbb{Y}} g(x,y)\bpsi_n(x,\mrm{d}y)\pi_n(\mrm{d}x) = \Esp[f_n(\bm{x}_n)] \to \Esp[f(\bm{x})] = \int_{\mathbb{X}}\int_{\mathbb{Y}} g(x,y)\bpsi(x,\mrm{d}y)\pi(\mrm{d}x),
\]
which is equivalent to $\Esp[g(\bm{x}_n,\bm{y}_n)] \to \Esp[g(\bm{x},\bm{y})]$. Since $g$ was chosen arbitrarily this yields $(\bm{x}_n,\bm{y}_n) \dto (\bm{x},\bm{y})$.
\end{proof}

\begin{lem}[Random sequences]\label{lem:conv_rs}
Let $\bm{x}^{(n)} = \big(\bm{x}_j^{(n)}\big)_{j=1}^{\infty}$ and $\bm{x}= (\bm{x}_j)_{j=1}^{\infty}$ be random random elements in $\X^{\infty}$. If
\[
\big(\bm{x}^{(n)}_1,\ldots,\bm{x}^{(n)}_m\big) \dto (\bm{x}_1,\ldots,\bm{x}_m) 
\]
for every $m \geq 1$, it follows that $\bm{x}^{(n)} \dto \bm{x}$.
\end{lem}

The proof of Lemma \ref{lem:conv_rs} can be found in \cite{K02} (Theorem 4.29). With these results at hand the proof of Theorem 
\ref{theo:MSB_conv}
follows easily as shown below.

\begin{proof}[Proof of Theorem 
\ref{theo:MSB_conv}]
Let $\bm{v} = (v_j)_{j \geq 1}$ be the underlying Markov process of $\bm{w}$, and let $\bm{v}^{(n)} = \big(v^{(n)}_j\big)$ be that of $\bm{w}^{(n)}$. Say that for some $m \geq 1$
\begin{equation}\label{eq:conv_v}
\bm{v}^{(n)}_{[m]} = \l(v^{(n)}_1,\ldots,v^{(n)}_m\r) \dto (v_1,\ldots,v_m) = \bm{v}_{[m]}
\end{equation}
as $n \to \infty$. Since $\bm{v}^{(n)}$ and $\bm{v}$ are Markov processes we get
\[
\Prob\l[v^{(n)}_{m+1} \in \cdot\mi \bm{v}^{(n)}_{[m]}\r] = \Prob\l[v^{(n)}_{m+1}\in \cdot\mi v^{(n)}_{m}\r] = \bpsi_m^{(n)}\l(v^{(n)}_{m},\cdot\r),
\]
for each $n \geq 1$, and analogously $\Prob[v_{m+1} \in \cdot \mid \bm{v}_{[m]}] = \bpsi_m(v_m,\cdot)$. By hypothesis we know $\bpsi^{(n)}(v_n,\cdot) \wto \bpsi(v,\cdot)$ as $v_n \to v$ in $[0,1]$. Thus, by Lemma \ref{lem:joint_conv_rcd} we get
\[
\bm{v}^{(n)}_{[m+1]} = \l(v^{(n)}_1,\ldots,v^{(n)}_{m+1}\r) \dto (v_1,\ldots,v_{m+1}) = \bm{v}_{[m+1]}.
\]
This inductive argument together with the fact that $\pi^{(n)} \wto \pi$, or equivalently $v^{(n)}_1 \dto v_1$, show that \eqref{eq:conv_v} holds for every $m \geq 1$, which in turn yields $\bm{v}^{(n)} \dto \bm{v}$ by Lemma \ref{lem:conv_rs}. Now, notice that the mapping
\[(v_1,v_2,v_3\ldots) \mapsto \l(v_1,v_2(1-v_1),v_3(1-v_1)(1-v_2),\ldots\r)
\]
is continuous with respect to the product topology. Hence, by Lemma \ref{lem:cont_map} we obtain $\bm{w}^{(n)} \dto \bm{w}$. Moreover, if we consider the collection of iid atoms, $\bm{\theta}^{(n)}$, of $\bm{P}^{(n)}$ and let $\bm{\theta}$ be that of $\bm{P}$, as $P^{(n)}_0 \wto P_0$ we get $\bm{\theta}^{(n)} \dto \bm{\theta}$. This together with the facts that $\bm{w}^{(n)} \dto \bm{w}$,  $\bm{\theta}^{(n)}$ is independent of $\bm{w}^{(n)}$ for each $n \geq 1$, and  $\bm{\theta}$ is  independent of $\bm{w}$, imply that $\l(\bm{\theta}^{(n)},\bm{w}^{(n)}\r) \dto (\bm{\theta},\bm{w})$. Thus Lemmas \ref{lem:dL_cont_map} and \ref{lem:cont_map} yield $\bm{P}^{(n)} \dwto \bm{P}$. In particular if $\X$ is $lcscH$, this condition is known to be equivalent to $\bm{P}^{(n)} \dto \bm{P}$ \citep[cf. Theorem 16.16 and Theorem 4.19 in][respectively]{K02,K17}.
\end{proof}

\subsection{Proof of Proposition 
\ref{cor:MSB_dec_sup}}\label{app:prop:dec_Ew}

The proof of this result follows easily from the combination of Propositions \ref{prop:MSB_supp_points} and \ref{prop:dec_Ew}.

\subsection{Functional second moments of a species sampling process}\label{app:funct_moments}

\begin{prop}\label{prop:funct_moments}
Let $\bm{P} = \sum_{j\geq 1}w_j\delta_{\theta_j} + \l(1-\sum_{j\geq 1}w_j\r)P_0$ be a species sampling process over $(\X,\B_\X)$. For any measurable and bounded functions $f,g:\X \to \R$ 
one has
\begin{equation}
	\label{eq:funct_second_moment}
	\Esp[\bm{P}(f)\bm{P}(g)] = \tau_p P_0(fg) + (1-\tau_p)P_0(f)P_0(g)
\end{equation}
where $\tau_p$ is the tie probability associated to $\bm{P}$.
\end{prop}

\begin{proof}
For simplicity set $w_0 = 1-\sum_{j\geq 1}w_j$. First note that
\[
1 = \Esp\bigg[\bigg(\sum_{j=0}^\infty w_j\bigg)^2\bigg] = \Esp\bigg[\sum_{j=1}^{\infty} w_j^2\bigg] + \Esp\bigg[\sum_{i \neq j}w_iw_j\bigg] + \Esp[w_0^2]
\]
by a monotone convergence argument. It follows from the definition of the size-biased random order and monotone convergence theorem that
\[
\tau_p = \Esp[\tilde{w}_1] = \Esp[\Esp[\tilde{w}_1\mid \bm{w}]] = \Esp\bigg[\sum_{j=1}^{\infty}w_j^2\bigg] = \sum_{j=1}^{\infty}\Esp[w_j^2]
\]
and $1-\tau_p = \sum_{i\neq j}\Esp[w_iw_j] + \Esp[w_0^2]$. Now, since $f$ and $g$ are bounded and $\bm{P}$ is a probability measure $\bm{P}(f),\bm{P}(g) < \infty$ almost surely. Then,
\begin{align*}
\bm{P}(f)\bm{P}(g) & = \bigg(\sum_{j \geq 1} w_jf(\theta_j) + w_0P_0(f)\bigg)\bigg(\sum_{j \geq 1} w_jg(\theta_j) + w_0P_0(g)\bigg)\\
& = \sum_{j \geq 1}w_j^2f(\theta_j)g(\theta_j)  + \sum_{i \neq j \geq 1} w_iw_jf(\theta_i)g(\theta_j) + w_0^2P_0(f)P_0(g)\\
& \quad \quad \quad \quad \quad \quad \quad + \bigg(\sum_{j \geq 1}w_0w_jg(\theta_j)\bigg)P_0(f)+ \bigg(\sum_{j \geq 1}w_0w_jf(\theta_j)\bigg)P_0(g).
\end{align*}
If $M$ and $N$ are bounds for $f$ and $g$ respectively, we get  $|\sum_{j=1}^nw_j^2f(\theta_j)g(\theta_j)|\leq NM$, $|\sum_{i=1}^n\sum_{j=1}^n w_iw_j f(\theta_i)g(\theta_j)\Ind_{\{i \neq j\}}| \leq NM$, $|\sum_{j=1}^n w_jw_0f(\theta_j)| \leq M$, $|\sum_{j=1}^nw_jw_0g(\theta_j)| \leq N|$, for every $n \geq 1$. Hence the linearity of the expectation and Lebesgue dominated convergence theorem yield
\begin{align*}
\Esp[\bm{P}(f)\bm{P}(g)] & = \sum_{j \geq 1}\Esp[w_j^2f(\theta_j)g(\theta_j) ]+ \sum_{i \neq j \geq 1} \Esp[w_iw_jf(\theta_i)g(\theta_j)]  +\Esp[w_0^2]P_0(f)P_0(g)\\
&  \quad \quad  \quad \quad  \quad \quad + \bigg(\sum_{j \geq 1}\Esp[w_0w_jg(\theta_j)]\bigg)P_0(f) + \bigg(\sum_{j \geq 1}\Esp[w_0w_jf(\theta_j)]\bigg)P_0(g) .
\end{align*}
Finally, since the weights are independent of the atoms and $\theta_j \iid P_0$ we obtain
\begin{align*}
\Esp[\bm{P}(f)\bm{P}(g)] & = \sum_{j \geq 1} \Esp[w_j^2]P_0(fg) + \sum_{i \neq j}\Esp[w_iw_j]P_0(f)P_0(g) + \Esp[w_0^2]P_0(f)P_0(g)\\
& = \tau_p P_0(fg) + (1-\tau_p)P_0(f)P_0(g),
\end{align*}
which proves \eqref{eq:funct_second_moment}.
\end{proof}

If instead we set  $f = g = \Ind_A$ we obtain $\Esp[\bm{P}(A)^2] = \tau_p P_0(A) +(1-\tau_p)P_0(A)^2$ and
\[
\Var(\bm{P}(A)) = \tau_p P_0(A)(1-P_0(A)).
\]
Similarly,  $f = \Ind_A$ and $g = \Ind_B$ yield $\Esp[\bm{P}(A)\bm{P}(B)] = \tau_pP_0(A\cap B) + (1-\tau_p)P_0(A)P_0(B)$, then
\[
\Cov(\bm{P}(A),\bm{P}(B)) = \tau_p (P_0(A\cap B)-P_0(A)P_0(B)).
\]

Proposition \ref{prop:funct_moments} showcases that several important quantities of a species sampling process are completely determined by the tie probability $\tau_p$ and the base measure $P_0$. For some MSBp, in addition to the numerical approximations of $\tau_p$, we can derive expressions for it. In particular, for a BMSBp with parameters $(N, \bm{\alpha},\bm{\beta})$ s.t. $\alpha_j = \alpha$ and $\beta_j = \beta$, for $j \geq 1$:
\[
\tau_p = \sum_{j=1}^{\infty}\,\bigg(\sum_{z_1 = 1}^{N} \ldots \sum_{z_j = 1}^N \l\{\prod_{i=1}^{j}\binom{N}{z_i}\frac{(\alpha'_i)_{z_i+2\Ind_{\{i = j\}}}(\beta'_i)_{N-z_i+2\Ind_{\{i < j\}}}}{(\alpha'_i+\beta'_i)_{2+N}}\r\}\bigg),
\]
where $\alpha'_1 = \alpha$, $\beta'_1 = \beta$ and $\alpha'_i = \alpha+z_{i-1}$, $\beta'_i = \beta +N -z_{i-1}$, for $i \geq 2$. As for LMSBp with parameters $(\rho,\bm{\pi})$, where $\pi_j = \pi = \Be(\alpha,\beta)$:
$$
\tau_p = \sum_{j=1}^{\infty}  \sum_{(t_1,\ldots,t_r)} \rho^{j-r}(1-\rho)^{r-1} \frac{(\alpha)_{2}\,(\beta)_{2(j-t_r)}}{(\alpha+\beta)_{2(j-t_r+1)}} \prod_{i=1}^{r-1}\frac{ (\beta)_{2(t_{i+1}-t_i)}}{(\alpha+\beta)_{2(t_{i+1}-t_i)}},
$$
where the inner sum ranges over all $(t_1,\ldots,t_r)$ such that $t_1 = 1 < t_2 < \cdots < t_r \leq j$.

\subsection{A characterization of $K_n$}\label{app:Kn_sum_geo}

It is well know \citep[cf.][]{P95,P96b} that given the weights in size-biased random order $(\tilde{w}_j)_{j = 1}^\infty$, the partition structure ${\Pi}_n$ can be constructed as follows:
\begin{itemize}
\item Set $A_1 = \{1\}$.
\item Once we have partitioned $\{1,\ldots,i\}$ obtaining ${\Pi}_i = \{A_1,\ldots,A_{K_i}\}$, the element $(i+1)$ is either added to $A_j$ with probability $\tilde{w}_j$, or it defines a new block $A_{K_i+1} = \{i+1\}$ with probability $1-\sum_{j=1}^{K_i} \tilde{w}_j$.
\end{itemize}
Now consider the elements $\kappa_1,\kappa_2,\ldots$ that form new blocks when added, i.e. $\kappa_j = \min A_j$. It then easy to see that $K_n = \max\{j: \kappa_j \leq n\}$ and $G_l = \kappa_{l} -\kappa_{l-1} \sim \mathsf{Geo}\big(1-\sum_{j=1}^{l-1} \tilde{w}_j\big)$ with $l > 1$. Hence we obtain
\[
K_n = \max\bigg\{j:\sum_{l=1}^{j}G_l \leq n\bigg\}
\]
with $G_1 = \kappa_1 = 1$.

\medskip

\section{A MSB characterization of the Pitman-Yor size-biased permuted weights}\label{sec:app:MSB_sb}

\smallskip

To prove Theorem 
\ref{theo:MSB_sb_PY}
we require some preliminary results. The proof of some of them can be found in Pitman and are stated in Subsection \ref{sec:app:MSB_sb:Pit}, further preparatory results are proven and stated in Subsection \ref{sec:app:MSB_sb:more}. 

\subsection{Preliminary results by \cite{P96b}}\label{sec:app:MSB_sb:Pit}

Throughout this section we denote by $\Delta_k = \{(x_1,\ldots,x_k) \in [0,1]^k: \sum_{i=1}^k x_j \leq 1\}$ to the space of all vectors of size $k$ of positive numbers such that its sum is equal or less than one. We will also denote by $\oDk$, and $\partial \Dk$ to the interior and boundary of $\Dk$ respectively, so that $\oDk = \{(x_1,\ldots,x_k) \in [0,1]^k: x_j > 0 \text{ and }\sum_{i=1}^k x_j < 1\}$ and $\partial \Dk = \Delta_k\setminus \oDk$. 

It is clear that the stick-breaking weights $\bm{w}$ can be written as a measurable function of their length variables $\bm{v}$. To ensure the converse also holds, so that $\bm{v}$ and $\bm{w}$ are measurable with respect to each other, throughout this section we work under the convention that $1$ is an absorbing state for $\bm{v}$. In other words, $v_j = 1$ if and only if $\sum_{i=1}^j w_i = 1$, this way we obtain
\[
v_1 = w_1, \quad v_j = \begin{cases}
\frac{w_j}{1-\sum_{i=1}^{j-1}w_i} &\text{ if } \sum_{i=1}^{j-1}w_i < 1\\
1 &\text{ if } \sum_{i=1}^{j-1}w_i = 1
\end{cases}, \quad j \geq 2.
\]

\begin{lem}[Theorem 2 by \cite{P96b}]\label{lem:theo2}
Let $\bm{w}$ be stick-breaking weights with independent length variables, $\bm{v}$, and such that $\sum_{j=1}^{\infty} w_j = 1$ a.s. Then, $\bm{w}$ is invariant under size-biased permutations if and only if one of the following holds:
\begin{enumerate}
\item $w_j > 0$ a.s.~for every $j \geq 1$ and $v_j \sim \mathsf{Be}(1-\sigma,\theta+j\sigma)$ for some $\sigma \in [0,1)$ and $\theta > -\sigma$.
\item $\{j: w_j > 0\} = \{1,\ldots,m\}$ a.s.~for some integer $m \geq 1$, and one of the following holds:
\begin{enumerate}
\item $v_j \sim \mathsf{Be}(1+\sigma,(m-j)\sigma)$ for every $j \in \{1,\ldots,m\}$, and some $\sigma > 0$.
\item $v_j = 1/(m-j+1)$ for every $j \in \{1,\ldots,m\}$, i.e. $w_j = 1/m$ for $j \in \{1,\ldots,m\}$, and $w_j = 0$ for $j > m$.
\item $m = 2$ and the distribution $F$ on $[0,1]$ defined by
\[
F(\mathrm{d}x) = (1-x)\Prob[w_1 \in \mathrm{d}x]/\Esp[1-w_1]
\]
is symmetric about $1/2$.
\end{enumerate}
\end{enumerate}
\end{lem}

In order to formally state the second result we will require a preliminary definition: Let $(v_1,v_2)$ be a pair of r.v.~each taking values in $[0,1]$, we say that the law of $(v_1,v_2)$ is \emph{acceptable} iff
\[
g(r,s) = \Esp[v_1^r(1-v_1)^{s+1}v_2^s]
\]
is symmetric, i.e. $g(r,s) = g(s,r)$, for all pairs of non-negative integers $r$ and $s$.

\begin{lem}[Corollary 7 by \cite{P96b}]\label{lem:cor7}
Let $\bm{w}$ be stick-breaking weights with length variables, $\bm{v}$. Then $\bm{w}$ is invariant under size-biased permutations if and only if all of the following holds:
\begin{enumerate}
\item The law of $(v_1,v_2)$ is acceptable.
\item On the event $\{\sum_{j=1}^k w_j < 1\}$, the conditional law of $(v_{k+1},v_{k+2})$ given $\bm{w}_{[k]} = (w_1,\ldots,w_k)$ is acceptable a.s.
\item The conditional law in 2 can be chosen to dependent symmetrically on the first $k$ weights, $\bm{w}_{[k]} = (w_1,\ldots,w_k)$.
\end{enumerate}
\end{lem}

\begin{cor}\label{cor7'}
In the proof of Corollary 7 \cite{P96b} showed that \emph{3.} in Lemma \ref{lem:cor7} can be replace by 
\begin{itemize}
\item[3$'$.] On the event $\{\sum_{j=1}^k w_j < 1\}$, the conditional law of $v_{k+1}$ can be chosen to dependent symmetrically on the first $k$ weights, $\bm{w}_{[k]} = (w_1,\ldots,w_k)$.
\end{itemize}
\end{cor}

\begin{lem}[Remark 8 by \cite{P96b}]\label{lem:rem8}
Let $\bm{w}$ be stick-breaking weights and set $W_k = \sum_{j=1}^k w_j$. For each $k \geq 1$ and $n \geq 1$, on the event $\{W_k < 1\}$, there is a version of the conditional law of $(w_{k+n}/(1-W_k))_{n=1}^{\infty}$ given $\bm{w}_{[k]}$ that is invariant under size-biased permutations and that depends symmetrically on $\bm{w}_{[k]}$. 
\end{lem}

\begin{rem}\label{rem:rem?}
In particular if the length variables are independent, and under the assumption $v_j = 1$ if and only if $\sum_{i=1}^j w_i = 1$, Lemma \ref{lem:rem8} states that on the event $\{v_k < 1\}$, $(v_{k+n})_{n=1}^{\infty}$ define stick-breaking weights that are invariant under size-biased permutations.
\end{rem}

\begin{lem}[Lemma 10 together with Lemma 12 by \cite{P96b}]\label{lem:lem10,12}
Let $v_1$ and $v_2$ be independent r.v.~each taking values in $[0,1]$. Then law of $(v_1,v_2)$ is acceptable if and only if one of the following holds:
\begin{enumerate}
\item $v_2 = 1$ a.s.~and the distribution of $F$ on $[0,1]$ defined by
\[
F(\mathrm{d}x) = (1-x)\Prob[v_1 \in \mathrm{d}x]\Esp[1-v_1]
\]
is symmetric about $1/2$.
\item $v_1 \sim \mathsf{Be}(\alpha,\beta)$ and $v_2 \sim\mathsf{Be}(\alpha,\beta-\alpha+1)$ for some $0 < \alpha < \beta +1$.
\item $v_1 = c$ and $v_2 = c/(1-c)$ for some constant $c$ with $0 < c < 1/2$.
\end{enumerate}
\end{lem}

\subsection{More preliminary results}\label{sec:app:MSB_sb:more}

Recall $\Delta_k = \{(x_1,\ldots,x_k) \in [0,1]^k: \sum_{i=1}^k x_j \leq 1\}$  and $\oDk$ denotes its interior.

\begin{rem}\label{rem:wk=0}
It follows from equation 
\eqref{eq:sb_pick}
that if $\bm{w}$ is invariant under size-biased permutations and $\sum_{j=1}^{\infty}w_j = 1$, then $w_k = 0$ if and only if $\sum_{j=1}^{k-1}w_j = 1$ in which case we also have $w_{k+j} = 0$ for every $j \geq 1$.
\end{rem}

\begin{lem}\label{lem:symm_A}
Let $\bm{w}$ be a weights sequence that it is invariant under size-biased permutations, with $\sum_{j=1}^{\infty} w_j = 1$ a.s.~and such that $w_k > 0$ a.s. Then for any permutation $\bm{\sigma} = (\sigma_1,\ldots\sigma_k)$ of $[k]$, and every measurable set $A \subseteq \Dk$ with $\Prob[\bm{w}_{[k]} \in A] = 1$, we have $\Prob[\bm{w}_{[k]} \in \bm{\sigma}(A)] = 1$ where $\bm{\sigma}(A) = \{ \bm{x}_{\bm{\sigma}}=(x_{\sigma_1},\ldots,x_{\sigma_k}): \bm{x} = (x_1,\ldots,x_k) \in A\}$.
\end{lem}

\begin{proof}
Let $\bm{\varrho} = (\varrho_j)_{j=1}^{\infty}$ be as in 
\eqref{eq:sb_pick}. 
Then,
\[
1=\Prob[\bm{w}_{[k]} \in A] = \Prob[\bm{w}_{\bm{\varrho}_{[k]}} \in A] = \sum_{\bm{\alpha}} \Prob[\bm{w}_{\bm{\alpha}} \in A\mid \bm{\varrho}_{[k]} = \bm{\alpha}] \Prob[\bm{\varrho}_{[k]} = \bm{\alpha}]
\]
with $\bm{w}_{\bm{\varrho}_{[k]}} = (w_{\varrho_1},\ldots,w_{\varrho_k})$, $\bm{w}_{\bm{\alpha}}  = (w_{\alpha_1},\ldots,w_{\alpha_k})$, and where sum ranges over all deterministic vectors of size $k$, $\bm{\alpha} = (\alpha_1,\ldots, \alpha_k)$, of distinct positive integers such that $ \Prob[\bm{\varrho}_{[k]} = \bm{\alpha}] > 0$. As $\sum_{\bm{\alpha}}\Prob[\bm{\varrho}_{[k]} = \bm{\alpha}] = 1$ and $\Prob[\bm{w}_{\bm{\alpha}} \in A\mid \bm{\varrho}_{[k]} = \bm{\alpha}] \leq 1$ we get
\[
\Prob[\bm{\varrho}_{[k]} = \bm{\alpha}] > 0 \quad \Rightarrow \quad  \Prob[\bm{w}_{\bm{\alpha}} \in A\mid \bm{\varrho}_{[k]} = \bm{\alpha}] = 1.
\]
In particular, as a consequence of $w_k > 0$ a.s.~and Remark \ref{rem:wk=0}, if $\bm{\alpha} = (\alpha_1,\ldots, \alpha_k)$ defines a permutation of $[k]$, we get
\begin{equation}\label{eq:symm_A_a}
\Prob[\bm{\varrho}_{[k]} = \bm{\alpha}] = \Esp[\Prob[\bm{\varrho}_{[k]} = \bm{\alpha}\mid \bm{w}]]  = \Esp\l[\prod_{j=1}^k w_{\alpha_j} \l(1-\sum_{i=1}^{j-1}w_{\alpha_i}\r)^{-1}\r] \geq \Esp\l[\prod_{j=1}^k w_{j}\r] >0,
\end{equation}
(under the convention that the empty sum equals $0$) which implies
$
\Prob[\bm{w}_{\bm{\alpha}} \in A\mid \bm{\varrho}_{[k]} = \bm{\alpha}] = 1.
$
Next we will prove that $\Prob[\bm{w}_{\bm{\alpha}} \in A] =  1$ for the permutation $\bm{\alpha}$. If $\Prob[\bm{w}_{\bm{\alpha}} \in A] <  1$, then $\Prob[\bm{w}_{\bm{\alpha}} \in A^c]  > 0$ and we also have 
\begin{equation}\label{eq:symm_A_b}
\Prob[\bm{w}_{\bm{\alpha}} \in A^c \cap B] > 0,
\end{equation}
with $B = \{\bm{x} = (x_1,\ldots,x_k) \in \Dk: x_j > 0\}$ (this because $\Prob[\bm{w}_{\bm{\alpha}} \in B] = \Prob[w_j > 0, \forall j \leq k] =1$ by Remark \ref{rem:wk=0}). Noting that $\prod_{j=1}^k w_j > 0$ under the event $\{\bm{w}_{\bm{\alpha}} \in A^c \cap B\}$, we get
\begin{equation}\label{eq:symm_A_c}
\begin{aligned}
\Prob[\bm{\varrho}_{[k]} = \bm{\alpha} \mid \bm{w}_{\bm{\alpha}} \in A^c \cap B] 
& = \frac{\Esp[\Ind_{\{\bm{\varrho}_{[k]} = \bm{\alpha} \}}\Ind_{\{\bm{w}_{\bm{\alpha}} \in A^c \cap B\}}]}{\Prob[\bm{w}_{\bm{\alpha}} \in A^c \cap B]}\\
 &= \frac{\Esp[\Prob[\bm{\varrho}_{[k]} = \bm{\alpha} \mid \bm{w}]\Ind_{\{\bm{w}_{\bm{\alpha}} \in A^c \cap B\}}]}{\Prob[\bm{w}_{\bm{\alpha}} \in A^c \cap B]}
\\
& \geq \frac{\Esp[\prod_{j=1}^k w_j\Ind_{\{\bm{w}_{\bm{\alpha}} \in A^c \cap B\}}]}{\Prob[\bm{w}_{\bm{\alpha}} \in A^c \cap B]} > 0
\end{aligned}
\end{equation}
Putting together equations \eqref{eq:symm_A_a}--\eqref{eq:symm_A_c} we obtain
\[
\Prob[\bm{w}_{\bm{\alpha}} \in A^c \cap B\mid \bm{\varrho}_{[k]} = \bm{\alpha}] = \frac{\Prob[\bm{\varrho}_{[k]} = \bm{\alpha} \mid \bm{w}_{\bm{\alpha}} \in A^c \cap B]\Prob[\bm{w}_{\bm{\alpha}} \in A^c \cap B] }{\Prob[\bm{\varrho}_{[k]} = \bm{\alpha}]} >0,
\]
and $\Prob[\bm{w}_{\bm{\alpha}} \in A^c \mid \bm{\varrho}_{[k]} = \bm{\alpha}] > 0$ which contradicts $\Prob[\bm{w}_{\bm{\alpha}} \in A\mid \bm{\varrho}_{[k]} = \bm{\alpha}] = 1$.
Hence for every permutation $\bm{\alpha}$ we have $\Prob[\bm{w}_{\bm{\alpha}} \in A] =  1$. At this stage, for a fixed permutation $\bm{\sigma} = (\sigma_1,\ldots,\sigma_k)$ of $[k]$ we can take $\bm{\alpha}$ equal to the inverse permutation of $\bm{\sigma}$ and obtain
\[
\Prob[\bm{w}_{[k]} \in \sigma(A)] = \Prob[\bm{w}_{\bm{\alpha}} \in A] = 1.
\]
\end{proof}

\begin{cor}\label{cor:Sk_sym}
Let $\bm{w}$ be a weights sequence that it is invariant under size-biased permutations with $\sum_{j=1}^{\infty} w_j = 1$ a.s.~and such that   $w_k > 0$ a.s. Then the support of $\bm{w}_{[k]}$, $\mathcal{S}$, is symmetric in the sense that $\bm{x} \in \mathcal{S}$ if and only if $\bm{x}_{\bm{\sigma}} = (x_{\sigma_1},\ldots, x_{\sigma_k}) \in \mathcal{S}$ for every permutation $\bm{\sigma} = (\sigma_1,\ldots,\sigma_k)$ of $[k]$.
\end{cor}

\begin{proof}
By definition $\mathcal{S}$ is the smallest closed set such that $\Prob[\bm{w}_{[k]} \in \mathcal{S}] = 1$, and it follows from Lemma \ref{lem:symm_A} that $\Prob[\bm{w}_{[k]} \in \bm{\sigma}(\mathcal{S})] = 1$ for every permutation $\bm{\sigma}$ of $[k]$. It is also easy to see that that $\bm{\sigma}(\mathcal{S})$ is closed. Indeed, if $\bm{x}$ is an accumulation point of $\bm{\sigma}(\mathcal{S})$ then $\bm{x}_{\sigma^{-1}}=(x_{\sigma^{-1}_1},\ldots,x_{\sigma^{-1}_k})$ is an accumulation point of $\mathcal{S}$ which means $\bm{x}_{\bm{\sigma}^{-1}} \in \mathcal{S}$ and therefore $\bm{x} = \bm{x}_{(\bm{\sigma}\circ \bm{\sigma}^{-1})} \in \bm{\sigma}(\mathcal{S})$. That is $\bm{\sigma}(\mathcal{S})$ contains all its accumulation points. Now, since the intersection of closed sets is closed,
$\bigcap_{\bm{\sigma}} \bm{\sigma}(\mathcal{S})$ is closed and being that the countable intersection almost sure events occurs a.s., $\Prob[\bm{w}_{[k]} \in \bigcap_{\bm{\sigma}} \bm{\sigma}(\mathcal{S})] = 1$, where the intersection ranges over all $k!$ permutation of $[k]$. Hence $\mathcal{S} \subseteq \bigcap_{\bm{\sigma}} \bm{\sigma}(\mathcal{S}) \subseteq \mathcal{S}$.
\end{proof}

\begin{cor}\label{cor:Sk_Sj}
Let $\bm{w}$ be a weights sequence that it is invariant under size-biased permutations with $\sum_{j=1}^{\infty} w_j = 1$ a.s.~and such that $w_m > 0$ a.s.~for some $m \geq 1$.Then for any $k \leq m$, and any vector of distinct positive integers $\bm{\alpha} = (\alpha_1,\ldots,\alpha_k) \subseteq [m]$, the support of $\bm{w}_{\bm{\alpha}} = (w_{\alpha_1},\ldots,w_{\alpha_k})$ is equal to the support of $\bm{w}_{[k]} = (w_1,\ldots,w_k)$.
\end{cor}

\begin{proof}
Fix $k \leq m$ and let $\bm{\alpha} = (\alpha_1,\ldots,\alpha_k)$ be a vector of distinct positive integers with $\alpha_j \in [m]$. Let $\bm{\sigma}$ be the permutation of $[m]$ given by $\bm{\sigma}_{\alpha_i} = i$, $\bm{\sigma}_i = \alpha_i$, and $\bm{\sigma}_l = l$ for every other number in $[m]$. Then by Lemma \ref{lem:symm_A} and since $1 = \Prob[\bm{w}_{[k]} \in \mathcal{S}]$, where $\mathcal{S}$ is the support of $\bm{w}_{[k]}$, we get
\begin{align*}
1 = \Prob[\bm{w}_{[k]} \in \mathcal{S}] & = \Prob[\bm{w}_{[k]} \in \mathcal{S}, (w_{k+1},\ldots,w_{m}) \in \Delta_{m-k}]\\
& = \Prob[(w_{\sigma_1},\ldots,w_{\sigma_k}) \in \mathcal{S}, (w_{\sigma_{k+1}},\ldots,w_{\sigma_m}) \in \Delta_{m-k}]\\
& = \Prob[w_{\bm{\alpha}} \in \mathcal{S}].
\end{align*}
This proves that the support of $w_{\bm{\alpha}}$ is contained in the support of $\bm{w}_{[k]}$ and by symmetry we also have the converse contention. Thus the support of $\bm{w}_{[k]}$ and that of $\bm{w}_{\bm{\alpha}}$ are equal.
\end{proof}

\begin{cor}\label{cor:0}
Let $\bm{w}$ be a weights sequence that it is invariant under size-biased permutation with $\sum_{j=1}^{\infty} w_j = 1$ and such that $w_j > 0$, for every $j \geq 1$, a.s. Then, the origin, $\bm{0} = (0,\ldots,0)$, belongs to the support, $\mathcal{S}$, of $\bm{w}_{[k]}$.
\end{cor}

\begin{proof}
For $\varepsilon > 0$ define $t_{\varepsilon} = \inf\{j \geq 1: w_i < \varepsilon, \, \forall i \geq j\}$. Since $\sum_{j=1}^{\infty} w_j = 1$ a.s.~we have that 
\[
1 \geq \Prob[t_\varepsilon < \infty] = \Prob\l[\bigcup_{j \geq 1}\bigcap_{i \geq j} \{w_i < \varepsilon\}\r] \geq \Prob\l[\sum_{j=1}^{\infty} w_j < \infty\r] = 1.
\]
Now, fix $k \geq 1$ and say that $\bm{0} \not\in \mathcal{S}$, being that $\mathcal{S}$ is closed, this means that $\bm{0}$ is not an accumulation point of $\mathcal{S}$. Thus the there exist $r > 0$ such that the ball with radius $r$ and center $\bm{0}$, $\mathcal{B}_r(\bm{0})$, is contained in the complement of $\mathcal{S}$. It follows from Corollary \ref{cor:Sk_Sj} that $\mathcal{B}_r(\bm{0})$ must be contained in the complement of the support of $(w_{j+1},\ldots,w_{j+k})$, for every $j \geq 1$. Thus, $\Prob[w_{j+1} < r, \ldots, w_{j+k} < r] = 0$, which means
\[
0 \leq \Prob[t_r < \infty] = \Prob\l[\bigcup_{j \geq 1} \bigcap_{i \geq j} \{w_i < r\}\r] \leq \Prob\l[\bigcup_{j \geq 1} \bigcap_{i = 1}^k \{w_{j+i} < r\}\r] \leq \sum_{j \geq 1} \Prob\l[\bigcap_{i = 1}^k \{w_{j+i} < r\}\r] = 0,
\]
contradicting $\Prob[t_r < \infty] = 1$. Hence $\bm{0}$ must be an accumulation point of $\mathcal{S}$, i.e. $\bm{0} \in \mathcal{S}_{\bm{w}_{[k]}}$.
\end{proof}

\begin{lem}\label{lem:keyA}
Let $\bm{w}$ be stick-breaking weights that remain invariant under size-biased permutations with $\sum_{j=1}^{\infty}w_j = 1$ a.s.~and $\Prob[w_j > 0] > 0$ for every $j \geq 1$. Say that
\begin{itemize}
\item[\emph{A0.}] $v_3$ is conditionally independent of $v_1$ given $v_2$, i.e $\Prob[v_3 \in \cdot\mid v_2] \aseq \Prob[v_3 \in \cdot\mid v_1,v_2]$ (this is already given if $\bm{w}$ is an MSB).
\item[\emph{A1.}] The law of $\bm{v}_{[2]}$ and the Lebesgue measure, $\lambda^2$ over $([0,1]^2,\B_{[0,1]^2})$, are mutually absolutely continuous. This is for every $A \in \B_{[0,1]^2}$, $\Prob[\bm{v}_{[2]} \in A] = 0$ if and only if $\lambda^2(A) = 0$.
\end{itemize}
Then $v_1,v_2$ and $v_3$ are independent Beta distributed r.v.
\end{lem}

\begin{proof}
\textbf{Part 1:} Since $\bm{w}$ is invariant under size-biased permutations by Corollary \ref{cor7'} there is a version of the conditional law of $v_{3}$ given $\bm{w}_{[2]}$, $\Prob[v_{3}\in \cdot \mid \bm{w}_{[2]}]$, that depends symmetrically on $\bm{w}_{[2]} = (w_1,w_2)$. In other words there exist a measurable function $\mu:\D2 \to \mathcal{P}([0,1])$, where $\mathcal{P}([0,1])$ denotes the space of all probability measures over $([0,1],\B_{[0,1]})$,  such that $\mu$ is symmetric, i.e. $\mu((x_1,x_2)) = \mu((x_2,x_1))$, and $\mu(\bm{w}_{[2]}) \aseq \Prob[v_{3}\in \cdot \mid \bm{w}_{[2]}]$.
Now, from the assumption $v_j = 1$ if and only if $\sum_{i=1}^j w_i = 1$, we know that $\bm{w}_{[2]}$ and $\bm{v}_{[2]}$ are measurable with respect to each other, putting this together with the assumption that $v_3$ is conditionally independent of $v_1$ given $v_2$, we get a.s.
\begin{equation}\label{eq:MSB_v3}
\Prob[v_{3} \in \cdot\mid \bm{w}_{[2]}] = \Prob[v_{3} \in \cdot\mid \bm{v}_{[2]}] =\Prob[v_{3} \in \cdot\mid v_2].
\end{equation}
Hence, regarding $v_2 = f(\bm{w}_{[2]})$ as a function of $\bm{w}_{[2]}$, i.e.
\begin{equation}\label{eq:f2}
f(\bm{w}_{[2]}) = \begin{cases}
w_2/(1-w_1) & \text{ if } w_1 < 1\\
1 & \text{ if } w_1 = 1,
\end{cases}
\end{equation}
we get that $\nu = \phi \circ f:\D2 \to \mathcal{P}([0,1])$ is another version of $\Prob[v_{3} \in \cdot\mid \bm{w}_{[2]}]$, where $\phi:[0,1] \to \mathcal{P}([0,1])$ is any version of $\Prob[v_3\in\cdot\mid v_2]$, so 
\[
\mu(\bm{w}_{[2]})  =  \phi(f(\bm{w}_{[2]})) =: \nu(\bm{w}_{[2]}),
\]
a.s., this is there exist a measurable set $A \subseteq \D2$ such that $\Prob[\bm{w}_{[2]} \in A] = 1$ and for every $\bm{x} = (x_1,x_2) \in A$, $\mu(\bm{x}) = \nu(\bm{x})$. Now, as $w_1 + w_2 < 1$ a.s. (this follows from A1) Remark \ref{rem:wk=0} and Lemma \ref{lem:symm_A} explain that we may assume without loss of generality that $A$ is symmetric (otherwise we can replace it with $A \cap \bm{\sigma}(A)$, where $\sigma_1 = 2$ and $\sigma_2 = 1$, which also has probability one under the law of $\bm{w}_{[2]}$). Under this consideration, also note that $\nu$ is symmetric on $A$. Indeed, for $\bm{x} = (x_1,x_2), \bm{x}_{\bm{\sigma}}= (x_2,x_1) \in A$,
\[
\nu(\bm{x}_{\bm{\sigma}}) = \mu(\bm{x}_{\bm{\sigma}}) = \mu(\bm{x}) = \nu(\bm{x}),
\]
because $\mu$ is symmetric. Thus, the set 
$B = \{(x_1,x_2) \in \D2: \nu((x_1,x_2)) =  \nu((x_2,x_1))\}$
also satisfies $\Prob[\bm{w}_{[2]} \in B] = 1$ and can be assumed to be symmetric likewise. 

\smallskip

\textbf{Part 2:} Now define the equivalence relation over $[0,1]$, $u \sim v$ if and only if $\phi(u) = \phi(v)$ and let $\mathcal{A} = \{A_i\}_{i \in I}$ be the set of equivalence classes of $\sim$. Note that $\mathcal{A}$ is a measurable partition of $[0,1]$ because $\phi:[0,1] \to \mathcal{P}([0,1])$ is measurable (with respect to the Borel $\sigma$-algebra $\B_{\mathcal{P}([0,1])}$, of $\mathcal{P}([0,1])$). We will prove that there exist $i \in I$ such that $\Prob[v_2 \in A_i] = 1$ by contradiction. So assume there exist two distinct indexes $i,j \in I$ with  $\Prob[v_2 \in A_i] > 0$,  $\Prob[v_2 \in A_j] >0.$ Set $C_l = \{(x_1,x_2) \in \D2: x_2/(1-x_1) \in A_l\}$ for $l \in \{i,j\}$ so that $v_2 \in A_l$ if and only if $\bm{w}_{[2]} \in C_l$. For any $\bm{x} = (x_1,x_2) \in C_i\cap \oD2$ and $\bm{y}= (y_1,y_2) \in C_j\cap \oD2$ set $v_x = f(\bm{x}) = x_2/(1-x_1)$ and $v_y = f(\bm{y}) = y_2/(1-y_1)$ and consider the point $\bm{z} = (z_1,z_2)$ given by
\[
z_1 = \frac{v_y(1-v_x)}{1-v_x v_y}, \quad \text{ and } \quad z_2 = \frac{v_x(1-v_y)}{1-v_x v_y}
\]
Note that $z_2/(1-z_1) = v_x$ and $z_1/(1-z_2) = v_y$ hence $\bm{z} \in C_i\cap\oD2$ and $\bm{z}_{\bm{\sigma}} = (z_2,z_1) \in C_j\cap \oD2$ with $\bm{\sigma} = (2,1)$. Now, if $\bm{z} \in B$ we have
\[
\nu(\bm{x}) = \phi(v_x) = \nu(\bm{z}) = \nu(\bm{z}_{\bm{\sigma}}) = \phi(v_y) = \nu(\bm{y}).
\]
In particular $\phi(v_x) = \phi(v_y)$ with $v_x \in A_i$ and $v_y \in A_j$ which contradicts that $A_i$ and $A_j$ are two distinct equivalence classes. We must therefore have $\bm{z} \not\in B$. In other words, if we set $D_l = C_l\cap\oD2$ and regard $\bm{z}$ as a function of $\bm{x}$ and $\bm{y}$ i.e. $\bm{z} = g(\bm{x},\bm{y})$ with $g:\oD2\times\oD2 \to \oD2$ given by
\[
g(\bm{x},\bm{y}) = \l(\frac{f(\bm{x})(1-f(\bm{y}))}{1-f(\bm{x})f(\bm{y})}, \frac{f(\bm{y})(1-f(\bm{x}))}{1-f(\bm{x})f(\bm{y})}\r)
\]
and $f$ as in \eqref{eq:f2}, then the direct image under $g$ of $D_i\times D_j$ satisfies $g[D_i \times D_j] \subseteq \D2 \setminus B$. Furthermore, by construction we have that
\[
g[D_i\times D_j] = D_i \cap \bm{\sigma}(D_j) = \{(z_1,z_2) \in \oD2: z_2/(1-z_1)\in A_i, z_1/(1-z_2) \in A_j\}.
\]
So that $\bm{w}_{[2]} \in g[D_i\times D_j]$ if and only if $v_2 \in A_i$ and $v_1/(1-v_1(1-v_2)) \in A_j$. It then follows from assumption A1 in the statement of the Lemma that
\[
\Prob[\bm{w}_{[2]} \in \oD2\setminus B] \geq \Prob[\bm{w}_{[2]} \in g[D_i \times D_j]] = \Prob[v_2 \in A_i,v_1/(1-v_1(1-v_2)) \in A_j] > 0
\]
(because $\Prob[v_2 \in A_l] > 0$ implies $\lambda(A_l) > 0$ and the mapping $(v_1,v_2) \mapsto (v_1/(1-v_1(1-v_2)),v_2)$ defines a continuous bijection). Since $\Prob[\bm{w}_{[k]} \in B] = 1$, this cannot occur. The contradiction arises from  the assumption that there exist two distinct indexes $i,j \in I$ with  $\Prob[v_2 \in A_i] > 0$,  $\Prob[v_2 \in A_j] >0.$ Thus there can only be one index $i \in I$ such that $\Prob[v_2 \in A_i] > 0$, and since $\mathcal{A} = \{A_i\}_{i \in I}$ is a measurable partition we must have $\Prob[v_2 \in A_i] = 1$. This is, $\phi$ is a.s.~constant and by the tower property of conditional expectation we must have
\[
\Prob[v_3 \in \cdot\,] = \Esp[\Prob[v_3 \in \cdot\mid v_2]] = \Esp[\phi(v_2)] = \phi(u)
\]
for $u \in A_i$. Thus $\phi = \Prob[v_3 \in \cdot\,]$ a.s.~which means that the law of $v_3$ is version of $\Prob[v_3 \in \cdot\mid v_2]$ and from \eqref{eq:MSB_v3} we also get it is a version of $\Prob[v_3 \in \cdot\mid \bm{v}_{[2]}]$ and $\Prob[v_3 \in \cdot\mid \bm{w}_{[2]}]$, proving that $v_3$ is independent of the pair $\bm{v}_{[2]} = (v_1,v_2)$.

\smallskip

\textbf{Part 3:}  Next we show that $v_1$ and $v_2$ are also independent. Notice that as $v_3$ is conditionally independent of $v_1$ given $v_2$, $\Prob[v_3 \in \cdot\,]$ is a version of $\Prob[v_3 \in \cdot \mid v_{[2]}]$. In particular
\[
\Prob[v_3 \in \cdot \mid v_1] = \Esp[\Prob[v_3 \in \cdot \mid \bm{v}_{[2]}]\mid v_1] = \Prob[v_3 \in \cdot\,]
\]
and 
\[
\Prob[v_3 \in \cdot \mid \bm{v}_{[2]}] = \Prob[v_3 \in \cdot\,] = \Prob[v_3 \in \cdot \mid v_1]
\]
a.s. That is $v_3$ is independent of $v_1$, and $v_2$ and $v_3$ are conditionally independent given $v_1$. Putting this together with Lemmas \ref{lem:cor7} (\emph{2}), Lemma  \ref{lem:lem10,12}, and the fact that $w_1 + w_2 < 1$ a.s. (equiv. $v_1,v_2 < 1$ a.s.~due to A1),  we get that one of the following holds:
\begin{enumerate}
\item $v_3 = 1$ a.s.~and the distribution of $F$ on $[0,1]$ defined by
\[
F(\mathrm{d}x) = (1-x)\Prob[v_2 \in \mathrm{d}x\mid v_1]/\Esp[1-v_2 \mid v_1]
\]
is symmetric about $1/2$. This cannot occur because if $v_3 = 1$ a.s then $w_j = 0$ a.s.~for every $j \geq 4$ contradicting $\Prob[w_j > 0] > 0$ for every $j \geq 1$.
\item There exist some $v_1$-measurable r.v.~$\alpha = \alpha(v_1)$ and $\beta = \beta(v_1)$ with $0 < \alpha < \beta+1$ a.s.~such that the Beta distributions $\Be(\alpha,\beta)$ and $\Be(\alpha,\beta-\alpha+1)$ are regular versions of $\Prob[v_2 \in \cdot\mid v_1]$ and $\Prob[v_3 \in \cdot \mid v_1]$, respectively. 
Being that $v_3$ is in particular independent of $v_1$, we get
\begin{align*}
\alpha(v_1) &= \l(\frac{\Esp[v_3\mid v_1](1-\Esp[v_3\mid v_1])}{\Var(v_3 \mid v_1)}\r)\Esp[v_3\mid v_1] = \l(\frac{\Esp[v_3](1-\Esp[v_3])}{\Var(v_3)}\r)\Esp[v_3]\\
\gamma(v_1) &= \l(\frac{\Esp[v_3\mid v_1](1-\Esp[v_3\mid v_1])}{\Var(v_3 \mid v_1)}\r)(1-\Esp[v_3\mid v_1]) = \l(\frac{\Esp[v_3](1-\Esp[v_3])}{\Var(v_3)}\r)(1-\Esp[v_3]),\\
\end{align*}
where $\gamma(v_1) = \beta(v_1)-\alpha(v_1)+1$, so $\Prob[v_2 \in \cdot \mid v_1] = \Be(\alpha,\beta)$, for some constants $0 < \alpha < \beta +1$, and thus $v_2$ is also independent of $v_1$. In this case it also follows from Lemma \ref{lem:cor7} (\emph{1}) and  Lemma \ref{lem:lem10,12} (\emph{2}) that $v_1$ is Beta distributed.
\item There exist some $v_1$-measurable r.v.~$c = c(v_1)$ with $0 < c < 1/2$ such that $v_2 = c$ and $v_3 = c/(1-c)$ a.s. Equivalently,  $w_2 = w_3 > 0$ a.s.~and it follows from Lemma \ref{lem:symm_A} (for the set $A = \{\bm{x} = (x_1,x_2,x_3) \in \Delta_3: x_2 = x_3\}$ and for the permutation, $\bm{\sigma} = (3,1,2)$, of $[3]$) that $\Prob[w_1 = w_2] = 1 = \Prob[w_2 = w_3]$. This implies the law of $v_1$ and $v_2$ concentrates in the set $\{\bm{u} \in [0,1]^2: u_2 = u_1/(1-u_1)\}$ which has Lebesque measure zero, contradicting A1. Hence this cannot occur either.
\end{enumerate}
We conclude (2) holds and we have $v_1,v_2$ and $v_3$ are independent Beta distributed r.v.
\end{proof}

\begin{lem}\label{lem:keyB}
Let $\bm{w}$ be stick-breaking weights that remain invariant under size-biased permutations with $\sum_{j=1}^{\infty}w_j = 1$, $w_1+w_2 < 1$ a.s.~and $\Prob[w_j > 0] > 0$ for every $j \geq 1$.  Let $\Delta_{2,\varepsilon} = \{(x_1,x_2) \in \D2: x_1+x_2 \leq \varepsilon\}$ and define $E = \inf\{\varepsilon > 0: \mathcal{S} \subseteq \Delta_{2,\varepsilon}\}$ where $\mathcal{S}$ is the support of $\bm{w}_{[2]}$. Say that
\begin{itemize}
\item[\emph{B0.}] $v_3$ is conditionally independent of $v_1$ given $v_2$, i.e $\Prob[v_3 \in \cdot\mid v_2] \aseq \Prob[v_3 \in \cdot\mid v_1,v_2]$ (this is already given if $\bm{w}$ is an MSB).
\item[\emph{B1.}] For every $0 < \varepsilon < E$, $\mathcal{D}_{2,\varepsilon} = \{(x,x) \in \Delta_{2,\varepsilon}: 0 < x \leq \varepsilon/2\}$, is contained in the interior of the support $\mathcal{S}_{\bm{w}_{[2]}}$. 
\item[\emph{B2.}] There is a version $\phi:[0,1] \to \mathcal{P}([0,1])$, of the conditional law $\Prob[v_{3}\in \cdot \mid v_2]$ that is left continuous (w.r.t. the weak topology). This is, for every $u_1,u_2,\ldots$ in $[0,1]$ such that $u_n < u$ and $u_n \to u$ we get $\phi(u_n) \wto \phi(u)$. 
\end{itemize}
Then $v_1,v_2$ and $v_3$ are independent Beta distributed r.v.
\end{lem}

\begin{proof}
\textbf{Part 1:} As in the proof of Lemma \ref{lem:keyA} and we get $\nu = \phi \circ f$ is a version of $\Prob[v_3  \in \cdot\mid \bm{w}_{[2]}]$, where $\phi$ is as in B1 and $f$ as in \eqref{eq:f2}. Moreover the set $B = \{(x_1,x_2) \in \D2: \nu((x_1,x_2)) =  \nu((x_2,x_1))\}$ satisfies $\Prob[\bm{w}_{[2]} \in B] = 1$, can be assumed to be symmetric. Also notice that $B$ must be dense in the support, $\mathcal{S}$, of $\bm{w}_{[2]} = (w_1,w_2)$,  because the closure of $B$, $\overline{B}$, is closed and $\Prob[\bm{w}_{[2]} \in \overline{B}\,] = 1$, hence $\mathcal{S} \subseteq \overline{B}$.

\smallskip

\textbf{Part 2.1:} Next we will prove that $\nu$ is constant on $\mathcal{D}_{2,E}$, to do so we will first show that $\nu$ is constant on $\mathcal{D}_{2,\varepsilon}$ for every $0 < \varepsilon < E$. So fix  $0 < \varepsilon < E$, and $\bm{x} = (x,x) \in \mathcal{D}_{2,\varepsilon}$, by assumption B1, there exist $r > 0$ such that the open ball with radius $r$ and center $\bm{x}$, $\mathcal{B}_r(\bm{x})$ is contained in $\mathcal{S}$. We will see that  $\mathcal{B}_r(\bm{x}) \subseteq B$, so further fix $\bm{u} = (u_1,u_2) \in \mathcal{B}_r(\bm{x})$ and note that the set 
$$
D_{n} = \mathcal{B}_r(\bm{x})\cap \B_{1/n}(\bm{u})\cap \{\bm{z} = (z_1,z_2) \in \oD2 : z_j < u_j, \, \forall j \in \{1,2\}\}
$$
is a non-empty open subset of $\mathcal{S}$. Using the fact that $B$ is dense in $\mathcal{S}$ we must have $B \cap D_{n} \neq \emptyset$. Thus, there exist $\bm{u}^{(n)} = (u^{(n)}_{1},u^{(n)}_{2}) \in B \cap D_{n}$ for every $n \geq 1$. Note that defining $\bm{u}^{(n)} _{\bm{\sigma}} = (u^{(n)}_{2},u^{(n)}_{1})$, the sequences $\{\bm{u}^{(n)}\}_{n=1}^{\infty}$ and $\{\bm{u}^{(n)}_{\bm{\sigma}}\}_{n=1}^{\infty}$ satisfy the following:
\begin{itemize}
\item[(a)] Being that $\bm{u}^{(n)} \in \B_{1/n}(\bm{u})$,  
$\bm{u}^{(n)} \to \bm{u} \in \D2 \setminus \Omega_2$ and $\bm{u}^{(n)}_{\bm{\sigma}} \to \bm{u}_{\bm{\sigma}} = (u_2,u_1)\in \D2 \setminus \Omega_2$ with $\Omega_2 = \{(x_1,x_2) \in \Delta_2: x_1+x_2 = 1\}$. As the function $f$ in \eqref{eq:f2} is continuous in $\D2 \setminus \Omega_2$, we obtain
\[
f(\bm{u}^{(n)}) \to f(\bm{u}) \quad \text{ and } f(\bm{u}^{(n)}_{\bm{\sigma}}) \to f(\bm{u}_{\bm{\sigma}}).
\]
\item[(b)] Since $\bm{u}^{(n)} \in \{\bm{z} = (z_1,z_2) \in \oD2: z_j < u_j, \, \forall j \in \{1,2\}\}$, 
\[
f(\bm{u}^{(n)}) = \frac{u^{(n)}_2}{1-u^{(n)}_1} < \frac{u_2}{1-u_1} = f(\bm{u}).
\]
and similarly
\[
f(\bm{u}^{(n)}_{\bm{\sigma}}) = \frac{u^{(n)}_1}{1-u^{(n)}_2} < \frac{u_1}{1-u_2} = f(\bm{u}_{\bm{\sigma}}).
\]
\item[(c)] As $\bm{u}^{(n)} \in B$ and $B$ is symmetric, then $\bm{u}^{(n)}_{\bm{\sigma}} \in B$ and
$
\nu(\bm{u}^{(n)}) = \nu(u^{(n)}_{1},u^{(n)}_{2}) = \nu(u^{(n)}_{2},u^{(n)}_{1}) = \nu(\bm{u}^{(n)}_{\bm{\sigma}}).
$
\end{itemize}
Putting together (a), (b) and (c) above, and using the facts that $\phi$ is left continuous and $\nu = \phi \circ f$, we get
\[
\nu(\bm{u}^{(n)}) \wto \nu(\bm{u}), \quad \nu(\bm{u}^{(n)}_{\bm{\sigma}}) \wto \nu(\bm{u}_{\bm{\sigma}}),\quad \text{ with }\quad  \nu(\bm{u}^{(n)}) = \nu(\bm{u}^{(n)}_{\bm{\sigma}}).
\]
The uniqueness of limits with respect to weak convergence then yields $\nu(\bm{u}) =  \nu(\bm{u}_{\bm{\sigma}})$ which proves $\mathcal{B}_r(\bm{x}) \subseteq B$. The next step is to see that there exist $s > 0$ such that for every $a \in (-s,s)$ we have that $\nu(\bm{x}) = \nu(\bm{x}+a)$ with $\bm{x}+a = (x+a,x+a)$, so set $r' = r/2$ and define $s$ by
\[
s = \frac{r'(1-2x)}{1+r'}
\]
Note that since $\bm{x} \in \mathcal{D}_{2,\varepsilon}$ we get $1-2x > \varepsilon - 2x > 0$, hence $0 < s$. Now for $a$ such that $|a| < s$ define $\bm{u} = (u_1,u_2)$ by
\[
u_1 = x+a\l(1+\frac{x+a}{1-2x-a}\r), \quad  \text{and} \quad u_2 = x\l(1-\frac{a}{1-2x-a}\r).
\]
It is easy to check that (after some algebra) $d(\bm{x},\bm{u}) = [(x_1-u_1)^2+(x_2-u_2)^2]^{-1/2} < r$, this is $\bm{u} \in \mathcal{B}_r(\bm{x})\subseteq B$. It is also easy to verify that
\[
\frac{u_2}{1-u_1} = \frac{x}{1-x} \quad \text{ and } \frac{u_1}{1-u_2} = \frac{x+a}{1-x-a},
\]
i.e. $f((u_1,u_2)) = f(\bm{x})$ and $f((u_2,u_1)) = f(\bm{x}+a)$. Since $\nu$ is symmetric in $B$ and $\bm{u} \in B$ we obtain
\[
\nu(\bm{x}) = \nu((u_1,u_2)) = \nu((u_2,u_1)) = \nu(\bm{x}+a).
\]
This proves that $\nu$ is constant in the set $\{(x+a,x+a): |a| < s\}$. As we chose $\varepsilon < E$ and $\bm{x} = (x,x) \in \mathcal{D}_{2,\varepsilon}$ arbitrarily, we get $\nu$ must be constant in $\mathcal{D}_{2,\varepsilon}$ for every $0 < \varepsilon < E$. The following step is to check that $\nu$ is constant on $\mathcal{D}_{2,E}$, to do so it is enough to prove that $\nu((E/2,E/2)) = \nu((\varepsilon/2,\varepsilon/2))$ for some $\varepsilon < E$. So let $0 < \varepsilon = \varepsilon_1 < \varepsilon_2 < \cdots < E$ such that $\varepsilon_n \to E$, and fix $\bm{x}_n = (\varepsilon_n/2,\varepsilon_n/2)$. Then it is easy to see $f(\bm{x}_n) < f(\bm{x})$ with $\bm{x} = (E/2,E/2)$, hence, as $\phi$ is left continuous, we get $\nu(\bm{x}_n) \wto \nu(\bm{x})$. Being that $\nu$ in constant on $\mathcal{D}_{2,\varepsilon_n}$ for each $n\geq 1$, we must have $\nu(\bm{x}_1) = \nu(\bm{x}_n) = \nu(\bm{x})$ for any $n \geq 1$ as desired.

\smallskip

\textbf{Part 2.2:} Now, we will prove that $\nu$ is constant a.s. Let $\bm{x} = (x_1,x_2) \in B  \cap \oD2 \cap \mathcal{S}$ and set $\gamma = f(\bm{x}) = x_2/(1-x_1)$. First assume $0 < x_2 \leq x_1$,  it easy to see that $z = \gamma/(1+\gamma)$ satisfies $z \leq E/2$ and $z/(1-z) = \gamma$. This is,  $\bm{z} = (z,z)\in \mathcal{D}_{2,E}$ and  $\nu(\bm{z}) = \nu(\bm{x})$. In case $0 < x_1 \leq x_2$, we have that for $\bm{x}_{\bm{\sigma}} = (x_2,x_1) \in B \cap (\D2 \setminus \Omega_2) \cap \mathcal{S}$, it holds $\nu(\bm{z}) = \nu(\bm{x}_{\bm{\sigma}})$ for some $\bm{z} \in \mathcal{D}_{2,E}$. In any case, by definition of $B$ we get $\nu(\bm{x}) = \nu(\bm{x}_{\bm{\sigma}}) = \nu(\bm{z})$ for some $\bm{z} \in \mathcal{D}_{2,E}$ and since $\nu$ is constant in $\mathcal{D}_{2,E}$ it follows that $\nu$ is constant in the set
$
B \cap \oD2 \cap \mathcal{S}.
$
Finally notice that since $w_1+w_2 < 1$ a.s.~then $w_3 > 0$ a.s.~and $w_1,w_2 > 0$ a.s. (by Remark \ref{rem:wk=0}) i.e. $\Prob[\bm{w}_{[k]} \in \oD2] =1$ which yields $\Prob[\bm{w}_{[k]} \in B \cap \oD2 \cap \mathcal{S}] =1$ and we get $\nu$ is constant a.s.

\smallskip

\textbf{Part 2.3:} By the tower property of conditional expectation we have
\[
\Prob[v_3 \in \cdot\,] = \Esp[\Prob[v_3 \in \cdot\mid \bm{w}_{[2]}]] = \Esp[\nu(\bm{w}_{[2]})] = \nu(\bm{x})
\]
for any $\bm{x} \in B \cap \oD2 \cap \mathcal{S}$. That is $\nu \aseq \Prob[v_3 \in \cdot\,]$ which means that the law of $v_3$ is version of $\Prob[v_3 \in \cdot\mid \bm{w}_{[2]}]$ and from \eqref{eq:MSB_v3} we also get it is also a version of $\Prob[v_3 \in \cdot\mid v_2]$ and $\Prob[v_3 \in \cdot\mid \bm{v}_{[2]}]$, proving that $v_3$ is independent of the pair $\bm{v}_{[2]} = (v_1,v_2)$. 

\smallskip

\textbf{Part 3:} The rest of the proof is identical to the proof of Lemma  \ref{lem:keyA} (here, in Part 3, the case 3. $v_2 = c$ and $v_3 = c/(1-c)$ for some $v_1$-measure r.v.~$c= c(v_1)$ cannot occur because this implies the support of $\bm{w}_{[2]}$ is contained in the diagonal contradicting B1).

\end{proof}

\begin{cor}\label{cor:keyB}
Let $\bm{w}$ be stick-breaking weights that are invariant under size-biased permutations with $\sum_{j=1}^{\infty}w_j = 1$, and $w_j > 0$ a.s.~for every $j \geq 1$. Say that
\begin{itemize}
\item[\emph{B0.}] $v_3$ is conditionally independent of $v_1$ given $v_2$, i.e $\Prob[v_3 \in \cdot\mid v_2] \aseq \Prob[v_3 \in \cdot\mid v_1,v_2]$ (this is already given if $\bm{w}$ is an MSB).
\item[\emph{B1.}]  The support $\mathcal{S}$ of $\bm{w}_{[2]}$ is convex.
\item[\emph{B2.}] There is a version $\phi:[0,1] \to \mathcal{P}([0,1])$, of the conditional law $\Prob[v_{3}\in \cdot \mid v_2]$ that is left continuous (w.r.t. the weak topology).
\end{itemize}
Then $v_1,v_2$ and $v_3$ are independent Beta distributed r.v.
\end{cor}

\begin{proof}
First note that since $w_j > 0$ and $\sum_{j=1}^{\infty}w_j = 1$ a.s.~we get $w_1+w_2 < 1$ a.s.~and $\Prob[w_j > 0] > 0$. Now, let $\Delta_{2,\varepsilon} = \{(x_1,x_2) \in \D2: x_1+x_2 \leq \varepsilon\}$ and define $E = \inf\{\varepsilon > 0: \mathcal{S} \subseteq \Delta_{2,\varepsilon}\}$ due to Lemma \ref{lem:keyB} it is enough to prove that for every $0 < \varepsilon < E$ the set $\mathcal{D}_{2,\varepsilon} = \{(x_1,x_2) \in \Delta_{2,\varepsilon}: x_1 = x_2\}$, is contained in the interior of $\mathcal{S}$. To this aim, first we will show that the law of $\bm{w}_{[2]}$ does not concentrates all of the mass in the diagonal. If we assume this holds then we get $w_1 = w_2$ a.s., the invariance under size-biased permutations of $\bm{w}$ the yields that for some possibly random $M \geq 3$ we must have $w_1 = w_2 \ldots = w_M = 1/M$ and $w_j = 0$ for $j > M$. In this instance the support of $\bm{w}_{[2]}$ must be contained in the set $\{(1/m,1/m): m \geq 2\}$ contradicting that $\mathcal{S}$ is convex. We then have that there exist $\bm{x} = (x_1,x_2) \in \mathcal{S}\setminus \mathcal{D}_2$, where $\mathcal{D}_2 = \{(x_1,x_2) \in \D2 : x_1 = x_2\}$ is the diagonal in $\D2$. For such $\bm{x}$ set $\gamma = x_1 + x_2$. Now, we know by Corollary \ref{cor:0} that the origin $\bm{0} = (0,0)$ is contained in $\mathcal{S}$, furthermore by Corollary \ref{cor:Sk_sym} we also have that $\bm{x}_{\bm{\sigma}} = (x_2,x_1)\in \mathcal{S}$. Thus $\{\bm{x},\bm{x}_{\bm{\sigma}},\bm{0}\} \subset \mathcal{S}$ and by condition B1 we must have that the polygon with these vertices in contained in $\mathcal{S}$, it follows that for every $\varepsilon < \gamma$, $\mathcal{D}_{2,\varepsilon}$ is contained in the interior of $\mathcal{S}$. If $\gamma = E$ this proves the result. Otherwise, by definition of $E$, we must have that $\gamma < E$ and for every $\gamma < \varepsilon < E$, there exist $\bm{y} = (y_1,y_2) \in \mathcal{S}$ with $\varepsilon < y_1 + y_2 \leq E$. If $y_1 = y_2$, we get that the polygon with vertices $\{\bm{x},\bm{y},\bm{x}_{\bm{\sigma}},\bm{0}\}$ is contained in $\mathcal{S}$, otherwise, if $y_1 \neq y_2$, then the polygon with vertices $\{\bm{x},\bm{y},\bm{y}_{\bm{\sigma}},\bm{x}_{\bm{\sigma}},\bm{0}\}$ is contained in $\mathcal{S}$. In any case we get $\mathcal{D}_{2,\varepsilon}$ is contained in the interior of $\mathcal{S}$.
\end{proof}

\begin{lem}\label{lem:keyC}
Let $\bm{w}$ be stick-breaking weights are invariant under size-biased permutations with $\sum_{j=1}^{\infty}w_j = 1$, $w_1+w_2 < 1$, a.s., and $\Prob[w_j > 0] > 0$ for every $j \geq 1$. Say that
\begin{itemize}
\item[\emph{C0.}] $v_3$ is conditionally independent of $v_1$ given $v_2$, i.e $\Prob[v_3 \in \cdot\mid v_2] \aseq \Prob[v_3 \in \cdot\mid v_1,v_2]$ (this is already given if $\bm{w}$ is an MSB).
\item[\emph{C1.}]  There exist $0 <\varepsilon < 1$ such that 
the set $C_{\varepsilon} = \{(x_1,x_2) \in \D2: x_2/(1-x_1) = \varepsilon\}$ is contained in the support of $\bm{w}_{[2]} = (w_1,w_2)$. Equivalently, the set $\{(x_1,x_2) \in [0,1]^2: x_2 = \varepsilon\}$ is contained in the support of $\bm{v}_{[2]} = (v_1,v_2)$.
\item[\emph{C2.}] There is a version $\phi:[0,1] \to \mathcal{P}([0,1])$, of the conditional law $\Prob[v_{3}\in \cdot \mid v_2]$ that is continuous (w.r.t. the weak topology). This is, for every $u_1,u_2,\ldots$ in $[0,1]$ such that  $u_n \to u$ we get $\phi(u_n) \wto \phi(u)$. 
\end{itemize}
Then $v_1,v_2$ and $v_3$ are independent Beta distributed r.v.
\end{lem}

\begin{proof}
\textbf{Part 1:} As in the proof of Lemmas \ref{lem:keyA} and \ref{lem:keyB} we get $\nu = \phi \circ f$ is a version of $\Prob[v_3 \in \cdot \mid \bm{w}_{[2]}]$, where $\phi$ is as in C1 and $f$ as in \eqref{eq:f2}. Moreover, the set $B = \{(x_1,x_2) \in \D2: \nu((x_1,x_2)) =  \nu((x_2,x_1))\}$ satisfies $\Prob[\bm{w}_{[2]} \in B] = 1$, it is dense in support, $\mathcal{S}$, of $\bm{w}_{[2]}$, and it can be assumed to be symmetric.

\smallskip

\textbf{Part 2.1:} Next we will prove that the set $C_{\varepsilon} = \{(x_1,x_2) \in \oD2: x_2/(1-x_1) = \varepsilon\}$ is contained in $B$. So fix $\bm{x} = (x_1,x_2) \in C_{\varepsilon}$, so that $f(\bm{x}) = x_2/(1-x_1) = \varepsilon$. Since $B$ is dense in $\mathcal{S}$ and $C_{\varepsilon} \subseteq \mathcal{S}$ by C1 in the statement of the Lemma, we may take a sequence $\bm{x}^{(1)},\bm{x}^{(2)},\ldots$ with $\bm{x}^{(n)} = (x^{(n)}_1,x^{(n)}_2) \in B$ and $\bm{x}^{(n)} \to \bm{x}$. Note that we also have $\bm{x}^{(n)}_{\bm{\sigma}} \to \bm{x}_{\bm{\sigma}}$ with $\bm{x}^{(n)}_{\bm{\sigma}} = (x^{(n)}_2,x^{(n)}_1)$ and $\bm{x}_{\bm{\sigma}} = (x_2,x_1)$. Since $\bm{x}^{(n)},\bm{x}^{(n)}_{\bm{\sigma}} \in B$, the function $f$ in \eqref{eq:f2} is continuous in $\oD2$, $\phi$ is continuous (w.r.t. the weak topology) and $\nu = \phi\circ f$, we get
\[
\nu(\bm{x}^{(n)}) \wto \nu(\bm{x}), \quad \nu(\bm{x}^{(n)}_{\bm{\sigma}}) \wto \nu(\bm{x}_{\bm{\sigma}}), \quad \text{ and } \quad  \nu(\bm{x}^{(n)}) = \nu(\bm{x}^{(n)}_{\bm{\sigma}}).
\]
The uniqueness of limits with respect to weak convergence then yields $\nu(\bm{x}) =  \nu(\bm{x}_{\bm{\sigma}})$ which proves $\bm{x} \in B$ and therefore $C_{\varepsilon} \subseteq B$. 

\smallskip

\textbf{Part 2.2:} Our next aim is to prove that $\nu$ is constant on $B\cap (\D2\setminus\Theta_2)$, where $\Theta_2 = \{(x_1,x_2) \in \Delta_2: x_1+x_2 = 1\}$. So fix $\bm{z} = (z_1,z_2) \in B\cap (\D2\setminus\Theta_2)$ and $\alpha = z_2/(1-z_1)$. Define $\bm{x} = (x_1,x_2)$ by 
\[
x_1 = \frac{\alpha(1-\varepsilon)}{1-\alpha\varepsilon}, \quad x_2 = \frac{\varepsilon(1-\alpha)}{1-\alpha\varepsilon}.
\]
Then, it is easy to check $f(\bm{x}) = x_2/(1-x_1) = \varepsilon$ and $f(\bm{x}_{\bm{\sigma}}) = x_1/(1-x_2) = \alpha$ with $\bm{x}_{\bm{\sigma}} = (x_2,x_1)$. Since $\bm{x} \in C_{\varepsilon}\subseteq B$ and $B$ is symmetric we get
\[
 \nu(\bm{x}) = \nu(\bm{x}_{\bm{\sigma}}) = (\phi\circ f)(\bm{x}_{\bm{\sigma}}) = \phi(\alpha) = (\phi\circ f)(\bm{z})  = \nu(\bm{z}).
\]
That is for every $\bm{z} \in B\cap (\D2\setminus\Theta_2)$ there exist $\bm{x} \in C_{\varepsilon}$ such that $\nu(\bm{z}) = \nu(\bm{x})$. Being that $\nu$ is constant in $C_{\varepsilon}$ by definition of $\phi$, we get $\nu$ is constant in $B\cap (\D2\setminus\Theta_2)$. Now, since $w_1 + w_2 < 1$ a.s., $\Prob[\bm{w}_{[k]} \in \Theta_2] = 0$. Thus, recalling that $\Prob[\bm{w}_{[2]} \in B] = 1$, we have $\Prob[\bm{w}_{[2]} \in B \cap (\D2\setminus\Theta_2)] = \Prob[\bm{w}_{[2]} \in (\D2\setminus\Theta_2)] = 1$ which in turn yields $\nu$ is constant a.s.

\smallskip

\textbf{Part 2.3:} As in Lemma \ref{lem:keyB} this yields $\Prob[v_3 \in \cdot]$ is a version of $\Prob[v_3 \in \cdot\mid v_2]$, $\Prob[v_3 \in \cdot \mid \bm{w}_{[2]}]$ and $\Prob[v_3 \in \cdot\mid \bm{v}_{[2]}]$, so that $v_3$ is independent of $\bm{v}_{[2]}$

\smallskip

\textbf{Part 3:} The rest of the proof is identical to Part 3 in the proof of Lemma \ref{lem:keyB}. Here, the case 3. $v_2 = c$ and $v_3 = c/(1-c)$ for some $v_1$-measure r.v.~$c= c(v_1)$ cannot occur because this implies the support of $\bm{w}_{[2]}$ is contained in the diagonal contradicting C1.

\end{proof}

\subsection{Proof of Theorem 
\ref{theo:MSB_sb_PY}}\label{sec:app:MSB_sb:proof}

\begin{proof}
If the length variables are independent with $v_j \sim \Be(1-\sigma, \theta+j\sigma)$ for some $\sigma \in [0,1)$ and $\theta > -\sigma$, by Lemma \ref{lem:theo2}, $\bm{w}$ is invariant under size-biased permutation. Conversely, assume that $\bm{w}$ is invariant under size-biased permutation. First note that if $\sum_{j=1}^{\infty}w_j = 1$ a.s.~and $w_j > 0$ a.s.~we trivially get $\Prob[w_j > 0] > 0$ and $w_1+w_2 < 1$ a.s. By Lemma \ref{lem:keyA}, Corollary \ref{cor:keyB}, or Lemma \ref{lem:keyC} we know that $v_1$, $v_2$ and $v_3$ are independent. We will show by induction on $k$ that $v_1,\ldots,v_k$ are independent and Beta distributed, so assume that this holds for some $k \geq 3$. Define $\bm{v}' = (v_{k-1},v_k,v_{k+1},\ldots)$ and let $\bm{w}'$ be the corresponding stick-breaking weights i.e. $w'_j = w_{k-2+j}/\big(1-\sum_{i \leq k-2}w_i\big)$. It follows  from the assumption that $v_1,\ldots,v_k$ are independent and Beta distributed, together with the Markov property of $\bm{v}$ that:
\begin{itemize}
\item $\bm{w}'$ remains in variant under size-biased permutations and $\sum_{j=1}^{\infty} w'_j = 1$ a.s. (see Remark \ref{rem:rem?} and Lemma \ref{lem:rem8}).
\item $\bm{v}'$ is a Markov chain. In particular $v'_3 = v_{k+1}$ is conditionally independent of $v'_1 = v_{k-1}$ given $v'_2 = v_k$.
\item The law of $\bm{v}'_{[2]} = (v'_1,v'_2)$ is mutually absolutely continuous with respect to the Lebesgue measure.
\end{itemize}
By Lemma \ref{lem:keyA} we have that $v'_1 = v_{k-1}$, $v'_2 = v_k$ and $v'_3 = v_{k+1}$ are independent and Beta distributed. Hence $v_1,\ldots,v_{k+1}$ are independent and Beta distributed. This proves that for every $j \geq 1$, $v_j \ind \Be(\alpha,\beta_j)$ for some constants $\alpha_j,\beta_j > 0$. The fact that $\alpha = 1-\sigma$  and $\beta_j = \theta +  j \sigma$ for some $\sigma \in [0,1)$ and $\theta > -\sigma$ follows from Lemma \ref{lem:theo2}.
\end{proof}

\subsection{Further remarks on Theorem 
\ref{theo:MSB_sb_PY}}\label{sec:app:MSB_sb:remarks}

Note that another version of Theorem 
\ref{theo:MSB_sb_PY} 
can be derived by replacing B1 in its statement with B1 in the statement of Lemma \ref{lem:keyB}. Moreover it follows from Lemmas \ref{lem:keyA}, \ref{lem:keyB} and \ref{lem:keyC} that it is possible to relax the hypothesis $w_j > 0$ a.s.~for every $j \geq 1$, and instead request $\Prob[w_j > 0] > 0$ for $j \geq 1$ and $w_1+w_2 < 1$ a.s. Under this consideration, we now provide a counterexample that shows that not all MSBw that are invariant under size-biased permutations are the size-biased permuted Pitman-Yor weights. So let $M$ be some r.v.~taking values in $\{3,4,\ldots\}$ such that $\Prob[M = m] > 0$ for every $m \geq 3$ and define the weights, $\bm{w}$, by $w_j = 1/M$ if $j \leq M$ and $w_j = 0$ if $j > M$. The length variables, $\bm{v}$, of $\bm{w}$ are easily seen to be $v_j = 1/(M-j+1)$ for $j \leq M$ and $v_j = 1$ for $j > M$. Note that $\bm{w}$ satisfies the following:
\begin{itemize}
\item $\bm{w}$ is invariant under size-biased permutations because any size-biased permutation of $\bm{w}$ will be identical to $\bm{w}$ 
\item $\bm{w}$ is a MSBw with $v_1 = 1/M$ and transition probability kernel
\[
\Prob[v_{j+1} \in \cdot \mid v_j] = \begin{cases}
\delta_{v_j/(1-v_j)} & \text{ if } v_j < 1\\
\delta_1 & \text{ if } v_j = 1
\end{cases}
\]
\item $\sum_{j=1}^{\infty}w_j = \sum_{j=1}^{M} 1/M = 1$, $w_1 + w_2 < 1$ a.s.~because $M \geq 3$ a.s., and $\Prob[w_j > 0] > 0$ for every $j \geq 1$ because $\Prob[M \leq j] > 0$.
\end{itemize}
However, the support of $\bm{w}_{[2]}$ is
\[
\mathcal{S} = \{(1/m,1/m) \in \D2: m \geq 3\}
\]
In particular, $\mathcal{S}$ is contained in the diagonal of $\D2$ so A1, B1 and C1 in the statement of Lemmas \ref{lem:keyA}, \ref{lem:keyB} and \ref{lem:keyC} do not hold. 

\medskip

\section{Proofs of the results in Section \ref{sec:examples}}\label{sec:app:examples}

\smallskip

\subsection{Proof of Theorem \ref{theo:Upsilon}}\label{app:theo:Upsilon}

\begin{proof}
(i.a) Let $\varepsilon$, $J$ and $(a_j)_{j=1}^{\infty}$ be as in the statement of (i.a). First note that since $F_j$ is strictly increasing and continuous, so is $F_j^{-1}$ for every $j \geq 1$. This yields $\Upsilon_j$ and $\Upsilon^{[j]}$ will also be strictly increasing and continuous, because they are compositions of such functions. This said, let $v \in (0,\varepsilon)$ and note that the hypothesis implies
\begin{equation}\label{eq:Upsilon_i.a}
\Upsilon_j(v) \geq v\frac{a_{j+1}}{a_j},
\end{equation}
for $j \geq J$. Define $v^* = \Upsilon^{[J]}(v)/a_J > 0$, we will prove by induction that 
\begin{equation}\label{eq:Upsilon_v*}
\Upsilon^{[j]}(v) \geq v^* a_{j}.
\end{equation}
for each $j \geq J$. So assume that \eqref{eq:Upsilon_v*} holds for $j \geq J$. Since $\Upsilon_j$ is increasing and by \eqref{eq:Upsilon_i.a}, this yields
\[
\Upsilon^{[j+1]}(v) = \Upsilon_{j}\big(\Upsilon^{[j]}(v)\big) \geq \Upsilon_j(v^* a_{j}) \geq v^* a_{j+1}.
\]
This together with the fact that $\Upsilon^{[J]}(v) =v^* a_{J}$ prove that \eqref{eq:Upsilon_v*} holds for every $j \geq J$. Thus,
\[
\sum_{j = 1}^{\infty}\Upsilon^{[j]}(v) \geq \sum_{j=J}^{\infty}\Upsilon^{[j]}(v) \geq v^*\sum_{j=J}^{\infty}a_j = \infty
\]
because $\sum_{j=1}^{\infty}a_j = \infty$. Being that $\sum_{j = 1}^{\infty}\Upsilon^{[j]}(v) = \infty$ holds for every $0 < v < \varepsilon$ and $\Upsilon^{[j]}$ increasing we get that it also holds for every $v \in (0,1)$. Finally, since $\pi_1 = \pi$ is diffuse we have proven 
\eqref{eq:sum_1_Upsilon} 
in the main document which in turn implies $\sum_{j=1}^{\infty}w_j = 1$.

(i.b) By hypothesis and Monotone convergence theorem, we get $\Esp[\sum_{j=1}^{\infty}v_j] < \infty$, in particular we must have $\sum_{j=1}^{\infty}v_j < \infty$ a.s., thus $\sum_{j=1}^{\infty} w_j < 1$ a.s.

(ii) By hypothesis we have $\Upsilon_j(v) \leq v$, which yields 
\[\Upsilon^{[j+1]}(v) = \Upsilon_j(\Upsilon^{[j]}(v)) \leq \Upsilon^{[j]}(v),
\] 
for each $j \geq 1$ and $v \in [0,1]$. Also note that since $\Upsilon^{[j]}$ is continuous and strictly increasing 
\[
\Upsilon^{[j]}(v) > 0, \text{ for } v > 0.
\] 
Now, let $\bm{v}$ be the length variables of $\bm{w}$ so that $v_1 \sim \pi = \pi_1$ and $v_{j+1} = \Upsilon_j(v_j)$ for $j \geq 1$. Equivalently, we can write $v_j = \Upsilon^{[j]}(v)$ with $v = v_1 \sim \pi$. Being that $\pi$ is diffuse we know $v > 0$ a.s., thus it follows from the stick breaking decomposition and the aforementioned annotations that a.s.
\[
w_{j+1} =\frac{\Upsilon^{[j+1]}(v)\{1-\Upsilon^{[j]}(v)\}}{\Upsilon^{[j]}(v)}w_j \leq w_j,
\]
for each $j \geq 1$. This proves that $\bm{w}$ is decreasing a.s. To prove the second statement of (ii) we first mention that for proper species sampling processes, the property of having full support is equivalent to 
\[
\Prob\l[\max_{j \geq 1}w_j < \varepsilon\r] > 0,
\]
for every $\varepsilon > 0$ \citep[Corollary 3 in][]{BO14}. As we have already shown the CDSBw, $\bm{w}$, in question are a.s.~decreasing, which means $v_1 = w_1 = \max_{j \geq }w_j$ a.s. Being that the distribution function, $F_1$, is continuous and strictly increasing we obtain $\Prob[v_1 < \varepsilon] = F_1(\varepsilon) > 0$, for every $\varepsilon > 0$, which proves $\bm{P}$ has full support.
\end{proof}

\subsection{Proof of Corollary \ref{cor:CDSB_PY}}\label{app:cor:CDSB_PY}

\begin{lem}\label{lem:A}
Let $0 < \alpha \leq 1$ and $\alpha +\beta >2$, then $v \mapsto \cl{I}_{v}(\alpha,\beta)$ is concave.
\end{lem}

\begin{proof}
It can be easily seen that
\[
\frac{\partial^2 \cl{I}_v(\alpha,\beta)}{\partial v^2} = \frac{v^{\alpha-2}(1-v)^{\beta-2}}{\cl{B}(\alpha,\beta)}\{(\alpha-1)(1-v)-(\beta-1)v\},
\]
which is non-positive if and only if $(\alpha-1)/(\beta+\alpha-2) \leq v$. This holds for every $v \in (0,1)$, as $(\alpha-1)/(\beta+\alpha-2) \leq 0$. Thus $v \mapsto \cl{I}_{v}(\alpha,\beta)$ is concave.
\end{proof}

\begin{lem}\label{lem:B}
Let $0 < \alpha \leq 1$ and $\beta \geq 1$. Then for every $n \in \{1,2,...\}$ and $v \in (0,1)$
\[
\cl{I}_v(\alpha,\beta+n) \geq \cl{I}_{nv/(n+1)}(\alpha,\beta+n+1).
\]
\end{lem}

\begin{proof}
Fix $v \in (0,1)$ and $n \in \{1,2,...\}$. By the mean value theorem we know that there exist $u$ satisfying $nv/(n+1) < u < v$, such that 
\[
\frac{\partial \cl{I}_x(\alpha,\beta+n+1)}{\partial x}\bigg|_u = \cl{I}'_u(\alpha,\beta+n+1) = \frac{\cl{I}_v(\alpha,\beta+n+1)-\cl{I}_{nv/(n+1)}(\alpha,\beta+n+1)}{v-nv/(n+1)}.
\]
By Lemma \ref{lem:A} we have that $x \mapsto \cl{I}_x(\alpha,\beta+n+1)$ is concave, which implies $\cl{I}'_u(\alpha,\beta+n+1) \geq \cl{I}'_v(\alpha,\beta+n+1)$. That is
\[
\frac{\cl{I}_v(\alpha,\beta+n+1)-\cl{I}_{nv/(n+1)}(\alpha,\beta+n+1)}{v
(n+1)^{-1}} \geq \cl{I}'_v(\alpha,\beta+n+1) = \frac{v^{\alpha-1}(1-v)^{\beta+n}}{\cl{B}(\alpha,\beta+n+1)},
\]
where $\cl{B}$ denotes the beta function. Evidently, $(\alpha+\beta+n) \geq (n+1)$, thus
\[
(\alpha+\beta+n)\l\{\frac{\cl{I}_v(\alpha,\beta+n+1)-\cl{I}_{nv/(n+1)}(\alpha,\beta+n+1)}{v}\r\} \geq \frac{v^{\alpha-1}(1-v)^{\beta+n}}{\cl{B}(\alpha,\beta+n+1)}.
\]
Recalling that $\cl{B}(a,b+1) = b\cl{B}(a,b)/(a+b)$ for $a,b >0$, the above equation can be written as
\[
\cl{I}_v(\alpha,\beta+n+1)-\cl{I}_{nv/(n+1)}(\alpha,\beta+n+1) \geq \frac{v^{\alpha}(1-v)^{\beta+n}}{(\beta+n)\cl{B}(\alpha,\beta+n)}.
\]
Further, recalling that $\cl{I}_v(a,b+1) = \cl{I}_v(a,b) + \{v^{a}(1-v)^{b}\}/\{b\cl{B}(a,b)\}$, for $a,b >0$, we obtain
\[
\cl{I}_v(\alpha,\beta+n) \geq \cl{I}_{nv/(n+1)}(\alpha,\beta+n+1).
\]
\end{proof}

\begin{lem}\label{lem:C}
Let $\bm{P}$ be a CDSBp with parameters $\bm{\pi} = (\pi_j)_{j=1}^{\infty}$ where $\pi_j = \Be(\alpha,\theta+j\sigma)$ for some $0 < \alpha \leq 1$, $0 \leq \sigma \leq 1$ and $\theta > -\sigma$. Then $\bm{P}$ is proper and it has full support. Moreover the CDSBw, $\bm{w}$, of $\bm{P}$ satisfy $w_{j} > 0$ and $w_{j+1}\leq w_j$ for every $j \geq 1$, a.s.
\end{lem}

\begin{proof}
As $\pi_j$ is diffuse, it follows immediately from Proposition 
\ref{prop:MSB_supp_points} 
that $w_j > 0$ for every $j \geq 1$, a.s. Next we prove that $\bm{w}$ is decreasing, by showing that the condition in Theorem (ii) holds for $F_j(v) = \mathcal{I}_v(1-\sigma,\theta+j\sigma)$.

As shown by \cite{Karp16}, for $v \in [0,1]$ and $\alpha > 0$ fixed, the mapping $\beta \mapsto \cl{I}_v(\alpha,\beta)$ is log-concave, which implies it is quasi-concave, that is for every $\beta_1,\beta_2 > 0$ and $\lambda \in [0,1]$
\begin{equation}\label{eq:quasi_concave}
\cl{I}_v(\alpha,\lambda\beta_1+(1-\lambda)\beta_2) \geq \min\{\cl{I}_v(\alpha,\beta_1),\cl{I}_v(\alpha,\beta_2)\}.
\end{equation}
Further, it is a well-known property of the regularized Beta function that
\begin{equation}\label{eq:reg_beta}
\cl{I}_v(\alpha,\beta + 1) = \cl{I}_v(\alpha,\beta) + \frac{v^{\alpha}(1-v)^{\beta}}{\beta\cl{B}(\alpha,\beta)} > \cl{I}_v(\alpha,\beta),
\end{equation}
where $\cl{B}(\alpha,\beta) = \Gamma(\alpha+\beta)/[\Gamma(\alpha)\Gamma(\beta)]$ denotes the Beta function. Hence by \eqref{eq:quasi_concave} and \eqref{eq:reg_beta} we obtain that for every $\beta > 0$ and $\varepsilon \in [0,1]$,
\[
\cl{I}_v(\alpha,\beta+\varepsilon) \geq  \min\{\cl{I}_v(\alpha,\beta),\cl{I}_v(\alpha,\beta+1)\} = \cl{I}_v(\alpha,\beta).
\]
That is $\beta \mapsto \cl{I}_v(\alpha,\beta)$ is monotonically increasing, particularly
\[
F_j(v) = \cl{I}_v(\alpha,\theta+j\sigma) \leq \cl{I}_v(\alpha,\theta+(j+1)\sigma) = F_{j+1}(v). 
\]
Thus for $\pi_j = \Be(\alpha,\theta+j\sigma)$, Theorem 
\ref{theo:Upsilon} 
shows $\bm{w}$ is decreasing a.s.~and if $\bm{P}$ is proper it also has full support. Thus, to finish the proof it only remains to prove $\bm{P}$ is proper. 

If $\sigma = 0$, $\bm{P}$ is a Geometric process and the result obvious. Otherwise, there exists $J \geq 1$ such that $\theta+J\sigma \geq 1$. Define $\hat{\theta} = \theta +J \sigma$, set $\hat{F}_j(v) = \mathcal{I}_v(\alpha,\hat{\theta}+j)$ and $\hat{\Upsilon}_j = \hat{F}_{j+1}^{-1}\circ \hat{F}_j$. By Lemma \ref{lem:B} we know $\hat{F}_j(v) \geq \hat{F}_{j+1}(v \, j/(j+1))$, hence by Theorem 
\ref{theo:Upsilon} 
(i.a), 
\begin{equation}\label{eq:hat_Upsilon}
\sum_{j=1}^{\infty} \hat{\Upsilon}^{[j]}(\hat{v}) = \infty,
\end{equation}
for every $\hat{v} \in (0,1)$, with $\hat{\Upsilon}^{[j]} = \hat{F}^{-1}_j \circ \hat{F}_1$. Now, since $b \mapsto \mathcal{I}_{\hat{v}}(\alpha,b)$, is monotonically increasing, $b \mapsto \mathcal{I}^{-1}_{\hat{v}}(\alpha,b)$ is monotonically decreasing and we get 
\[
\mathcal{I}^{-1}_{\mathcal{I}_{\hat{v}}(\alpha,\hat{\theta})}(\alpha,\hat{\theta}+j \sigma) \geq\mathcal{I}^{-1}_{\mathcal{I}_{\hat{v}}(\alpha,\hat{\theta})}(\alpha,\hat{\theta}+j) = \hat{\Upsilon}^{[j]}(\hat{v}).
\]
The choice $\hat{v} = \mathcal{I}^{-1}_{\mathcal{I}_{v}(\alpha,\theta+\sigma)}(\alpha,\hat{\theta})$ with $v \in (0,1)$ yields
\[
\Upsilon^{[J+j]}(v) = \big(F_{J+j}^{-1}\circ F_1\big)(v) \geq \hat{\Upsilon}^{[j]}(\hat{v})
\]
for every $j \geq 1$, where $F_j(v) = \mathcal{I}_v(\alpha,\theta+j\sigma)$. Hence, by \eqref{eq:hat_Upsilon} we obtain
\[
\sum_{j=1}^{\infty} \Upsilon^{[j]}(v) \geq \sum_{j=1}^{\infty}\Upsilon^{[J+j]}(v) \geq \sum_{j=1}^{\infty} \hat{\Upsilon}^{[j]}(\hat{v}) = \infty,
\]
for $v \in (0,1)$. This proves 
\eqref{eq:sum_1_Upsilon}
holds a.s.~for $v \sim \Be(\alpha,\theta+\sigma)$, which in turn implies $\sum_{j=1}^{\infty}w_j = 1$, i.e. $\bm{P}$ is proper.
\end{proof}

Corollary 
\ref{cor:CDSB_PY}
is simply a particular case of Lemma \ref{lem:C}.

\medskip

\section{Proofs of the results in Section 
\ref{sec:BBSB}}\label{sec:app:BMSB_sup}

\smallskip

\subsection{Marginal distribution of the length variables}\label{app:prop:BB_trans_marg}

\begin{prop}\label{prop:BB_trans_marg}
Let $\bm{v}$ be a Markov process with initial distribution $\pi = \pi_1 = \Be(\alpha_1,\beta_1)$ and Beta-Binomial transitions $\bm{\psi}$ as in 
\eqref{eq:BB_trans}
 with parameters $(N,\bm{\alpha},\bm{\beta})$. Then 
$v_j \sim \Be(\alpha_j,\beta_j)$ for every $j \geq 1$.
\end{prop}

\begin{proof}
Assume that $v_j \sim \Be(\alpha_j,\beta_j)$. Then $F_j(v_j) \sim \Un(0,1)$ recalling that $F_j$ is the distribution function of a $\Be(\alpha_j,\beta_j)$ distribution, which yields $\Upsilon_j(v_j) = F^{-1}_{j+1}(F_j(v_j)) \sim \Be(\alpha_{j+1},\beta_{j+1})$. Now, since $v_{j+1}\mid v_j \sim \psi_j(v_j,\cdot)$ where $\psi_j$ is the Beta-Binomial transition at time $j$, we might choose a random element $z_j$ such that $z_j \mid v_j \sim \Bin(N,\Upsilon_j(v_j))$ and $v_{j+1}\mid z_j \sim \Be(\alpha_{j+1}+z_j,\beta_{j+1} +N -z_j)$. This way, we have that  $z_j \mid \Upsilon(v_j) \sim \Bin(N,\Upsilon(v_j))$, where $\Upsilon_j(v_j) \sim \Be(\alpha_{j+1},\beta_{j+1})$. By Beta-Binomial conjugate model we have that $\Upsilon_j(v_j)\mid z_j \sim \Be(\alpha_{j+1}+z_j,\beta_{j+1} +N -z_j)$, identically as $v_{j+1}$. Thus we must have that $v_{j+1} \sim \Be(\alpha_{j+1},\beta_{j+1})$. This induction step together with the fact $v_1 \sim \Be(\alpha_1,\beta_1)$ prove the result.
\end{proof}

\subsection{Proof of Theorem 
\ref{theo:BMSBp}}\label{app:theo:BMSBp}

\begin{proof}
(ii) and (iii) are straight forward from Propositions 
\ref{prop:MSB_supp_points}
and 
\ref{prop:MSB_full_supp} 
 because Beta distributions are diffuse and their support is the interval $[0,1]$. Hence, it is enough to prove (i). Let $\bm{v}$ be the underlying Markov process of length variables with the Beta-Binomial transitions. Note that for fixed $j$ we can choose a random element $z_j$ such that $z_j \mid v_j \sim \Bin(N,\Upsilon_j(v_j))$ and $v_{j+1}\mid z_j \sim \Be(\alpha_{j+1}+z_j,\beta_{j+1} +N -z_j)$ independently of $v_j$. This way, we can easily compute
\[
\Esp[v_{j+1}\mid v_j] = \Esp[\Esp[v_{j+1}\mid z_j]\mid v_j] = \Esp\l[\frac{\alpha_{j+1}+z_j}{\alpha_{j+1}+\beta_{j+1}+N}\mi v_j\r] = \frac{\alpha_{j+1}+N\Upsilon_j(v_j)}{\alpha_{j+1}+\beta_{j+1}+N},
\]
which yields
\[
\Esp[(1-v_{j+1})\mid v_j] = \frac{\beta_{j+1}+N(1-\Upsilon_j(v_j))}{\alpha_{j+1}+\beta_{j+1}+N}\leq \frac{\beta_{j+1}+N}{\alpha_{j+1}+\beta_{j+1}+N}.
\]
By the Markov property of the length variables $\bm{v}$ we obtain
\begin{align*}
\Esp\l[\prod_{j=1}^{m}(1-v_j)\r]&  = \Esp\l[\Esp\l[\prod_{j=1}^{m-1}(1-v_j)\mi v_{m-1}\r]\Esp[(1-v_m)\mid v_{m-1}]\r]\\
& \leq \Esp\l[\prod_{j=1}^{m-1}(1-v_j)\r]\frac{\beta_{m}+N}{\alpha_{m}+\beta_{m}+N},
\end{align*}
and a simple induction argument proves
\[
0 \leq \Esp\l[\prod_{j=1}^{m}(1-v_j)\r] \leq \frac{\beta_1}{\alpha_1+\beta_1}\prod_{j=2}^{m}\frac{\beta_{j}+N}{\alpha_{j}+\beta_{j}+N}.
\]
Next we want to prove that the right side diverges to zero. For simplicity set $\gamma_j = \alpha_{j}/(\alpha_{j}+\beta_{j})$ and $c_j = \alpha_{j}/(\alpha_{j}+\beta_{j}+N)$, then it is enough to show $\sum_{j=1}^{\infty}c_j = \infty$. We know $\sum_{j=1}^{\infty}\gamma_j = \infty$, if  $\sum_{j=1}^{\infty}c_j < \infty$, by the converse of the one-sided comparison test for series, necessarily
\[
\infty = \limsup_{j \to \infty}\frac{\gamma_j}{c_j} = 1+N\limsup_{j \to \infty}\frac{1}{\alpha_j+\beta_j},
\]
which contradicts the hypothesis $\limsup_{j\to \infty}(\alpha_j+\beta_j)^{-1}< \infty$. Thus $\sum_{j=1}^{\infty}c_j = \infty$ and we obtain $\Esp\big[\prod_{j=1}^{m}(1-v_j)\big] \to 0$. Finally, being that this product is positive and bounded by one we get $1-\sum_{j=1}^m w_j = \prod_{j=1}^{m}(1-v_j) \to 0$ a.s.~and the result follows.
\end{proof}

\subsection{Proof of Theorem 
\ref{cor:BMSBp_conv}}\label{app:cor:BMSBp_conv}

To prove the main result we require some preliminary Lemmas.

\begin{lem}[P\'olya's theorem]\label{lem:Polya_theo}
Let $F_n$ be distribution functions that converges pointwise to a continuous distribution function $F$. Then $F_n \to F$ uniformly. That is
\[
\lim_{n \to \infty}\sup_{x \in [0,1]}|F_n(x)-F(x)| = 0
\]
\end{lem}

\begin{proof}
Fix $k \in \{1,2,\ldots\}$ since $F$ is a continuous distribution function we can find $-\infty = x_0 < x_1 < \cdots < x_k = \infty$ such that $F(x_i) = i/k$. Now for each $x \in \R$ we get there exist $i \in \{1,\ldots,k\}$ such that $x_{i-1} \leq x \leq x_{i}$. Hence, the fact that $F_n$ and $F$ are non-decreasing yield 
\[
F_n(x)-F(x) \leq F_n(x_i) - F(x_{i-1}) = F_n(x_i)-F(x_i)+1/k
\]
and
\[
F_n(x)-F(x) \geq F_n(x_{i-1}) - F(x_{i}) = F_n(x_{i-1})-F(x_{i-1})-1/k.
\]
Thus, by the pointwise convergence $F_n(x_i) \to F_(x_i)$ for each $i$, we get
\[
\lim_{n\to \infty}\sup_{x \in [0,1]}|F_n(x)-F(x)| \leq \lim_{n\to \infty}\max_{i \in \{0,\ldots,k\}}|F_n(x_i)-F(x_i)| +1/k = 1/k.
\]
Since this equality holds for each $k \in \{1,2,\ldots\}$. The result follows trivially.
\end{proof}

\begin{lem}\label{lem:Gn-1_G-1}
Let $G_n,G$ be continuous and strictly increasing distribution functions supported on $[0,1]$ and such that $G_n \to G$ uniformly. Then the inverse functions $G_n^{-1},G^{-1}:[0,1]\to[0,1]$ exist and $G_n^{-1} \to G^{-1}$ uniformly.
\end{lem}

\begin{proof}
The existence of the inverse function is obvious because $G_n$ and $G$ must be bijective on $[0,1]$. Furthermore, $G_n^{-1}$ and $G^{-1}$ are also strictly increasing and continuous distribution functions.

Fix $\varepsilon > 0$ and $x \in (0,1)$. First note that since $G^{-1}$ is continuous and increasing there exits $\delta > 0$ such that $0< x-\delta < x+\delta <1$ and
\[
G^{-1}(x)-\varepsilon \leq G^{-1}(x -\delta) < G^{-1}(x) < G^{-1}(x+\delta) \leq G^{-1}(x)+\varepsilon.
\]
Now, as $G_n \to G$, by P\'olya's theorem we also have $G_n \to G$ uniformly thus there exist $N\geq 1$ such that for every $n \geq N$
\[
\sup_{y \in [0,1]}|G_n(y)-G(y)| < \delta.
\]
In particular for $n \geq N$. choosing $y = G^{-1}(x-\delta)$ and later $y = G^{-1}(x+\delta)$ we obtain
\[
|G_n(G^{-1}(x-\delta))-(x-\delta)| < \delta \text{ and } \quad |G_n(G^{-1}(x+\delta))-(x+\delta)| < \delta
\]
which yields
\[
G_n(G^{-1}(x-\delta)) < x < G_n(G^{-1}(x+\delta))
\]
and since $G_n^{-1}$ increasing we obtain
\[
G^{-1}(x-\delta) < G_n^{-1}(x) < G^{-1}(x+\delta).
\]
Putting together this equation with the first one in the proof we get
\[
|G_n^{-1}(x) < G^{-1}(x)| < \varepsilon
\]
for all $n \geq N$. This proves that $G_n^{-1} \to G^{-1}$ pointwise and the uniform convergence follows from P\'olya's theorem.
\end{proof}

\begin{lem}\label{lem:Upsilon_UC}
Let $\pi_n,\pi,\lambda_n,\lambda$, be probability measures of over $([0,1],\B_{[0,1]})$ such that the corresponding distribution functions $F_n,F,G_n,G$ are continuous and strictly increasing. Define $\Upsilon_n = G_n^{-1}\circ F_n$ and $\Upsilon = G^{-1} \circ F$. If $\pi_n \wto \pi$ and $\lambda_n \wto \lambda$, then
\[
\lim_{n \to \infty}\sup_{x \in [0,1]}|\Upsilon_n(x)-\Upsilon(x)| = 0
\]
In particular, for every $u_n \to u$ in $[0,1]$ we get $\Upsilon_n(u_n) \to \Upsilon(u)$.
\end{lem}

\begin{proof}
First note that by hypothesis $G_n \to G$ and $F_n \to F$ pointwise, and by Lemmas \ref{lem:Polya_theo} and \ref{lem:Gn-1_G-1} we get $F_n \to F$ and $G_n^{-1} \to G^{-1}$ uniformly. Now fix $x \in [0,1]$ and note that
\begin{align*}
|\Upsilon_n(x)-\Upsilon(x)| &\leq |G_n^{-1}(F_n(x))-G^{-1}(F_n(x))| + |G^{-1}(F_n(x))- G^{-1}(F(x))|\\
&\leq \sup_{y \in [0,1]}|G_n^{-1}(y)-G^{-1}(y)| + |G^{-1}(F_n(x))- G^{-1}(F(x))|.
\end{align*}
As $G_n^{-1} \to G^{-1}$ uniformly, $F_n \to F$, and $G^{-1}$ is continuous, by making $n \to \infty$, we obtain $\Upsilon_n \to \Upsilon$ pointwise. Noting that $\Upsilon_n$ and $\Upsilon$ are themselves continuous and increasing distribution functions we also have by P\'olya's theorem that $\Upsilon_n \to \Upsilon$ uniformly. In particular if $u_n \to u$ in $[0,1]$ we get
\begin{align*}
|\Upsilon_n(u_n)-\Upsilon(u)| &\leq |\Upsilon_n(u_n)-\Upsilon(u_n)| + |\Upsilon(u_n)- \Upsilon(u)|\\
&\leq \sup_{v \in [0,1]}|\Upsilon_n(v)-\Upsilon(v)| + |\Upsilon(u_n)- \Upsilon(u)|,
\end{align*}
and the result follows by taking limits as $n \to \infty$.
\end{proof}

\begin{lem}\label{lem:Bin_L2}
For every $n \geq 1$ consider a Binomial r.v.~$z_n\sim \Bin(n,p_n)$ where $p_n \to p$ in $[0,1]$.
Then, as $n \to \infty$, $z_n/n \to p$ in $\mathcal{L}_2$.
\end{lem}

\begin{proof}
For $n \geq 1$, 
\begin{equation}\label{eq:bin_L2}
\begin{aligned}
\Esp\l[\l(\frac{z_n}{n}-p\r)^2\r] & = \frac{1}{n^2}\Esp\l[z_n^2\r] -\frac{2p}{n}\Esp[z_n] + p^2\\
& = \frac{p_{n}(1-p_{n})}{n} + (p_{n}-p)^2.
\end{aligned}
\end{equation}
By taking limits as $n \to \infty$ in \eqref{eq:bin_L2} we obtain
\begin{equation*}
\lim_{n \to \infty} \Esp\l[\l(\frac{z_n}{n}-p\r)^2\r] = 0.
\end{equation*}
\end{proof}
		
\begin{lem}\label{lem:BB_trans_conv}
For $n \in \{0,1,\ldots\}$, let $N^{(n)}$, $\bm{\alpha}^{(n)} = \big(\alpha^{(n)}_j\big)_{j=1}^{\infty}$ and $\bm{\beta}^{(n)} = \big(\beta^{(n)}_j\big)_{j=1}^{\infty}$ be as in satisfying (H0), (H1) in the main paper. Let $\bmpsi^{(n)} = \big(\psi^{(n)}_j\big)_{j=1}^{\infty}$ be Beta-Binomial transitions with parameters $\big(N^{(n)},\bm{\alpha}^{(n)},\bm{\beta}^{(n)}\big)$. Then for every $u_n \to u$ in $[0,1]$, and $j \geq 1$ $\bpsi_j^{(n)}(u_n,\cdot)$ converges weakly to $\delta_{\Upsilon_j(u)}$, where $\Upsilon_j$ is as in 
\eqref{eq:Upsilon}
with $F_j(v) = \mathcal{I}_v(\alpha_j,\beta_j)$.
\end{lem}

\begin{proof}
Fix $n,j \geq 1$ and let $\Upsilon_j^{(n)} = F_{n,j+1}^{-1}\circ F_{n,j}$ where $F_{n,j}$ is the distribution function of a $\Be\big(\alpha^{(n)}_j,\beta^{(n)}_j\big)$ distribution. Similarly let $\Upsilon_j = F_{j+1}^{-1}\circ F_{j}$ with $F_j(v) = \mathcal{I}_v(\alpha_j,\beta_j)$. Now, fix $j \geq 1$ and $u_{n} \to u \in [0,1]$, by hypothesis and Lemma \ref{lem:Upsilon_UC} we get $\Upsilon^{(n)}_j(u_n) \to \Upsilon_j(u)$, hence by Lemmas \ref{lem:coup} and \ref{lem:Bin_L2} we might construct on some probability space $\l(\Omega',\cl{F}',\Prob'\r)$ some r.v.~$z'^{(n)}\sim \Bin(n,\Upsilon_j(u_{n}))$ such that $z'^{(n)}/n \to \Upsilon_j(u)$ a.s.~as $n \to \infty$. Consider $v'^{(n)}$ such that
\[
\Prob\l[v'^{(n)}\in \mrm{d}v\mi z'^{(n)}\r] = \Be\l(\mrm{d}v\mi \alpha_{j+1}^{(n)}+z'^{(n)},\beta_{j+1}^{(n)}+n-z'^{(n)}\r),
\]
so that marginally $v'^{(n)} \sim \bpsi_j^{(n)}(u_n,\cdot)$. The moment generator function of $v'^{(n)}$  given $z'^{(n)}$ is 
\begin{equation}\label{eq:mgf_v|x}
\Esp'\l[e^{tv'^{(n)}}\mi z'^{(n)}\r] =  1 + \sum_{m=1}^{\infty}\l(\prod_{r=0}^{m-1}\frac{\alpha_{j+1}^{(n)}+z'^{(n)}+r}{\alpha^{(n)}_{j+1}+\beta_{j+1}^{(n)}+n+r}\r)\frac{t^m}{m!},
\end{equation}
for every $t \in \R$. By construction we have that $z'^{(n)}/n \to \Upsilon_j(u)$ a.s.~and by hypothesis $\alpha_{j+1}^{(n)}\to \alpha_{j+1}$ and $\beta_{j+1}^{(n)}\to \beta_{j+1}$ in $(0,\infty)$, which means that for every $r \geq 0$,
\begin{equation}\label{eq:mgf_v|x_lim}
\frac{\alpha_{j+1}^{(n)}+z'^{(n)}+r}{\alpha_{j+1}^{(n)}+\beta_{j+1}^{(n)}+n+r} = \l(\frac{\alpha_{j+1}^{(n)}+r}{n}+\frac{z'^{(n)}}{n}\r)\l(\frac{\alpha_{j+1}^{(n)}+\beta_{j+1}^{(n)}+r}{n}+1\r)^{-1} \to \Upsilon_j(u),
\end{equation}
a.s.~as $n \to \infty$. Hence by the tower property of conditional expectation, equations \eqref{eq:mgf_v|x} and \eqref{eq:mgf_v|x_lim}, and Lebesgue dominated convergence theorem (the corresponding functions are dominated by $e^t$) we obtain
\begin{align*}
\lim_{n \to \infty} \Esp'\l[e^{tv'^{(n)}}\r] & = \lim_{n \to \infty} \Esp'\l[\Esp'\l[e^{tv'^{(n)}}\mi z'^{(n)}\r]\r]\\
& = \Esp'\l[1 + \sum_{m=1}^{\infty}\l(\prod_{r=0}^{m-1}\lim_{n \to \infty}\frac{\alpha_{j+1}^{(n)}+z'^{(n)}+r}{\alpha_{j+1}^{(n)}+\beta_{j+1}^{(n)}+n+r}\r)\frac{t^m}{m!}\r]\\
& = \Esp'\l[1 + \sum_{m=1}^{\infty}\frac{(\Upsilon_j(u)t)^m}{m!}\r]\\
& = e^{t\Upsilon_j(u)},
\end{align*}
which proves $v'^{(n)} \dto \Upsilon_j(u)$, as $n \to \infty$, or equivalently $\bpsi^{(n)}_j(u_{n},\cdot) \wto \delta_{\Upsilon_j(u)}$. 
\end{proof}

\begin{proof}[Proof of Theorem  
\ref{cor:BMSBp_conv}]
(i) is evident, as for (ii), it follows from Theorem 
\ref{theo:MSB_conv} 
and Lemma \ref{lem:BB_trans_conv}.
\end{proof} 

\subsection{Proof of Proposition 
\ref{prop:BMSB_post}}\label{app:prop:BMSB_post}

\begin{proof}
As mentioned in the main document, we can introduce a sequence of r.v.~$\bm{z} = (z_j)_{j=1}^{\infty}$ such that the joint density function of $\bm{v}_{[m]} = (v_j)_{j=1}^m$ and $\bm{z}_{[m]} = (z_j)_{j=1}^{m}$ is 
\begin{equation*}
\bm{p}(\bm{v}_{[m]},\bm{z}_{[m]}) = \prod_{j=1}^{m}\Be(v_{j}\mid \alpha'_j,\beta'_j)\Bin(z_j\mid N, \Upsilon_j(v_j)),
\end{equation*}
or every $m \geq 1$, where $\alpha'_1 = \alpha_1$, $\beta'_1 = \beta_1$, $\alpha'_j = \alpha_j+z_{j-1}$ and $\beta'_j = \beta_j +N -z_{j-1}$, for $j \geq 2$. Hence, from equation 
\eqref{eq:all_var_cond} 
we get
\[
\bm{p}(\bm{v}_{[m]},\bm{z}_{[m]}\mid d_1,\ldots,d_n) \propto  \prod_{j=1}^{m}\Be(v_{j}\mid \alpha'_j,\beta'_j)\Bin(z_j\mid N, \Upsilon_j(v_j))v_j^{a_j}(1-v_j)^{b_j}.
\]
This way it is easy to see that $\bm{p}(\bm{z}_{[m]}\mid \bm{v}_{[m]},d_1,\ldots,d_n) = \prod_{j=1}^{m}\bm{p}(z_j\mid \bm{v}_{[m]},d_1,\ldots,d_n)$, with
\begin{align*}
\bm{p}(z_j\mid \bm{v}_{[m]}, d_1,\ldots,d_n) & \propto  \binom{N}{z_j}[\Upsilon_j(v_j)]^{z_j}[(1-\Upsilon_j(v_j))]^{N-z_j}\frac{v_{j+1}^{z_j}(1-v_{j+1})^{N-z_j}}{\Gamma(\alpha_{j+1}+z_j)\Gamma(\beta_{j+1}+N-z_j)} \\
& \propto \binom{N}{z_j}\frac{[\Upsilon_j(v_j)v_{j+1}]^{z_j}[(1-\Upsilon_j(v_j))(1-v_{j+1})]^{N-z_j}}{(\alpha_{j+1})_{z_j}(\beta_{j+1})_{N-z_j}},
\end{align*}
for $j < m$, and $\bm{p}(z_m\mid \bm{v}_{[m]},d_1,\ldots,d_n) = \Bin(z_m\mid N,\Upsilon_m(v_m))$. Similarly,
\[
\bm{p}(\bm{v}_{[m]}\mid \bm{z}_{[m]},d_1,\ldots,d_n) \propto \prod_{j=1}^{m} v_j^{\alpha'_j+a_j-1}(1-v_j)^{\beta'_j+b_j-1}[\Upsilon_j(v_j)]^{z_j}[(1-\Upsilon_j(v_j))]^{N-z_j}.
\]
\end{proof}

\medskip

\section{Proofs of the results in Section 
\ref{sec:LMSB}}\label{sec:app:LMSB_sup}

\smallskip

\subsection{Marginal distributions of the length variables}\label{app:prop:SS_trans_marg}

\begin{prop}\label{prop:SS_trans_marg}
Let $\bm{w}$ be a LMSBw with parameters $(\rho,\bm{\pi})$.   Then, $v_j \sim \pi_j$, for $j \geq 1$. 
\end{prop}

\begin{proof}
Since $v_{j+1}\mid v_j \sim \psi_j(v_j,\cdot)$, where $\psi_j$ is a lazy transition as in 
\eqref{eq:sSS_trans} 
and recalling that if $v \sim \pi_j$ then $\Upsilon_j(v) \sim \pi_{j+1}$ (as discussed in Section 
\ref{sec:examples}
and Proposition \ref{prop:BB_trans_marg}) by the tower property we get
\begin{align*}
\Prob[v_{j+1} \in B] & = \int_{[0,1]} \rho \, \delta_{\Upsilon_j(v)}(B) + (1-\rho)\pi_{j+1}(B) \,\pi_{j}(\mrm{d}v)\\
& = \rho\int_{[0,1]}  \, \delta_{\Upsilon_j(v)}(B) \pi_{j}(\mrm{d}v)+ (1-\rho)\pi_{j+1}(B) \\
& = \rho \, \pi_{j+1}(B) + (1-\rho)\pi_{j+1}(B) = \pi_{j+1}(B),
\end{align*}
i.e. $v_{j+1} \sim \pi_{j+1}$ This induction step together with the fact $v_1 \sim \pi_1$ proves that $v_j \sim \pi_j$ for every $j \geq 1$.
\end{proof}

\subsection{Proof of Theorem \ref{theo:LMSBp}}\label{app:theo:LMSBp}

\begin{proof}
The requirement that the distribution of $F_j$ of $\pi_j$ is continuous and strictly increasing assures $\pi_j$ is diffuse and supported on $[0,1]$. Thus (ii) and (iii) follow directly from Propositions 
\ref{prop:MSB_supp_points} 
and  
\ref{prop:MSB_full_supp}. 
Now assume (i.b) holds, then from equation 
\eqref{eq:sum_1_Upsilon}
 we know that for $\rho = 1$, $\bm{P}$ is proper. Also note that since $\Upsilon^{[j]}(v_1) \sim \pi_j$ with $v_1 \sim \pi_1$ and $\Upsilon^{[j]}$ as in 
 \eqref{eq:sum_1_Upsilon}, 
 by monotone convergence theorem we get
\[
\infty = \Esp\l[\sum_{j=1}^{\infty}\Upsilon^{[j]}(v)\r] = \sum_{j=1}^{\infty} \int x \,\pi_j(\mrm{d}x).
\]
This proves that the condition in (i.b) implies that in (i.a). Hence to finish the proof it is enough to prove (i.a). To this aim let $\bm{v}$ be the underlying Markov process of length variables with initial distribution $\pi = \pi_1$ and lazy transitions $\bm{\psi}$. Then, recalling that $v_j \sim \pi_j$, we can easily compute
\begin{align*}
\Esp\l[\prod_{j=1}^{m+1}(1-v_j)\r] & = \Esp\l[\Esp[(1-v_{m+1})\mid v_1,\ldots,v_{m}]\prod_{j=1}^{m}(1-v_j)\r]\\
& = \Esp\l[\l\{1-\l(\rho\Upsilon_m(v_m)+(1-\rho)\Esp[v_{m+1}]\r)\r\}\prod_{j=1}^{m}(1-v_j)\r]\\
& \leq \l\{1-(1-\rho)\Esp[v_{m+1}]\r\} \Esp\l[\prod_{j=1}^{m}(1-v_j)\r],
\end{align*}
and by induction on $m$ we obtain
\begin{equation}\label{eq:theo:LMSBp_leq_v}
0 \leq \Esp\l[\prod_{j=1}^{m}(1-v_j)\r] \leq  \prod_{j=1}^{m} \l\{1-(1-\rho)\Esp[v_{j}]\r\}.
\end{equation}
By hypothesis we know that if $\rho < 1$, $(1-\rho)\sum_{j=1}^{\infty}\Esp[v_j] = \infty$, hence the right side of \eqref{eq:theo:LMSBp_leq_v} goes to zero as $m \to \infty$, and we get $\Esp[\prod_{j=1}^{m}(1-v_j)] \to 0$.  Finally since these variables are positive and bounded by one, this yields  $1-\sum_{j=1}^m w_j = \prod_{j=1}^{m}(1-v_j) \to 0$ a.s.~and the result follows.
\end{proof}

\subsection{Proof of Theorem 
\ref{cor:LMSB_conv}}\label{app:lem:SS_trans_conv}

\begin{lem}\label{lem:SS_trans_conv}
Let $\bm{\psi}^{(n)} = \big(\psi^{(n)}_j\big)_{j=1}^{\infty}$ and  $\bm{\psi} = (\psi_j)_{j=1}^{\infty}$ be lazy transitions with parameters $\big(\rho^{(n)}, \bm{\pi}^{(n)}\big)$ and $(\rho,\bm{\pi})$, respectively. Say that as $n \to \infty$, $\rho^{(n)} \to \rho$ in $[0,1]$ and $\pi_j^{(n)} \wto\pi_j$. Then, for each $u_n \to u \in [0,1]$ and $j \geq 1$, $\psi_j^{(n)}(u_n,\cdot) \wto \psi_j(u,\cdot)$. In particular if $\rho = 0$, $\psi_j(u,\cdot) = \pi_{j+1}$ and if $\rho = 1$, $\psi_j(u,\cdot) = \delta_{\Upsilon_j(u)}$ for each $j \geq 1$, where $\Upsilon_j$ is as in 
\eqref{eq:Upsilon}.
\end{lem}

\begin{proof}
Let $\Upsilon^{(n)}_j = F_{n,j+1}^{-1}\circ F_{n,j}$ and $\Upsilon_j =  F_{j+1}^{-1}\circ F_{j}$ where $F_{n,j}$ is the distribution function of $\pi^{(n)}_j$ and $F_j$ is the distribution function of $\pi_j$. Fix $u_n \to u$ in $[0,1]$, and note that we can write
\[
\psi^{(n)}_j(u_n,\cdot) = \rho^{(n)}\, \delta_{x_n} + \big(1-\rho^{(n)}\big)\pi^{(n)}_{j+1} \quad  \text{ and } \quad \psi_j(u,\cdot) = \rho\, \delta_{x} + (1-\rho)\pi_{j+1},
\]
with $x_n = \Upsilon^{(n)}_j(u_n)$ and $x = \Upsilon_j(u)$. Now, let $f:[0,1] \to \R$ be a continuous and bounded function, by Lemma \ref{lem:Upsilon_UC} we get $f(x_n) \to f(x)$. Moreover, by hypothesis $\pi^{(n)}_{j+1} \wto \pi_{j+1}$, i.e.
\[
\pi^{(n)}_{j+1}(f) = \int f(v) \pi_{j+1}^{(n)}(\mrm{d}v) \to \int f(v) \pi_{j+1}(\mrm{d}v) = \pi_{j+1}(f).
\]
Putting everything together
\begin{align*}
\psi_j^{(n)}(u_n,f) & = \int f(v) \psi^{(n)}_j(u_n,\mrm{d}v) = \rho^{(n)}\, f(x_n) + \big(1-\rho^{(n)}\big)\pi^{(n)}_{j+1}(f)\\ 
& \longrightarrow \rho\, f(x) + (1-\rho)\pi_{j+1}(f) =  \int f(v) \psi_j(u,\mrm{d}v) = \psi_j(u,f).
\end{align*}
Hence $\psi_j^{(n)}(u_n,\cdot) \wto   \psi_j(u,\cdot)$. The rest of the proof follows by simple substitution.
\end{proof}

\begin{proof}[Proof of Theorem \ref{cor:LMSB_conv}]
It follows easily from Lemma \ref{lem:SS_trans_conv} and Theorem \ref{theo:MSB_conv}
\end{proof}

\subsection{Proof of Proposition 
\ref{prop:LMSB_post}}\label{app:prop:LMSB_post}

\begin{proof}
From 
\eqref{eq:MSB_post_v1} 
we know
\[
\bm{p}(v_1\mid \bm{v}_{-1},d_1,\ldots,d_n) \propto v_1^{a_1}(1-v_1)^{b_1}\pi_1(v_1)\psi_1(v_1,v_2)
\]
where $\pi_1 = \Be(\alpha_1,\beta_1)$ and $\psi_1(v_1,v_2) = \rho\,\delta_{\Upsilon_1(v_1)}(v_2) + (1-\rho) \Be(v_2\mid \alpha_2,\beta_2)$. Notice that here $\psi_1(v_1,v_2)$ denotes a \emph{mixed density} with a discrete part of size $\rho$ and a diffuse part of size $(1-\rho)$. Taking this notational device into account and noting that $\delta_{\Upsilon_1(v_1)}(v_2) = \Ind\{v_2 = \Upsilon_1(v_1)\} = \Ind\{v_1 = \Upsilon_1^{-1}(v_2)\}$ we get
\begin{align*}
\bm{p}(v_1\mid \bm{v}_{-1},d_1,\ldots,d_n) & \propto \Be(v_1\mid \alpha_1+a_1,\beta_1+b_1)\l[\rho\,\Ind\{v_1 = \Upsilon_1^{-1}(v_2)\}+ (1-\rho)\Be(v_2\mid \alpha_2,\beta_2)\r]\\
& \propto q_1\, \delta_{\Upsilon_1^{-1}(v_2)}(v_1) + q_2\, \Be(v_1\mid \alpha_1+a_1,\beta_1+b_1)
\end{align*}
with
\[
q_1 = \rho\,\Be(\Upsilon_1^{-1}(v_2)\mid \alpha_1+a_1,\beta_1+b_1) \quad \text{ and } \quad
q_2 = (1-\rho)\Be(v_2\mid \alpha_2,\beta_2).
\]
Thus
\[
\bm{p}(v_1\mid \bm{v}_{-1},d_1,\ldots,d_n) = \rho_1\, \delta_{\Upsilon_1^{-1}(v_2)}(v_1) + (1-\rho_1)\, \Be(v_1\mid \alpha_1+a_1,\beta_1+b_1),
\]
where $\rho_1 = q_1/(q_1+q_2)$, meaning that, $v_1 = \Upsilon_1^{-1}(v_2)$ with probability $\rho_1$, and with probability $(1-\rho_1)$ $v_1$ is sampled from a $\Be(\alpha_1+a_1,\beta_1+b_1)$ distribution.

Now, for $j > 1$, we get from 
\eqref{eq:MSB_post_vj}
\begin{equation}\label{eq:full_cond_LMSB}
\bm{p}(v_j\mid \bm{v}_{-j},d_1,\ldots,d_n) \propto v_j^{a_j}(1-v_j)^{b_j}\psi_{j-1}(v_{j-1},v_j)\psi_j(v_j,v_{j+1}),
\end{equation}
where $\psi_i(v_i,v_{i+1}) = \rho\, \delta_{\Upsilon_i(v_i)}(v_{i+1}) + (1-\rho)\Be(v_{i+1}\mid \alpha_i,\beta_i)$.  Noting that we may write
\[
\psi_i(v_i,v_{i+1})  = \rho \,\Ind\{\Upsilon_i(v_i) = v_{i+1}\} + (1-\rho) \Ind\{\Upsilon_i(v_i) \neq v_{i+1}\}\Be(v_{i+1}\mid \alpha_i,\beta_i)
\]
a.s., because Beta distributions are non-atomic, we get
\begin{equation}\label{eq:product_SS_trans}
\begin{aligned}
\psi_{j-1}&(v_{j-1},v_j)\psi_j(v_j,v_{j+1})\\
& = \Ind\{\Upsilon_{j-1}(v_{j-1}) = v_{j} = \Upsilon_j^{-1}(v_{j+1})\}\rho^{2}\\
& \quad + \Ind\{\Upsilon_{j-1}(v_{j-1}) = v_{j}\}\Ind\{v_j\neq \Upsilon_j^{-1}(v_{j+1})\}\rho(1-\rho)\Be(v_{j+1}\mid \alpha_{j+1},\beta_{j+1})\\
& \quad + \Ind\{\Upsilon_{j-1}(v_{j-1}) \neq v_{j}\}\Ind\{v_j= \Upsilon_j^{-1}(v_{j+1})\}\rho(1-\rho)\Be(\Upsilon_{j}^{-1}(v_{j+1})\mid \alpha_j,\beta_j)\\
& \quad + \Ind\{\Upsilon_{j-1}(v_{j-1}) \neq v_{j}\}\Ind\{v_j\neq \Upsilon_j^{-1}(v_{j+1})\}(1-\rho)^2 \Be(v_{j+1}\mid \alpha_{j+1},\beta_{j+1})\Be(v_j\mid \alpha_j,\beta_j)
\end{aligned}
\end{equation}
Here we can recognize two scenarios that must be treated separately, the first one arises when $\Upsilon_{j-1}(v_{j-1}) = \Upsilon_{j}^{-1}(v_{j+1})$. In this case we know that $v_j = \Upsilon_{j-1}(v_{j-1})$ and $v_{j+1} = \Upsilon_j(v_j)$ a.s., because the probability that a Beta distributed r.v.~equals a particular value is zero. Hence, if  $\Upsilon_{j-1}(v_{j-1}) = \Upsilon_{j}^{-1}(v_{j+1})$, 
\[
\bm{p}(v_j\mid \bm{v}_{-j},d_1,\ldots,d_n) = \delta_{\Upsilon_{j-1}(v_{j-1})} = \delta_{\Upsilon_{j}^{-1}(v_{j+1})}.
\]
For the other case when $\Upsilon_{j-1}(v_{j-1}) \neq \Upsilon_{j}^{-1}(v_{j+1})$, we find that 
\begin{align*}
\Ind\{\Upsilon_{j-1}(v_{j-1}) = v_{j}\}\Ind\{v_j\neq \Upsilon_j^{-1}(v_{j+1})\} &= \Ind\{\Upsilon_{j-1}(v_{j-1}) = v_{j}\}\\
\Ind\{\Upsilon_{j-1}(v_{j-1}) \neq v_{j}\}\Ind\{v_j = \Upsilon_j^{-1}(v_{j+1})\} &= \Ind\{v_j = \Upsilon_j^{-1}(v_{j+1})\}
\end{align*}
and $\Ind\{\Upsilon_{j-1}(v_{j-1}) = v_{j} = \Upsilon_j^{-1}(v_{j+1})\} = 0$. Hence, from \eqref{eq:full_cond_LMSB} and \eqref{eq:product_SS_trans} we get
\begin{align*}
\bm{p}&(v_j\mid \bm{v}_{-j},d_1,\ldots,d_n)\\
& \propto \Ind\{\Upsilon_{j-1}(v_{j-1}) = v_{j}\}\,\rho\,\Be(v_{j+1}\mid \alpha_{j+1},\beta_{j+1})[\Upsilon_{j-1}(v_{j-1})]^{a_j}[1-\Upsilon_{j-1}(v_{j-1})]^{b_j}\\
& \quad + \Ind\{v_j= \Upsilon_j^{-1}(v_{j+1})\}\,\rho\,\Be(\Upsilon_{j}^{-1}(v_{j+1})\mid \alpha_j,\beta_j)[\Upsilon_{j}^{-1}(v_{j+1})]^{a_j}[1-\Upsilon_{j}^{-1}(v_{j+1})]^{b_j}\\
&\quad + \Ind\{\Upsilon_{j-1}(v_{j-1}) \neq v_{j}\neq \Upsilon_j^{-1}(v_{j+1})\}(1-\rho) \Be(v_{j+1}\mid \alpha_{j+1},\beta_{j+1})\Be(v_j\mid \alpha_j,\beta_j)v_j^{a_j}(1-v_j)^{b_j}.
\end{align*}
and we can write
\[
\bm{p}(v_j\mid \bm{v}_{-j},d_1,\ldots,d_n) \propto q_{j,1}\,\delta_{\Upsilon_{j-1}(v_{j-1})}(v_j) + q_{j,2}\,\delta_{\Upsilon_{j}^{-1}(v_{j+1})}(v_j) + q_{j,3}\,\Be(v_j\mid \alpha'_j,\beta'_j)
\]
where $\alpha'_j = \alpha_j+a_j$, $\beta'_j = \beta_j+b_j$, and
\begin{align*}
q_{j,1} & = \rho\,\Be(v_{j+1}\mid \alpha_{j+1},\beta_{j+1})[\Upsilon_{j-1}(v_{j-1})]^{a_j}[1-\Upsilon_{j-1}(v_{j-1})]^{b_j}\\
q_{j,2} & = \rho\,\Be(\Upsilon_{j}^{-1}(v_{j+1})\mid \alpha_j,\beta_j)[\Upsilon_{j}^{-1}(v_{j+1})]^{a_j}[1-\Upsilon_{j}^{-1}(v_{j+1})]^{b_j}\\
q_{j,3} & = (1-\rho)\Be(v_{j+1}\mid \alpha_{j+1},\beta_{j+1}) \frac{(\alpha_j)_{a_j}(\beta_j)_{b_j}}{(\alpha_j+\beta_j)_{a_j+b_j}}. 
\end{align*}
Thus we obtain
\[
\bm{p}(v_j\mid \bm{v}_{-j},d_1,\ldots,d_n) = \rho_{j,1}\,\delta_{\Upsilon_{j-1}(v_{j-1})}(v_j) + \rho_{j,2}\,\delta_{\Upsilon_{j}^{-1}(v_{j+1})}(v_j) + \rho_{j,3}\,\Be(v_j\mid \alpha'_j,\beta'_j),
\]
where $\rho_{j,i} = q_{j,i}/(q_{j,1}+q_{j,2}+q_{j,3})$ for each $i \in \{1,2,3\}$. Summarizing, we conclude that if $\Upsilon_{j-1}(v_{j-1}) = \Upsilon_{j}^{-1}(v_{j+1})$ we simply set $v_j = \Upsilon_{j-1}(v_{j-1})$. Otherwise, $v_j = \Upsilon_{j-1}(v_{j-1})$ with probability $\rho_{j,1}$, $v_j = \Upsilon_{j}^{-1}(v_{j+1})$ with probability $\rho_{j,2}$, and with probability $\rho_{j,3}$ $v_j$ is sampled from a $\Be(\alpha'_j,\beta'_j)$ distribution.
\end{proof}

\newpage

\bibliographystyle{chicago}
\bibliography{references}  

\end{document}